\documentclass[a4paper]{article}

\usepackage{float}
\usepackage{dsfont}
\usepackage{amssymb}
\usepackage{latexsym}
\usepackage{amsmath}
\usepackage{color}
\usepackage{comment}
\usepackage{amsthm}
\usepackage{graphicx} 
\usepackage{layout} 
\usepackage{ulem}
\usepackage{enumerate}
\usepackage{bm}
\usepackage{bbm} 
\usepackage[bbgreekl]{mathbbol} 
\usepackage{authblk}
\usepackage[pdfborder={1 1 1}]{hyperref}  
\usepackage{multirow}
\usepackage{hyperref}

\newif\ifrs
\rstrue
\ifrs \usepackage{mathrsfs} \fi  
\newif\ifcol
\coltrue 

\newtheorem{theorem*}{Theorem}[section]

\newtheorem{note*}[theorem*]{Note}
\newtheorem{lemma*}[theorem*]{Lemma}
\newtheorem{definition*}[theorem*]{Definition}
\newtheorem{proposition*}[theorem*]{Proposition}
\newtheorem{corollary*}[theorem*]{Corollary}
\newtheorem{remark*}[theorem*]{Remark}
\newtheorem{example*}[theorem*]{Example}

\newtheorem{assk}{ASS(k)}

\DeclareMathOperator*{\argmax}{argmax}

\numberwithin{equation}{section}
\newif\ifcol
\coltrue
\ifcol
\newcommand{\colorr}{\color[rgb]{0.8,0,0}}

\newcommand{\colorn}{\color[rgb]{1,1,1}}

\else

\newcommand{\colorr}{\color{black}}

\newcommand{\colorn}{\color{black}}

\fi
\excludecomment{en-text}
\includecomment{jp-text}
\includecomment{comment}
\newcommand{\ep}{\epsilon}

\newcommand{\iku}{\rightarrow}

\newcommand{\im}{\item}
\newcommand{\reels}{\mathbb{R}}
\newcommand{\naturels}{\mathbb{N}}
\newcommand{\relatifs}{\mathbb{Z}}
\newcommand{\complex}{\mathbb{C}}

\newcommand{\idx}{\mathbb{n}}
\newcommand{\dims}{\mathbb{a}}

\newcommand{\mle}{\hat{\theta}_T}

\begin{document}

\title{Hawkes process and Edgeworth expansion with application to maximum likelihood estimator}
\author{Masatoshi Goda\footnote{Graduate School of Mathematical Sciences, University of Tokyo: 3-8-1 Komaba, Meguro-ku, Tokyo 153-8914, Japan. e-mail: goda@ms.u-tokyo.ac.jp}}
\affil{Graduate School of Mathematical Sciences, University of Tokyo\\
Japan Science and Technology, CREST, Japan
}
\maketitle

\begin{abstract}
We provide a rigorous mathematical foundation of the theory for the higher-order asymptotic behavior of the one-dimensional Hawkes process with an exponential kernel. As an important application, we give the second-order asymptotic distribution for the maximum likelihood estimator of the exponential Hawkes process.
\end{abstract}

\section{Introduction}
The Hawkes process was introduced by \cite{Hawkes}. It has a self-exciting property and has been used to model earthquakes and their aftershocks, events in social media, infectious diseases and so on.  Furthermore, in the field of finance, the multivariate Hawkes process has also been used for modeling the whole limit order book; for example, see \cite{LOB}. 

Regarding statistical inference for the Hawkes process, the quasi maximum likelihood estimator (QMLE) and the quasi Bayesian estimator (QBE) are practical. \cite{ClinetYoshida} established the consistency, the asymptotic normality and the convergence of moments of these estimators for the multivariate Hawkes process with exponential kernels. Roughly, the asymptotic normality of an estimator $\hat{\theta}_T$ is characterized as follows.
\begin{eqnarray*}
\left|E\left[f\left(\sqrt{T}(\hat{\theta}_T - \theta_0)\right)\right] - \int f(x) \phi(x; g^{-1})dx\right| = o(1)
\end{eqnarray*}
for some appropriate functions $f$, where $g$ is the Fisher information matrix, $\phi(x; g^{-1})$ is the probability density function of the normal distribution $N(0, g^{-1})$ and $\theta_0$ is the true parameter.

As a real problem, there are often situations where we can not obtain data with sufficient observation time. In such cases, it is not appropriate to approximate the error distribution of an estimator by the normal distribution. Then, the derivation of confidence intervals and hypothesis testing cannot be performed accurately. If we establish a theory of Edgeworth expansion for the distribution of an estimator $\hat{\theta}_T$, we get an improved error evaluation. Edgeworth expansion is obtained by formally expanding the characteristic function and applying the inverse Fourier transform. Roughly, we may have the following evaluation in the case of the second-order expansion.
\begin{eqnarray}
\label{aym ineq}
\left|E\left[f\left(\sqrt{T}(\hat{\theta}_T - \theta_0)\right)\right] - \int f(x) q_{T,3}(x)dx\right| = o(T^{-1/2}),
\end{eqnarray}
for some appropriate functions $f$, where $q_{T,3}(x)dx$ is some signed measure. \\*

In this paper, we deal with the one-dimensional Hawkes process with an exponential kernel and establish the Edgeworth expansion for the distribution of its maximum likelihood estimator (MLE). The outline of the concrete proofs is as follows. 

In the case of independent identical distribution, the validity of the Edgeworth expansion for an MLE is reduced to the expansion of a log-likelihood process. Same as the i.i.d. case, \cite{sakyos} gave the asymptotic expansion for an M-estimator of a functional of an $\ep$-Markov process with a mixing property. In the case of MLE, the essence of the theory of the asymptotic expansion lies in the analysis of a log-likelihood process. With this background, we deal with the asymptotic expansion for the class of functionals of the Hawkes process containing the derivatives of the log-likelihood process. 

To prove the validity of (\ref{aym ineq}), we prepare the theory of Edgeworth expansion for the distribution of a functional of a geometric mixing process. For discrete-time processes, the scheme of this theory was established by  \cite{GotzeHipp}. It was extended to a continuous-time case in \cite{KusuokaYoshida}. Moreover, \cite{Yoshidaparmix} dealt with a more general framework. We will reduce their theory to a simple framework without the Cram\'er-condition. Furthermore, we have developed a framework that is not confined to the non-degeneracy of variance by appropriately modifying the variance of the random variables.

Second, we will apply the theory of the Edgeworth expansion to the derivatives of the log-likelihood process of the exponential Hawkes process. In this application, we introduce the Hawkes core process.  In the proof of this theory, it is essential to confirm the conditions regarding the mixing property and the finiteness of moments of the Hawkes core process. The mixing property of the Hawkes core process follows from its Markovian property and geometric ergodicity. 
It is known that the exponential Hawkes intensity process has the Markovian property, see \cite{Oakes1975}. However, we introduce a new proof including a method applicable to the Hawkes core process. 
We investigate these properties by using the idea of the extended generator. 

Finally, we give the second-order asymptotic distribution for the MLE of the exponential Hawkes process. In this regard, we confirm some conditions for the log-likelihood process.

Numerical calculations of the asymptotic distribution using the Monte Carlo method are also presented. Furthermore, the results of the simulations with R confirm that the asymptotic distribution we introduced is a better approximation than the approximation by the normal distribution.\\*

Section 2 presents a theory of the Edgeworth expansion for the distribution of a functional of a geometric mixing process. It also describes how to apply the Edgeworth expansion to an MLE.
In Section 3, we see the properties of the one-dimensional Hawkes process with an exponential kernel. It is difficult to directly express the derivative of the log-likelihood process as a functional of the Hawkes intensity. Therefore, we introduce the Hawkes core process and also investigate its properties.
In Section 4, we will apply the theory of the Edgeworth expansion to a functional of the Hawkes core process. In particular, we give the concrete form of the second-order asymptotic distribution for the MLE of the exponential Hawkes process.
Finally, Section 5 shows the simulation results about the second-order asymptotic distribution for the MLE of the exponential Hawkes process.
The details of the proofs in each section are summarized in Appendix.
\

\section{Asymptotic expansion}

\subsection{Asymptotic expansion under geometric mixing condition}

In this section, we introduce the theory of Edgeworth expansion for the distribution of a functional of a geometric mixing process.  The following framework is given by Theorem 2.10 in \cite{GotzeHipp} for discrete-time processes.  It is extended to a continuous-time case under the conditional Cram\'er-condition in \cite{Yoshidaparmix}. In this paper, we rewrite this theory without using the Cram\'er-condition, referring to \cite{GotzeHipp1978}. 

Let $(\Omega, \mathscr{F}, P)$ be a probability space. Assume that we are given $\sigma$-fields $\{\mathscr{B}_I\}$ indexed by intervals $I \subset \reels_+$. We consider a process $Z = (Z_t)_{t \in \reels_+}:\Omega \times \reels_+ \iku \reels^d$ whose increment is adapted to $\mathscr{B}_I$, namely $Z_I = Z_t - Z_s \in \mathscr{F}\mathscr{B}_I$\footnote{$\mathscr{F}\mathscr{B}$ denote the set of $\mathscr{B}$-measurable functions.} for every closed interval $I = [s, t] \subset \reels_+$ with $s<t$ and $Z_0 \in \mathscr{F}\mathscr{B}_{\{0\}}$. We will derive the asymptotic expansion for the distribution of the normalized process $S_T = Z_T/\sqrt{T}$.  For this purpose, we assume that the following two conditions hold.

\begin{description}
\im[[A1\!\!]](Geometric mixing property)\\
There exists a positive constant $a$ such that for any $s, t \in \reels_+$ with $s \le t$, and for any $f \in \mathscr{F}\mathscr{B}_{[0,s]}$ and $g \in \mathscr{F}\mathscr{B}_{[t,\infty)}$ with $\left\| f \right\|_{\infty} \le 1$ and $\left\| g \right\|_{\infty} \le 1$,
\begin{eqnarray*}
\left| E[fg] - E[f]E[g] \right| \le a^{-1}e^{-a(t-s)}.
\end{eqnarray*}		

\im[[A2\!\!]](Moment property)\\
$\sup_{t \in \reels_+, 0 \le h \le \Delta} \left\| Z_{[t, t+h]}\right\|_{L^p(P)} < \infty$ and $E[Z_{[t, t+\Delta]}]=0$ for any $\Delta > 0$ and $p > 0$. Moreover, $Z_0 \in \bigcap_{p>1}L^p(P)$ and $E[Z_0]=0$.\\
\end{description}

These conditions [A1] and [A2] are needed for the validity of the formal Edgeworth expansion. Before starting the statement of the asymptotic expansion, we prepare some notation under the condition [A2]. The $r$-th cumulant functions $\chi_{T,r}(u)$ of $S_T$ are defined by
\begin{eqnarray*}
\chi_{T,r}(u) = \left.\left(\frac{d}{d\epsilon}\right)^r\right|_{\epsilon = 0} \log E\left[e^{i\epsilon u' S_T}\right],
\end{eqnarray*}
where $u'$ represents the transpose of $u$. Then, since $\chi_{T,1}(u) = E[iu' S_T] = 0$, the characteristic function of $S_T$ is formally expressed as
\begin{eqnarray*}
E\left[e^{iu' S_T}\right] 
&= \left.\exp\left(\log E\left[e^{i\ep u' S_T}\right] \right)\right|_{\ep = 1}
&= \left.\exp\left(\sum^{\infty}_{r=2} r!^{-1}\ep^{r-2}\chi_{T,r}(u) \right)\right|_{\ep = 1}.
\end{eqnarray*}
Next, we define functions $\tilde{P}_{T,r}(u)$ by the formal Taylor expansion at $\ep= 0$:
\begin{eqnarray}
\exp\left( \sum^{\infty}_{r=2} r!^{-1}\ep^{r-2}\chi_{T,r}(u)\right) = \exp\left(\frac{1}{2}\chi_{T,2}(u)\right) + \sum^{\infty}_{r=1} \ep^rT^{-\frac{r}{2}}\tilde{P}_{T,r}(u),
\label{formal Taylor expansion}
\end{eqnarray}
where  
\begin{eqnarray}
\tilde{P}_{T,r}(u) = \exp\left(\frac{1}{2}\chi_{T,2}(u)\right)\sum_{l=1}^r \sum_{\substack{r_1, \dots, r_l \in \naturels; \\ r_1+ \cdots + r_l = r}}\frac{\chi_{T, r_1+2}(u)\cdots\chi_{T, r_l+2}(u)}{l!(r_1+2)!\cdots(r_l+2)!}.
\end{eqnarray}
Let $\hat{\Psi}_{T,p}(u)$ be the partial sum of the right-hand side of (\ref{formal Taylor expansion}) with $\ep = 1$:
\begin{eqnarray}
\hat{\Psi}_{T,p}(u) = \exp\left(\frac{1}{2}\chi_{T,2}(u)\right) + \sum^{p-2}_{r=1} T^{-\frac{r}{2}}\tilde{P}_{T,r}(u).
\end{eqnarray}
We want to define a signed measure $\Psi_{T,p}$ as the Fourier inversion of $\hat{\Psi}_{T,p}(u)$. However, when $\chi_{T,2}(u)$ is not negative definite, $\Psi_{T,p}$ does not have the density function. To overcome this problem, we set $\Sigma_{T,D} = Var[S_T] + T^{-D}I$ for a positive constant $D$, where $I$ is an identity matrix. Then, we define 
\begin{eqnarray}
\label{Edge}
\hat{\Psi}_{T,p,D}(u) = \exp\left(-\frac{1}{2}u'\Sigma_{T,D}u\right) + \sum^{p-2}_{r=1} T^{-\frac{r}{2}} \exp\left(-\frac{1}{2}u'\Sigma_{T,D}u\right)\sum_{l=1}^r \sum_{\substack{r_1, \dots, r_l \in \naturels; \\ r_1+ \cdots + r_l = r}}\frac{\chi_{T, r_1+2}(u)\cdots\chi_{T, r_l+2}(u)}{l!(r_1+2)!\cdots(r_l+2)!}.
\end{eqnarray}
The cumulant functions also have the following representation,
\begin{eqnarray}
\label{2.3.A}
\chi_{T,k}(u) = i^k\sum_{a_1, \dots, a_k = 1}^d u_{a_1}\cdots u_{a_k}\lambda^{a_1\cdots a_k;}_T,
\end{eqnarray}
where $u = (u_1, \dots, u_d)$ and $\lambda^{a_1\cdots a_k;}_T$ is the $(a_1, \dots, a_k)$-cumulant of $S_T$, i.e.,
\begin{eqnarray*}
\lambda^{a_1\cdots a_k;}_T = (-i)^k\left.\frac{\partial^k}{\partial u_{a_1}\cdots \partial u_{a_k}}\right|_{u_{a_1}=\cdots=u_{a_k} = 0} \log E\left[e^{iu' S_T}\right].
\end{eqnarray*}
Consider the order of convergence, we put
\begin{eqnarray}
\label{modi cumulant}
\kappa^{a_1\cdots a_k;}_T = T^{(m-2)/2}\lambda^{a_1\cdots a_k;}_T.
\end{eqnarray}
Let $h_{a_1 \dots a_k}(z; \Sigma)$ be the Hermite polynomials, i.e.
\begin{eqnarray}
\label{Hermite}
h_{a_1 \dots a_k}(z; \Sigma) = \frac{(-1)^k}{\phi(z;\Sigma)}\frac{\partial^k}{\partial z_{a_1}\cdots \partial z_{a_k}}\phi(z;\Sigma),
\end{eqnarray}
where $\phi(x; \Sigma)$ is the probability density function of the normal distribution $N(0,\Sigma)$. We define a signed measure $\Psi_{T,p,D}$ as the Fourier inversion of $\hat{\Psi}_{T,p,D}(u)$. Then, from (\ref{Edge}), (\ref{2.3.A}), (\ref{modi cumulant}) and (\ref{Hermite}), the density function $p_{T,p,D}(z)$ of $\Psi_{T, p,D}$ is written as
\begin{eqnarray}
\label{p}
p_{T,p,D}(z) = \phi(z;\Sigma_{T,D}) + \sum_{r=1}^{p-2} T^{-\frac{r}{2}} \left\{ \sum_{l=1}^r \sum_{\substack{r_1, \dots, r_l \in \naturels; \\ r_1+ \cdots + r_l = r}}\frac{\kappa^{A_{r_1 + 2};}_T \cdots \kappa^{A_{r_l + 2};}_T }{l!(r_1+2)!\cdots(r_l+2)!}h_{A_{r_1 + 2} \cdots A_{r_l + 2}}(z;\Sigma_{T,D})\phi(z;\Sigma_{T,D}) \right\}
\end{eqnarray}
where $A_k$ represents the index sequence $a_1\cdots a_k$ and if there are same index sequences, summing up them with respect to $a_1\cdots a_k$ in accordance with the Einstein summation convention. Furthermore if there are different multiple index sequences, we distinguish them, for example,
\begin{eqnarray*}
K^{A_i;}L^{A_j;}M_{A_iA_j} = \sum_{a_1, \dots, a_i, a'_1, \dots, a'_j= 1}^d K^{a_1 \dots a_i;}L^{a'_1 \dots a'_j;}M_{a_1 \dots a_i a'_1 \dots a'_j}.
\end{eqnarray*}

For a vector of nonnegative integers $\alpha = (\alpha_1, \dots, \alpha_d)$, $t \in \reels^d$ and $f \in C^{|\alpha|}(\reels^d)$, let
\begin{eqnarray*}
|\alpha| = \sum_{i=1}^d \alpha_i, \quad
t^{\alpha} = \prod_{i=1}^dt^{\alpha_i} \quad \text{and} \quad \partial^\alpha f = \frac{\partial^{|\alpha|}f}{\partial x_1^{\alpha_1} \cdots \partial x_d^{\alpha_d}}.
\end{eqnarray*}
For positive constants $\Gamma, L_1, L_2$, we denote by $\mathscr{E}(\Gamma, L_1, L_2)$ a set of functions $f \in C^{\Gamma}(\reels^d)$ with $\sup_{|\alpha| \le \Gamma}| \partial^{\alpha} f(x)| \le L_2(1+|x|)^{L_1}$ for every $x \in \reels^d$. The following theorem is the main statement in this subsection. A proof can be found in Appendix. 

\begin{theorem*}
\label{Main Thm}
Let $p \in \naturels$ with $p \ge 2$ and $L_1, L_2>0$. Suppose that the conditions \textnormal{[A1]} and \textnormal{[A2]} are satisfied. Then, there exist  $D>0$ and $\Gamma \in \naturels$ such that for any $f \in \mathscr{E}(\Gamma, L_1, L_2)$,
\begin{eqnarray*}
\left| E\left[ f\left(S_T\right)\right] - \int_{\reels^d}f(z) p_{T,p,D}(z)dz \right| =  o\left(T^{-(p-2)/2}\right).
\end{eqnarray*}
\end{theorem*}




\subsection{Asymptotic expansion for maximum likelihood estimator}

Applying Theorem \ref{Main Thm}, we consider getting the asymptotic expansion for a maximum likelihood estimator (MLE) up to the second order. We refer to \cite{sakyos} to construct the following framework.

Let $(\Omega, \mathscr{F}, P)$ be a probability space, $\Theta \subset \reels^p$ be an open bounded convex set and $N_t(\theta)$ be a $\reels^d$-valued stochastic process parameterized by $\theta \in \Theta$. We assume that the log-likelihood process of $N_t(\theta_0)$ is given by $l_T : \Theta \times \Omega \iku \reels$, where $\theta_0 \in \Theta$ is the true parameter. Then, the MLE $\mle$ is defined by 
\begin{eqnarray*}
\mle(\omega) = \argmax_{\theta \in \Theta} l_T(\theta, \omega) \quad \text{for $\omega \in \Omega$}.
\end{eqnarray*}
If there is no confusion, we write $l_T(\theta) = l_T(\theta, \omega)$. For $r \in \naturels$ and a sequence of indexes $\dims = (a_1 ,\dots, a_r) \in \{1, \dots, p\}^r$, we write
\begin{eqnarray}
\label{D}
\mathbb{D}^\dims = \frac{\partial^r}{\partial \theta^{a_1;} \cdots \partial \theta^{a_r;}},
\end{eqnarray}
where $\theta^{a;}$ is the $a$-th component of $\theta \in \Theta$. Moreover, $\partial_\theta$ denotes the vector differential operator (the gradient operator). Let $l_{a_1 \dots a_k}(\theta) = \mathbb{D}^{(a_1, \dots, a_k)}l_T(\theta)$ and $\nu_{a_1 \dots a_k}(\theta) = E\left[\frac{1}{T}l_{a_1 \dots a_k}(\theta)\right]$. We write $g_T = (g_{ab})_{a,b=1, \dots, p} = (-\nu_{ab}(\theta_0))_{a,b=1, \dots, p}$. We assume that 
\begin{description}
\im[[A3\!\!]] \quad $|g_T - g| \to 0$ as $T \to \infty$, where the norm $|\cdot|$ is the Frobenius norm and $g$ is a non-singular matrix.
\end{description}
\begin{description}
\im[[B0\!\!]] \quad $l_T(\theta)$ satisfies the following conditions.
\begin{enumerate}
\setlength{\itemindent}{0pt}  
\renewcommand{\labelenumi}{(\roman{enumi})} 
\item $l_T \in C^4(\Theta) \ a.s.$
\item the score function $\partial_\theta l_T(\theta)$ satisfies $E\left[\partial_\theta l_T(\theta_0)\right]=0$.
\item $Var\left[\frac{1}{\sqrt{T}}\partial_\theta l_T(\theta_0)\right] = g_T$.
\item $\mathbb{D}^c\nu_{ab}(\theta) = \nu_{abc}(\theta)$.
\end{enumerate}
\end{description}
Under the condition [A1], $g_T$ is non-singular for sufficiently large $T$.  We write $g_T^{-1} = (g^{ab;})_{a,b=1, \dots, p}$. The following conditions are assumed for some positive constants $q_1, q_2, q_3$ and $\gamma$.
\begin{description}
\im[[B1\!\!]]$_{q_1}$ \quad $\sup_{T>0}\left\|T^{-\frac{1}{2}}l_a(\theta_0)\right\|_{L^{q_1}(P)} < \infty
$\quad for $a\in \{1, \dots, p\}$.

\im[[B2\!\!]]$_{q_2, \gamma}$ \quad $\sup_{T>0, \theta \in \Theta}\left\|T^{\frac{\gamma}{2}}\left(T^{-1}l_{a_1 \cdots a_k}(\theta) - \nu_{a_1 \cdots a_k}(\theta) \right) \right\|_{L^{q_2}(P)} < \infty$\quad for  $k=2,3$, $a_1, \dots, a_k \in \{1,\dots p\}$.

\im[[B3\!\!]] \quad There exist an open set $\tilde{\Theta}$ including $\theta_0$ and a positive constant $T_0$ such that
\begin{eqnarray*}
\inf_{T>T_0, \theta_1, \theta_2 \in \tilde{\Theta},|x|=1} \left| x' \int^1_0  \nu_{a b}\left(\theta_1 + s(\theta_2 - \theta_1)\right)ds\right| > 0.
\end{eqnarray*}

\im[[B4\!\!]]$_{q_3}$ \quad $\sup_{T>0}\left\|\sup_{\theta \in \Theta}\left|T^{-1}l_{a_1 \cdots a_4}(\theta)\right|\right\|_{L^{q_3}(P)} < \infty$\quad, $a_1, \dots, a_4 \in \{1,\dots p\}$.
\end{description}
\

To get the asymptotic expansion of an MLE, we approximate the MLE with the sum of log-likelihood processes. Let 
\begin{eqnarray*}
Z_a = \frac{1}{\sqrt{T}}l_a(\theta_0)
\quad \text{and} \quad
Z_{ab} =\sqrt{T}\left(\frac{1}{T}l_{ab}(\theta_0) - \nu_{ab}(\theta_0)\right).
\end{eqnarray*}
With the Einstein summation convention, for an index sequence $A$, we write
\begin{eqnarray*}
\nu^{a;}_A = g^{ab;} \nu_{bA}(\theta_0) \quad \text{and} \quad Z^{a;}_A = g^{ab;} Z_{bA}.
\end{eqnarray*}
Under the condition [B3], let $\tilde{\Theta}$ be the one in [B3]. Set $\Omega_T = \{\omega \in \Omega ~|~ \exists! \mle(\omega) \in \tilde{\Theta}  \ s.t. \ \partial_{\theta}l_T(\mle(\omega), \omega) = 0\}$. Write $\bar{\theta}^{a_1 \dots a_k;} = T^{\frac{k}{2}}( \mle - \theta_0 )^{a_1;} \cdots ( \mle - \theta_0 )^{a_k;}$. On the set $\Omega_T$, from the Taylor expansion of $\frac{1}{T}l_a(\theta)$ at $\theta=\theta_0$, we immediately get the following two stochastic expansions
\begin{eqnarray*}
\sqrt{T}( \mle - \theta_0 )^{a;} 
&=& Z^{a; } + T^{-\frac{1}{2}}\left(Z^{a;}_{a_1}\bar{\theta}^{a_1;} + \frac{1}{2}\nu^{a;}_{a_1 a_2}\bar{\theta}^{a_1 a_2;}\right) + T^{-1}\bar{R}^{a;}_2\\
&=& Z^{a; } + T^{-\frac{1}{2}}\bar{R}^{a;}_1
\end{eqnarray*}
for any $a = 1,\dots,p$, where
\begin{eqnarray*}
\bar{R}^{a;}_2 = \frac{1}{2}Z^{a;}_{a_1 a_2}\bar{\theta}^{a_1 a_2;} +\frac{1}{2}\left\{\int^1_0(1-u)^2g^{ab;}\left( \frac{1}{T}l_{b a_1 a_2 a_3}\left(\theta_0 + u(\mle-\theta_0)\right)\right)du\right\} \bar{\theta}^{a_1 a_2 a_3;}.
\end{eqnarray*}
and
\begin{eqnarray*}
\bar{R}^{a;}_1 = Z^{a;}_{a_1}\bar{\theta}^{a_1;} + \frac{1}{2}\nu^{a;}_{a_1 a_2}\bar{\theta}^{a_1 a_2;} + T^{-\frac{1}{2}}\bar{R}^{a;}_2.
\end{eqnarray*}
From these two expressions, we get
\begin{eqnarray}
\label{asym mle}
\sqrt{T}( \mle - \theta_0 )^{a;} = Z^{a; } + T^{-\frac{1}{2}}\left(Z^{a;}_{a_1}Z^{a_1; } + \frac{1}{2}\nu^{a;}_{a_1 a_2}Z^{a_1; }Z^{a_2; }\right) + T^{-1}\check{R}^{a;}_2,
\end{eqnarray}
where
\begin{eqnarray*}
\check{R}^{a;}_2 = Z^{a;}_{a_1}\bar{R}^{a_1; }_1 + \bar{R}^{a;}_2 + T^{-\frac{1}{2}}\left(\frac{1}{2}\nu^{a;}_{a_1 a_2}\bar{R}^{a_1; }_1\bar{R}^{a_2; }_1 \right).
\end{eqnarray*}

\

We consider applying the transformation formula for the asymptotic expansion. Let
\begin{eqnarray*}
Z_T^{(1)} = T^{\frac{1}{2}}(Z_1, \dots, Z_p) \quad \text{and} \quad Z_T^{(2)} =  T^{\frac{1}{2}}(Z_{11}, \dots, Z_{1p}, Z_{21}, \dots, Z_{2p},\dots, Z_{p1}, \dots, Z_{pp}).
\end{eqnarray*}
Moreover, we put $Z_T = (Z_T^{(1)}, Z_T^{(2)})$ and $\Sigma_{T,D} = Var\left[Z_T/\sqrt{T}\right]+T^{-D}I$ for a positive constant $D$.
Let a $(p+p^2)\times(p+p^2)$-matrix $C_T$ be
\begin{eqnarray*}
C_T = \left(
    \begin{array}{ccc}
     g_T^{-1}  & O \\
      O  & G_T \\
    \end{array}
  \right),\\
\end{eqnarray*}
where
\begin{eqnarray*}
G_T = \left(
    \begin{array}{ccc}
     G_{11}&\cdots  & G_{1p} \\
     \vdots & \ddots & \vdots \\
     G_{p1}&\cdots  & G_{pp} \\
    \end{array}
  \right)
\quad \text{and} \quad
G_{ij} = \left(
    \begin{array}{cccc}
     g^{ij;} & 0 &\cdots & 0 \\
     0 & g^{ij;} &\cdots  & 0\\
     \vdots & \vdots & \ddots & \vdots \\
     0 &\cdots & 0 & g^{ij;}  \\
    \end{array}
  \right) \quad \text{: $p \times p$-matrix for $i,j=1,\dots,p$.}
\end{eqnarray*}
Then, we define
\begin{eqnarray}
\label{bZ_T}
\bar{Z}_T 
= \frac{1}{\sqrt{T}}C_TZ_T = (Z^{1;}, \dots, Z^{p;}, Z^{1;}_1, \dots, Z^{1;}_p, Z^{2;}_1, \dots, Z^{2;}_p,\dots, Z^{p;}_1, \dots, Z^{p;}_p),
\end{eqnarray}
and we write $\bar{Z}_T^{(1)} = (Z^{1;}, \dots, Z^{p;})$ and  $\bar{Z}_T^{(2)} = (Z^{1;}_1, \dots, Z^{p;}_p)$. We define the polynomial $Q_1(z)$ for a $p+p^2$-dimensional vector $z =  (z^{1;}, \dots, z^{p;}, z^{1;}_1, \dots, z^{p;}_p)$ such that $a$-th element of $Q_1(z)$ is
\begin{eqnarray}
\label{Q}
Q^{a;}_1(z) =  z^{a;}_{a_1}z^{a_1; } + \frac{1}{2}\nu^{a;}_{a_1 a_2}z^{a_1; }z^{a_2; }.
\end{eqnarray}
From (\ref{asym mle}), (\ref{bZ_T}) and (\ref{Q}), we get 
\begin{eqnarray}
\label{main form of MLE}
\sqrt{T}( \mle - \theta_0 ) = \bar{Z}_T^{(1)} + T^{-\frac{1}{2}}Q_1\left(\bar{Z}_T^{(1)}, \bar{Z}_T^{(2)}\right) + T^{-1}\check{R}_2.
\end{eqnarray}
From (\ref{main form of MLE}), we can give the asymptotic expansion for the MLE by using the transformation formula. However, it is complicated to calculate the concrete form of the density function $q_{T,3,D}$ described later. In order to simplify this calculation, the orthogonalization of $\bar{Z}_T$ is convenient. We put $\bar{\Sigma}_T^{(i, j)} = Cov\left[\bar{Z}_T^{(i)},\bar{Z}_T^{(j)}\right]$ for $i,j = 1,2$. Remark that $g_T^{-1}+T^{-D}(g_T^{-1})^2$ is non-singular for sufficiently large $T>0$. Thus, we can define $\tilde{g}_T^{-1} = (\tilde{g}^{ab;})_{a, b=1,\dots, p} = g_T^{-1}+T^{-D}(g_T^{-1})^2$ and $\tilde{g}_T = (\tilde{g}_{ab})_{a, b=1,\dots, p} = (\tilde{g}_T^{-1})^{-1}$. Let
\begin{eqnarray*}
M_{T, D}
= \left(
    \begin{array}{cc}
     I & O \\
     -\bar{\Sigma}_T^{(2, 1)}\tilde{g}_T & I \\
     \end{array}
  \right).
\end{eqnarray*}
We set $\tilde{Z}_T = M_{T, D} \bar{Z}_T = (\bar{Z}_T^{(1)}, \bar{Z}_T^{(2)} -\bar{\Sigma}_T^{(2, 1)}\tilde{g}_T\bar{Z}_T^{(1)})$, and we write $\tilde{Z}_T^{(1)} = \bar{Z}_T^{(1)}$ and $\tilde{Z}_T^{(2)} = \bar{Z}_T^{(2)} -\bar{\Sigma}_T^{(2, 1)}\tilde{g}_T\bar{Z}_T^{(1)}$. Then an elementary calculation yields
\begin{eqnarray*}
\tilde{\Sigma}_{T,D} 
=M_{T, D}C_T\Sigma_{T,D}C'_TM'_{T, D}
= \left(
    \begin{array}{cc}
     \tilde{\Sigma}_{T,D}^{(1,1)} & O \\
     O & \tilde{\Sigma}_{T,D}^{(2, 2)} \\
     \end{array}
  \right),
\end{eqnarray*}
where $\tilde{\Sigma}_{T,D}^{(1,1)} = \tilde{g}_T^{-1}$ and $\tilde{\Sigma}_{T,D}^{(2, 2)} = \bar{\Sigma}_T^{(2, 2)} + T^{-D}G_T^2 -\bar{\Sigma}_T^{(2, 1)}\tilde{g}_T\bar{\Sigma}_T^{(1, 2)}$. In terms of $\tilde{Z}_T$, $\sqrt{T}( \mle - \theta_0 )$ is rewritten as
\begin{eqnarray}
\label{ortho form of MLE}
\sqrt{T}( \mle - \theta_0 ) = \tilde{Z}_T^{(1)} + T^{-\frac{1}{2}}\tilde{Q}_1\left(\tilde{Z}_T^{(1)}, \tilde{Z}_T^{(2)}\right) + T^{-1}\check{R}_2,
\end{eqnarray}
where $\tilde{Q}_1$ is a polynomial for $\tilde{z}^{(1)} =  (\tilde{z}^{1;}, \dots, \tilde{z}^{p;})$ and $\tilde{z}^{(2)} =  (\tilde{z}^{1;}_1, \dots, \tilde{z}^{p;}_p)$
satisfying 
\begin{eqnarray*}
\tilde{Q}_1^{a;}\left(\tilde{z}^{(1)}, \tilde{z}^{(2)}\right) 
= Q_1^{a;}\left(\tilde{z}^{(1)}, \tilde{z}^{(2)} + \bar{\Sigma}_T^{(2, 1)}\tilde{g}_T\tilde{z}^{(1)}\right)
= \tilde{z}^{a;}_{a_1}\tilde{z}^{a_1;} + \tilde{\mu}^{a;}_{a_1a_2}\tilde{z}^{a_1;}\tilde{z}^{a_2;}
\end{eqnarray*}
for $\tilde{V}^{a;}_{bc} = Cov[Z^{a;}_{b}, Z^{a_1;}]\tilde{g}_{a_1c}$ and $\tilde{\mu}^{a;}_{a_1a_2} = (\tilde{V}^{a;}_{a_1a_2} + \tilde{V}^{a;}_{a_2a_1} + \nu^{a;}_{a_1a_2})/2$.

\

Focus on the main part of (\ref{ortho form of MLE}). We set  $\tilde{S}_T =  \tilde{Z}_T^{(1)} + T^{-\frac{1}{2}}\tilde{Q}_1\left(\tilde{Z}_T^{(1)}, \tilde{Z}_T^{(2)}\right)$. The following proposition holds, see Appendix for a proof.

\begin{proposition*}
\label{tildeS_T}
Let $L_1, L_2>0$. Suppose that the conditions \textnormal{[A1]-[A3]} and \textnormal{[B0]} hold. Then, there exist  $D>0$ and $\Gamma \in \naturels$ such that for any $f \in \mathscr{E}(\Gamma, L_1, L_2)$,
\begin{eqnarray*}
\left| E\left[ f\big(\tilde{S}_T\big)\right] - \int_{\reels^d}f(z^{(1)}) \tilde{q}_{T,3,D}(z^{(1)})dz^{(1)} \right| = o\left(T^{-1/2}\right),
\end{eqnarray*}
where
\begin{eqnarray*}
\tilde{q}_{T,3,D}(z^{(1)})
&=& \phi(z^{(1)}; \tilde{g}_T^{-1}) + \frac{1}{\sqrt{T}}\Bigg\{ \left(\frac{1}{6}\tilde{\kappa}^{a_1a_2a_3;}_T + \tilde{\mu}^{a_3;}_{b_1b_2}\tilde{g}^{b_1a_1;}\tilde{g}^{b_2a_2;}\right)h_{a_1a_2a_3}(z^{(1)}; \tilde{g}_T^{-1}) \\
&&+ \tilde{\mu}^{a_1;}_{b_1b_2}\tilde{g}^{b_1b_2;}h_{a_1}(z^{(1)}; \tilde{g}_T^{-1}) \Bigg\}\phi(z^{(1)}; \tilde{g}_T^{-1}).
\end{eqnarray*}
\end{proposition*}


\

We need to evaluate the remainder term $\check{R}_2$ in (\ref{ortho form of MLE}). The conditions [B0]-[B4] lead to the following proposition.  We will give a proof in Appendix.

\begin{proposition*}
\label{evaluateR_2}
Let $L>1$, $\gamma \in (0,1)$ and $q_1, q_2, q_3>0$ with
\begin{eqnarray}
\label{restriction}
q_1 > 3L,\quad
q_2 > \max\left(p, \frac{3q_1L}{q_1-3L}\right),\quad
q_3 > \frac{q_1L}{q_1-3L},\quad
\frac{2}{3} + \max\left(\frac{L}{q_2}, \frac{L}{3q_3}\right) < \gamma < 1 - \frac{L}{q_1}.
\end{eqnarray}
For these constants, we assume that the conditions \textnormal{[B0]-[B4]} hold. Then, there exist $C>0$ and $\ep' \in (1,0)$ such that 
\begin{eqnarray*}
P\left[\Omega_T \cap \left\{ T^{-1}|\check{R}_2^{a;}| \le CT^{-\frac{1 + \ep'}{2}}, \ a = 1, \dots, p \right\} \right] = 1 - o\left(T^{-\frac{L}{2}}\right).
\end{eqnarray*}
\end{proposition*}
Finally, we assume the following condition.
\begin{description}
\im[[C1\!\!]] \quad $\sup_{T>0}\big\| \sqrt{T}( \mle - \theta_0 ) \big\|_{L^k(P)} < \infty$ for any $k > 0$.
\end{description}

From Proposition \ref{tildeS_T} and Proposition \ref{evaluateR_2}, we get the asymptotic expansion for the distribution of the MLE. Remark that, since $g_T^{-1}$ is non-singular for large $T>0$ under the condition [A3], we can get rid of a constant $D$. Let $V^{a;}_{bc} = Cov[Z^{a;}_{b}, Z^{a_1;}]g_{a_1c}$ and $\mu^{a;}_{a_1a_2} = \big( V^{a;}_{a_1a_2} + V^{a;}_{a_2a_1} + \nu^{a;}_{a_1a_2} \big)/2$. The following theorem is the conclusion of this section.

\begin{theorem*}
\label{Main Thm2 New}
Let $L_1, L_2>0$. Suppose that the conditions \textnormal{[A1]-[A3]} and \textnormal{[C1]} are satisfied. Moreover, assume that there exist $L > 1$, $\gamma \in (0,1)$ and $q_1, q_2, q_3 > 0$ such that \textnormal{[B0]-[B4]} with (\ref{restriction}) hold. Then, there exists  $\Gamma \in \naturels$ such that for any $f \in \mathscr{E}(\Gamma, L_1, L_2)$,
\begin{eqnarray*}
\left| E\left[ f\big(\sqrt{T}( \mle - \theta_0 )\big)\right] - \int_{\reels^d}f(z^{(1)}) q_{T,3}(z^{(1)})dz^{(1)} \right| = o\left(T^{-\frac{1}{2}}\right),
\end{eqnarray*}
where 
\begin{eqnarray*}
q_{T,3}(z^{(1)})
&=& \phi(z^{(1)}; g_T^{-1}) + \frac{1}{\sqrt{T}}\Bigg\{ \left(\frac{1}{6}\tilde{\kappa}^{a_1a_2a_3;}_T + \mu^{a_3;}_{b_1b_2}g^{b_1a_1;}g^{b_2a_2;}\right)h_{a_1a_2a_3}(z^{(1)}; g_T^{-1}) \\
&&+ \mu^{a_1;}_{b_1b_2}g^{b_1b_2;}h_{a_1}(z^{(1)}; g_T^{-1}) \Bigg\}\phi(z^{(1)}; g_T^{-1}).
\end{eqnarray*}
\end{theorem*}

\section{Hawkes process with an exponential kernel}

In this section, we will define the Hawkes process with an exponential kernel whose intensity process starts from a point $x+\mu$. Furthermore, we will see its properties. In the second half, we define the Hawkes core process and investigate its properties. Let $(\Omega, \mathscr{F}, P)$ be a probability space and $\{\mathscr{F}_t\}_{t \in \reels_+}$ be a filtration that satisfies the usual conditions.

\subsection{Definition}
\label{hawkes def}
First, we define a point process and its intensity. A sequence of stopping times $\{\tau_n\}_{n \in \mathbb{N}}$ with respect to $\{\mathscr{F}_t\}_{t \in \reels_+}$ is called {\bf a point process}, if it satisfies the following properties:
\begin{description}
\im [(i)] $\tau_1 > 0$ a.s.
\im [(ii)] $\tau_n < \tau_{n+1}$ on $\{\tau_n < \infty\}$ a.s.\footnote{Such notation means that $P[\{\tau_n < \tau_{n+1}\}\cap\{\tau_n < \infty\}] = P[\tau_n < \infty]$.}  
\im [(iii)] $\tau_n = \tau_{n+1}$ on $\{\tau_n = \infty\}$ a.s.
\end{description}
Let $\tau_{\infty} = \lim_{n \to \infty}\tau_n$. Define a stochastic process $N_t$ by $N_t = \sum_{n \ge 1} 1_{\{\tau_n \le t\}} 1_{\{\tau_{\infty} > t\}}$. Then, $(N_t, \mathscr{F}_t)$ is also called {\bf a point process}. Define {\bf the intensity process of $(N_t, \mathscr{F}_t)$} as a nonnegative $\mathscr{F}_t$-progressively measurable process $\lambda_t$ such that $\int_0^t\lambda_sds$ is the compensator of $(N_t, \mathscr{F}_t)$. The filtration $\mathscr{F}_t$ is called the history of $N_t$, in the sense that $\sigma(N_s; s\le t) \subset \mathscr{F}_t$.

\begin{definition*}[Hawkes process with an exponential kernel]
\label{HawkesProcess}
{\bf A Hawkes process with an exponential kernel (whose intensity starts from $x+\mu$)} is a point process $(N^x_t, \mathscr{F}_t)$ with the $\mathscr{F}_t$-predictable intensity
\begin{eqnarray}
\lambda^x_t = \mu + xe^{-\beta t} + \int_{(0,t)}\alpha e^{-\beta(t-s)}dN^x_s,
\end{eqnarray}
where $\mu, \alpha$ and $\beta$ are positive constants with $\alpha/\beta<1$ and $x \ge 0$.
\end{definition*}

The following remarks are fundamental. Let $\mathscr{F}^x_t = \sigma(N^x_s; s\le t)$.

\begin{remark*}
{\rm There exists a Hawkes process with an exponential kernel $N^x_t$ with the history $\{\mathscr{F}^x_t\}_{t \in \reels_+}$. One may prove this existence in the same fashion as the proof of Theorem 1(a) in \cite{Bremaud}. Moreover, in this construction, the measurability of $x \mapsto N^x_t$ and $x \mapsto \lambda^x_t$ are brought about spontaneously.}
\end{remark*}

\begin{remark*}
\label{finite path}
{\rm Denote the $n$-th jump time of $N^x$ by $\tau^x_n$ and let $\tau^x_{\infty} = \lim_{n \iku \infty} \tau^x_n$. Then $\tau^x_{\infty} = \infty \ a.s$. In particular, $N^x_t$ and $\lambda^x_t$ have finite paths.}
\end{remark*}

\begin{proof}[{\bf Proof of Remark \ref{finite path}}]
Fix $t \ge 0$ arbitrarily. From the definition of $N_t^x$, we have
\begin{eqnarray*}
E\left[N^x_{t\wedge\tau^x_n} \right] 
= E\left[\int_{(0,t\wedge\tau^x_n]}\lambda^x_sds \right] 
&=& \mu E[t\wedge\tau^x_n] + \frac{x}{\beta}E[1-e^{-\beta(t\wedge\tau^x_n)}] + \frac{\alpha}{\beta} E\left[\int_{(0,t\wedge\tau^x_n]}(1-e^{-\beta(t\wedge\tau^x_n-u)})dN^x_u\right]\\
&\le&\mu t +  \frac{x}{\beta} + \frac{\alpha}{\beta}E\left[N^x_{t\wedge\tau^x_n} \right].
\end{eqnarray*}
 Thus, 
\begin{eqnarray*}
\label{3A}
E\left[N^x_{t\wedge\tau^x_n} \right] \le \frac{\beta\mu t + x}{\beta-\alpha}.
\end{eqnarray*}
However, if $P[\tau^x_{\infty} \le t] > 0$ holds, the above inequality and the monotone convergence theorem imply that $N^x_{\tau^x_{\infty}-} < \infty$ on $\{\tau^x_{\infty} \le t\}$ a.s. and it contradict the definition of $\tau^x_{\infty}$. Therefore, $\tau^x_{\infty} > t$ a.s. Since $t$ is arbitrarily, $\tau^x_{\infty} = \infty$ holds almost surely.
  \end{proof}


\subsection{Markovian property of exponential Hawkes intensity}

The main purpose of this subsection is revealing the Markovian property of the Hawkes process with an exponential kernel. This property is well-known, see \cite{Oakes1975}. However, in \cite{Oakes1975}, the way of definition of the Hawkes process is somewhat different from ours. Therefore, to strictly handle the Markovian property under our settings and to make this paper self-contained, we will give another proof via the extended generator of the intensity process $\lambda^x_t$ that is defined later. In this subsection, we deal with the outline only. The detail of proofs can be found in Appendix.\\

Before getting into the main topic, we define some symbols used throughout this and the next Subsection. Let $N^x_t$ be a Hawkes process as defined in Definition \ref{HawkesProcess}, and set $\{\mathscr{F}^x_t\}_{t \in \reels_+}$ as the history of $N^x$. By considering a sufficiently rich $\Omega$, $\{\mathscr{F}^x_t\}_{t \in \reels_+}$ is regarded as a right-continuous filtration, see Lemma 18.4 in \cite{LS}. Moreover, we regard that $\{\mathscr{F}^x_t\}_{t \in \reels_+}$ is augmented and use the same notation $\{\mathscr{F}^x_t\}_{t \in \reels_+}$. Then,  $\{\mathscr{F}^x_t\}_{t \in \reels_+}$ satisfies the usual condition. \\

First, we see that $\lambda^x_t$ has the finiteness of its moment. To show this property, we prepare the following lemma that is shown in the proof of Proposition4.5 in \cite{ClinetYoshida}. 

\begin{lemma*}
\label{generator}
Let $\alpha, \beta$ and $\mu$ be parameters of the Hawkes process $N^x_t$. For a differentiable function $f$, we define the operator $\mathscr{A}$ by
\begin{eqnarray}
\mathscr{A}f(y) = y\left(f(y+\alpha) - f(y)\right) - \beta(y-\mu)\frac{d}{dy}f(y).
\label{exgenerator}
\end{eqnarray}
Then, there exist positive constants $M_1$, $K_1$ and $K_2$ such that for $V(y) = e^{M_1y}$
\begin{eqnarray*}
\mathscr{A}V(y) \le -K_1V(y) + K_2.
\end{eqnarray*}
\end{lemma*}

\noindent Lemma \ref{generator} ensures the existence of the moment-generating function of $\lambda^x_t$ on a neighborhood of the origin:

\begin{proposition*}
\label{strong moment}
There exists a positive constant $M_1$ such that 
\begin{eqnarray*}
\sup_{t \in \reels_+}E\left[e^{M_1\lambda^x_t}\right] < \infty.
\end{eqnarray*}
\end{proposition*}

Second, we see the Markov property of the exponential Hawkes intensity. To show this property, we use the idea of the extended generator that is an extension of the infinitesimal generator. It can be found, for instance, by  \cite{MeynTweedie3}. 

\begin{definition*}
Let $(X_t, \mathscr{F}_t)$ be a $d$-dimensional adapted process. We denote by $Dom(\mathscr{A})$ the set of all measurable functions $f : \reels^d \iku \reels$ for which there exists a measurable function $U : \reels^d \iku \reels$ such that $f(X_t) - f(X_0) - \int_{(0, t]}U(X_s)ds$ is a $\mathscr{F}_t$-martingale. Then, we write $U = \mathscr{A}f$ and call $\mathscr{A}$ as {\bf the extended generator of $(X_t, \mathscr{F}_t)$}.
\end{definition*}

For a differentiable function $f : \reels \to \reels$, we define the operator $\mathscr{A}$ by (\ref{exgenerator}) as one of the extended generator of $(\lambda^x_t, \mathscr{F}^x_t)$. We investigate the domain of $\mathscr{A}$. Let $\mathscr{P} = \{\text{polynomial functions on $\reels$}\}$. Then, the next lemma follows.

\begin{lemma*}
\label{P in DomA}
$\mathscr{P} \subset Dom(\mathscr{A})$.
\end{lemma*}

\noindent From Lemma \ref{P in DomA} and the definition of $\mathscr{A}$, $\mathscr{A}^kp \in Dom(\mathscr{A})$ holds for any $p \in \mathscr{P}$ and $k \in \naturels$. Therefore, for any $p \in \mathscr{P}$, an inductive calculation yields
\begin{eqnarray*}
E\left[\left.p(\lambda^x_t) - p(\lambda^x_s)\right| \mathscr{F}^x_s\right]
&=& E\left[\left.\int_{(s,t]}\mathscr{A}p(\lambda^x_{u_1})du_1\right| \mathscr{F}^x_s\right] = \int_{(s,t]}E\left[\left.\mathscr{A}p(\lambda^x_{u_1})\right| \mathscr{F}^x_s\right]du_1\\
&=& (t-s)\mathscr{A}p(\lambda^x_s) + \int_{(s,t]}E\left[\left.\mathscr{A}p(\lambda^x_{u_1}) - \mathscr{A}p(\lambda^x_s)\right| \mathscr{F}^x_s\right]du_1\\
&=& (t-s)\mathscr{A}p(\lambda^x_s) + \int_{(s,t]}E\left[\left.\int_{(s,u_1]}\mathscr{A}^2p(\lambda^x_{u_2})du_2\right| \mathscr{F}^x_s\right]du_1\\
&=& (t-s)\mathscr{A}p(\lambda^x_s) + \int_{(s,t]}\int_{(s,u_1]}E\left[\left.\mathscr{A}^2p(\lambda^x_{u_2})\right| \mathscr{F}^x_s\right]du_2du_1\\
&=&\cdots =\sum^{n-1}_{k=1}\frac{(t-s)^k}{k!}\mathscr{A}^kp(\lambda^x_{s}) + \int_{(s,t]}\int_{(s,u_1]}\cdots\int_{(s,u_{n-1}]}E\left[\left.\mathscr{A}^np(\lambda^x_{u_n})\right| \mathscr{F}^x_s\right]du_n\dots du_2du_1,
\end{eqnarray*}
where $\mathscr{A}^kp(\lambda^x_t)$ means $\mathscr{A}^kp(y)|_{y = \lambda^x_t}$. Furthermore, the remainder term converges to $0$ in $L^1$-sense, that is;

\begin{lemma*}
\label{Remainder}
For any $p \in \mathscr{P}$,
\begin{eqnarray*}
 \int_{(s,t]}\int_{(s,u_1]}\cdots\int_{(s,u_{n-1}]}E\left[\left.\mathscr{A}^np(\lambda^x_{u_n})\right| \mathscr{F}^x_s\right]du_n\dots du_2du_1 \iku 0 \quad \text{as $n\iku\infty$ in $L^1$.}
\end{eqnarray*}
\end{lemma*}

\noindent Set the operator $e^{(t-s)\mathscr{A}}f(\lambda^x_s) = \sum^{\infty}_{k=0}\frac{(t-s)^k}{k!}\mathscr{A}^kf(\lambda^x_{s})$ for a function $f$. From Lemma \ref{Remainder}, we have $E[p(\lambda^x_t)| \mathscr{F}^x_s] = e^{(t-s)\mathscr{A}}p(\lambda^x_s) \ a.s.$ for any $p \in \mathscr{P}$. Then, since $e^{(t-s)\mathscr{A}}p(\lambda^x_s)$ is $\sigma(\lambda^x_s)$-measurable, we immediately get the Markovian property only for $p \in \mathscr{P}$, i.e.
\begin{eqnarray}
\label{Markov for polynomial}
E[p(\lambda^x_t)| \mathscr{F}^x_s] = E[p(\lambda^x_t)| \lambda^x_s] \quad a.s.
\end{eqnarray}
For the sake of Proposition \ref{strong moment}, the equation (\ref{Markov for polynomial}) can be extended to the Markovian property for any bounded function. We summarize this statement as the following theorem.

\begin{theorem*}
\label{Markovian property}
For any bounded measurable function $f$, $E[f(\lambda^x_t)| \mathscr{F}^x_s] = E[f(\lambda^x_t)| \lambda^x_s] \ \ a.s.$
\end{theorem*}


\subsection{Markovian property and Ergodicity of Hawkes core process}

To consider the asymptotic expansion for the MLE of the Hawkes process, we have to deal with the derivatives of the log-likelihood process of the Hawkes process with respect to its parameters. Moreover, these derivatives are represented as functionals of the derivatives of the Hawkes intensity process with respect to time. First, we introduce the concept of the Hawkes core process.  For $x_1 \in \reels_+$, let $N^{x_1}_t$ be an exponential Hawkes process with intensity $\lambda_t^{x_1} = \mu + x_1e^{-\beta t} + \int_{(0,t)}\alpha e^{-\beta(t-u)}dN_u^{x_1}$ defined as in Definition \ref{hawkes def}. 

\begin{definition*}
For $n \in \naturels$ and $x = (x_1, \dots, x_n)\in \reels^{n}$, we define
\begin{eqnarray*}
X_t^{x,(n)} = e^{-\beta t} \sum_{k=1}^{n} \Bigg( \begin{array}{cc} n-1 \\ k-1 \end{array} \Bigg) x_k t^{n - k} + \int_{(0,t)} \alpha(t-u)^{n-1} e^{-\beta(t-u)}dN_u^{x_1}.
\end{eqnarray*}
We call $X_t^{x,(n)}$ as {\bf the $n$-th Hawkes core process of $N^{x_1}_t$ }.
\end{definition*}

Remark that $X_t^{x,(n)}$ is obviously $\sigma\big(X_u^{x_1,(1)} ; u \le t\big)$-measurable for any $n \in \naturels$. However, $X_t^{x,(n)} - X_s^{x,(n)}$ is not $\sigma\big(X_u^{x_1,(1)} ; s < u \le t\big)$-measurable for $0 \le s < t$ and $n \ge 2$. When we consider the second order derivative of the Hawkes intensity $\lambda_t^{x_1}$, the following process is essential:

\begin{eqnarray}
\label{X}
X_t^x 
= \left(X_t^{x,(1)}, X_t^{x,(2)}, X_t^{x,(3)}  \right)'
\end{eqnarray}
where $x = (x_1, x_2, x_3) \in \reels^3_+$.\footnote{$X_t^{x,(1)}$ and $X_t^{x,(2)}$ mean $X_t^{x_1,(1)}$ and $X_t^{(x_1, x_2),(2)}$, respectively.} The detailed reason why we consider $X_t^x$ is explained in Section 4. Here, we reveal the properties of $X_t^x$. In this subsection, we deal only with the overview. The detail of proofs can be found in Appendix. 

We can also deduce the Markovian property of the process $X^x_t$.

\begin{proposition*}
\label{X Markovian property}
For any bounded measurable function $f$, $E[f(X^x_t)| \mathscr{F}^x_s] = E[f(X^x_t)| X^x_s] \ \ a.s.$ where $\mathscr{F}^x_t = \sigma( X^x_s ; s \le t )$.
\end{proposition*}

Second, we see the time-homogeneous Markovian property. We define the Markov kernel as below. For $x \in \reels^3_+$ and $A \in \mathscr{B}(\reels^3_+)$,
\begin{eqnarray}
\label{Markov kernel}
P^t(x, A) = P[X^x_t  \in A].
\end{eqnarray}
For this Markov kernel, the time-homogeneous property holds, i.e.;
\begin{proposition*}
\label{X homo property}
For any bounded measurable function $f$,
\begin{eqnarray*}
E\left[\left.f(X^x_t) \right| \mathscr{F}^x_s\right] = \int_{\reels^3_+} f(y) P^{t-s}(X^x_s, dy) \quad a.s.
\end{eqnarray*}
\end{proposition*}

We concretely give the invariant measure of $X_t^x$ under our settings. On some probability space $(\bar{\Omega}, \bar{\mathscr{F}}, \bar{P})$, there exists a stationary multivariate Hawkes process $(\bar{N}=(\bar{N}^{(1)},\bar{N}^{(2)},\bar{N}^{(3)}), \mathscr{F}^{\bar{N}}_t)$ with the $\mathscr{F}^{\bar{N}}_t$-intensity $\bar{\lambda}_t = (\bar{\lambda}^{(1)}_t,\bar{\lambda}^{(2)}_t,\bar{\lambda}^{(3)}_t)$ such that
\begin{eqnarray*}
\bar{\lambda}^{(1)}_t = \mu + \int_{(-\infty, t)}\alpha e^{-\beta(t-s)}d\bar{N}^{(1)}_s,\quad
\bar{\lambda}^{(2)}_t = \mu + \int_{(-\infty, t)}\alpha (t-s)e^{-\beta(t-s)}d\bar{N}^{(1)}_s\\
\text{and} \quad
\bar{\lambda}^{(3)}_t = \mu + \int_{(-\infty, t)}\alpha (t-s)^2e^{-\beta(t-s)}d\bar{N}^{(1)}_s,
\end{eqnarray*}
where $\mathscr{F}^{\bar{N}}_t = \vee_{i=1}^3 \mathscr{F}^{\bar{N}^{(i)}}_t$ and $ \mathscr{F}^{\bar{N}^{(i)}}_t = \sigma(\bar{N}^{(i)}[C]; C \in \mathscr{B}(\reels), C \subset (-\infty, t])$, see Theorem 7 in \cite{Bremaud}. We write $\bar{N} = \bar{N}^{(1)}$, allowing the abuse of the notation. Then, the following process $\bar{X}$ is stationary:
\begin{eqnarray*}
\bar{X}_t
&=& \Bigg(\int_{(-\infty,t)} \alpha e^{-\beta(t-u)}d\bar{N}_u, \quad \int_{(-\infty,t)} \alpha (t-u)e^{-\beta(t-u)}d\bar{N}_u, \quad \int_{(-\infty,t)} \alpha (t-u)^2e^{-\beta(t-u)}d\bar{N}_u\Bigg)'.
\end{eqnarray*}

Denote the distribution of $\bar{X}_t$ by $P^{\bar{X}}$.\footnote{An abusive use of ``P" : $P^{\bar{X}}[A] = \bar{P}[\bar{X} \in A]$ for $A \in \mathscr{B}(\reels^3_+)$.}

\begin{proposition*}
\label{inv meas}
$P^{\bar{X}}$ is the invariant probability measure for $X^x_t$, i.e. for any $t \ge 0$ and $A \in \mathscr{B}(\reels^3_+)$, 
\begin{eqnarray*}
P^{\bar{X}}[A] = \int_{\reels^3_+} P^t(x, A) P^{\bar{X}}(dx).
\end{eqnarray*}
\end{proposition*}

Furthermore, $X^x_t$ has a strong finiteness of moments same as $\lambda^x_t$

\begin{proposition*}
\label{X strong moment}
There exists a positive constant vector $M = (M_1, M_2, M_3)$ such that 
\begin{eqnarray*}
\sup_{t \in \reels_+}E\left[e^{MX^x_t}\right] < \infty.
\end{eqnarray*}
\end{proposition*}

Finally, similarly to Proposition 4.5 in \cite{ClinetYoshida}, it is ensured that $X_t^x$ has the geometric ergodicity in the following meaning.

\begin{proposition*}
\label{geometric ergodicity}
There exist a positive constant vector $M = (M_1, M_2, M_3)$ and positive constants $B > 0$ and $r \in (0, 1)$ such that
\begin{eqnarray*}
\left\| P^t(x, \cdot) - P^{\bar{X}} \right\|_{e^{M\cdot}} \le B(e^{Mx} + 1)r^t.
\end{eqnarray*}
Here, for a measurable function $V\ge1$, $\|\cdot\|_V$ designates the $V$-variation norm, i.e. for any signed measure $\mu$ on a measurable space
$(S, \mathscr{S})$,
\begin{eqnarray*}
\| \mu \|_V = \sup_{\psi; |\psi| \le V} \left| \int_S \psi(x) \mu(dx)\right|.
\end{eqnarray*}
\end{proposition*}


\section{Edgeworth expansion for functionals related the Hawkes process}

As mentioned in Introduction, both computation of the maximum likelihood estimator (MLE)
and simulation methods for the one-dimensional exponential Hawkes process was revealed in \cite{OgataMLE}. Furthermore, it was proved that the quasi maximum likelihood estimator (QMLE) of the multi-dimensional exponential Hawkes process has the asymptotic normality and the convergence of moments, see \cite{ClinetYoshida}. In this section, we apply Theorem \ref{Main Thm2 New} to the MLE of the one-dimensional exponential Hawkes process whose intensity starts from $\mu_0$ and we will give the second order asymptotic expansion for the distribution of the MLE. Proofs of each statement are given in Appendix. First, we prepare the necessary notation and establish the conditions.\\

Let $(\Omega, \mathscr{F}, P)$ be a probability space. As in Definition \ref{hawkes def}, let $N_t$ be an exponential Hawkes process with the $\mathscr{F}_t = \sigma(N_s; s\le t)$-predictable intensity 
\begin{eqnarray*}
\lambda_t = \mu_0 + \int_{(0,t)}\alpha_0 e^{-\beta_0(t-s)}dN_s.
\end{eqnarray*}
Moreover, we define parametrized intensity process by 
\begin{eqnarray*}
\lambda_t(\theta) = \mu + \int_{(0,t)}\alpha e^{-\beta(t-s)}dN_s \quad \text{for $\theta = (\mu, \alpha, \beta)$.}
\end{eqnarray*}
Furthermore, we define $X_t(\theta)$ by referring (\ref{X}):
\begin{eqnarray*}
X_t(\theta) 
&=& \Bigg( \int_{(0,t)} \alpha e^{-\beta(t-u)}dN_u, \quad  \int_{(0,t)} \alpha (t-u)e^{-\beta(t-u)}dN_u, \quad \int_{(0,t)} \alpha (t-u)^2e^{-\beta(t-u)}dN_u \Bigg)'
\end{eqnarray*}
for $\theta = (\mu, \alpha, \beta)$.
We consider a relatively compact and open parameter set $\Theta \subset \reels_+^3$. Assume that $\theta_0 = (\mu_0, \alpha_0 , \beta_0 ) \in \Theta$ is the true parameter. If there is no confusion, we often omit the true parameter, i.e. write $\lambda_t(\theta_0) = \lambda_t$, $X_t(\theta_0) = X_t$ and so on. The log-likelihood process of $\lambda_t(\theta)$ is defined by
\begin{eqnarray*}
l_T(\theta) = \int^T_0 \log\left(\lambda_s(\theta)\right) dN_s - \int^T_0 \lambda_s(\theta) ds
\end{eqnarray*}
for $\theta \in \Theta$.  Since Lemma A.5 in \cite{ClinetYoshida} guarantees a verification of the permutation of the symbols $\partial_{\theta}$ and $\int_0^T$, the derivative of the log-likelihood process with respect to their parameter can be calculated as below.
\begin{eqnarray}
\label{1st deriv}
\partial_{\theta} l_T(\theta) |_{\theta=\theta_0}
= \int^T_0  \frac{\partial_{\theta}\lambda_s}{\lambda_s}d\tilde{N}_s,
\end{eqnarray}
where $\tilde{N}_t = N_t - \int_0^t \lambda_s ds$. Moreover, 
\begin{eqnarray*}
\partial_{\theta}^2 l_T(\theta) |_{\theta=\theta_0}
=  \int^T_0  \frac{\partial_{\theta}^2\lambda_s-\left(\partial_{\theta}\lambda_s\right)^{\otimes 2}}{\lambda_s^2}d\tilde{N}_s
- \int^T_0  \frac{\left(\partial_{\theta}\lambda_s\right)^{\otimes 2}}{\lambda_s}ds,
\end{eqnarray*}
where, for a vector $x \in \reels^k$, $x^{\otimes 2}$ stands for the product of $x$ and its transpose, i.e. $x^{\otimes 2} = xx' \in \reels^{k \times k}$. Let the operator $\mathbb{D}$ be as in (\ref{D}), for example, $\mathbb{D}^2= \frac{\partial}{\partial\alpha}$, $\mathbb{D}^{(1,2)} = \frac{\partial^2}{\partial\mu \partial\alpha}$, etc. Note that, when we write $X_t = (X^{(1)}_t, X^{(2)}_t, X^{(3)}_t)$, then $\partial_{\theta}\lambda_s$ and $\partial^2_{\theta}\lambda_s$ are computed as
\begin{eqnarray*}
\partial_{\theta}\lambda_s
=\left(
    \begin{array}{c}
      \mathbb{D}^{1}\lambda_s \\
      \mathbb{D}^{2}\lambda_s \\
      \mathbb{D}^{3}\lambda_s 
    \end{array}
  \right)
=\left(
    \begin{array}{c}
      1 \\
      \alpha_0^{-1}X^{(1)}_t\\
      -X^{(2)}_t
    \end{array}
  \right)
\quad \text{and} \quad 
\partial_{\theta}^2\lambda_s
=\left( \mathbb{D}^{(i,j)}\lambda_s  \right)_{i,j=1,2,3}
=\left(
    \begin{array}{ccc}
      0 & 0 & 0\\
      0 & 0 & -\alpha_0^{-1}X^{(2)}_t\\
      0 & -\alpha_0^{-1}X^{(2)}_t  & X^{(3)}_t
    \end{array}
  \right)
\end{eqnarray*}
respectively. Corresponding to Subsection 2.2, we introduce some notation. Let $l_{a_1 \cdots a_k}(\theta) = \mathbb{D}^{(a_1, \dots ,a_k)} l_T(\theta)$ and $\nu_{a_1 \cdots a_k}(\theta) = E\left[\frac{1}{T}l_{a_1 \cdots a_k}(\theta)\right]$ for integers $a_1, \dots ,a_k$. We have the following representations:
\begin{eqnarray*}
\partial_{\theta} l_T 
=\left(
     l_1, l_2, l_3
  \right)'
=\left(
      \int^T_0  \frac{1}{\lambda_s}d\tilde{N}_s, \quad
      \int^T_0  \frac{X^{(1)}_s}{\alpha_0 \lambda_s}d\tilde{N}_s, \quad
      -\int^T_0  \frac{X^{(2)}_s}{\lambda_s}d\tilde{N}_s
  \right)'
\end{eqnarray*}
and $\partial_{\theta}^2 l_T = \left(l_{ij} \right)_{i,j=1, 2, 3}$, where $l_{ij}$ is symmetric with respect to $i, j$ and each component has the following representation:
\begin{eqnarray*}
&&l_{11}  = -\left( \int^T_0  \frac{1}{\lambda_s^2}d\tilde{N}_s + \int^T_0  \frac{1}{\lambda_s}ds \right), \quad 
l_{12}  = -\frac{1}{\alpha_0} \left( \int^T_0  \frac{X^{(1)}_s}{\lambda_s^2}d\tilde{N}_s + \int^T_0  \frac{X^{(1)}_s}{\lambda_s}ds \right),  \\
&&l_{13}  = \int^T_0  \frac{X^{(2)}_s}{\lambda_s^2}d\tilde{N}_s + \int^T_0  \frac{X^{(2)}_s}{\lambda_s}ds, \quad 
l_{22}  = -\frac{1}{\alpha_0^2} \left( \int^T_0  \frac{\big(X^{(1)}_s\big)^2}{ \lambda_s^2}d\tilde{N}_s + \int^T_0  \frac{\big(X^{(1)}_s\big)^2}{ \lambda_s}ds \right),  \\
&&l_{23} = \frac{1}{\alpha_0} \left( \int^T_0  \frac{\big(X^{(1)}_s - 1\big)X^{(2)}_s}{\lambda_s^2}d\tilde{N}_s + \int^T_0  \frac{X^{(1)}_sX^{(2)}_s}{ \lambda_s}ds \right), \  \text{and }  \
l_{33} =  \int^T_0  \frac{X^{(3)}_s - \big(X^{(2)}_s\big)^2}{\lambda_s^2}d\tilde{N}_s - \int^T_0  \frac{\big(X^{(2)}_s\big)^2}{\lambda_s}ds.
\end{eqnarray*}

Put $\mathscr{B}_I = \bigcap_{\ep>0} \sigma(X_u ; u \in [s, t+\ep]) \vee \mathscr{N}$ for $I = [s, t] \subset \reels_+$, where $\mathscr{N}$ is the $\sigma$-field generated by null sets in $\mathscr{F}$. Let 
\begin{eqnarray*}
Z_a = \frac{1}{\sqrt{T}}l_a
\quad \text{and} \quad
Z_{ab} =\sqrt{T}\left(\frac{1}{T}l_{ab} - \nu_{ab}\right).
\end{eqnarray*}
Moreover, we write
\begin{eqnarray*}
Z_T^{(1)} = T^{\frac{1}{2}}(Z_1, Z_2, Z_3) \quad \text{and} \quad Z_T^{(2)} =  T^{\frac{1}{2}}(Z_{11}, Z_{12}, Z_{13}, Z_{21},  Z_{22}, Z_{23}, Z_{31}, Z_{32}, Z_{33}).
\end{eqnarray*}
Finally, we set $Z_T = (Z_T^{(1)}, Z_T^{(2)})$. Then, $Z_t - Z_s$ is $\mathscr{B}_{[s,t]}$-measurable for every $s, t \in \reels_+$, $0 \le s \le t$ and $Z_0 \in \mathscr{F}\mathscr{B}_{[0]}$. About the definition of $\mathscr{B}_I $, note the following points.

\begin{remark*}
{\rm Obviously, $\sigma\big( X^{(1)}_s; s \in I \big) = \sigma( \lambda_s; s \in I )$ holds for any interval $I$.}
\end{remark*}

\begin{remark*}
{\rm For any $s\ge0$, $\sigma( X_s; s \in [0,t] ) \subset \sigma( \lambda_s; s \in [0,t] )$ holds. However, for a general interval $I$ and $s\in I$, $X^{(2)}_s$ and $X^{(3)}_s$ are not always measurable with respect to $\sigma( \lambda_t; t \in I)$. Thus, if we consider the expansion of the distribution of $Z_T$, we have to extend $\sigma( \lambda_t; t \in I)$. In this reason, we introduced the process $X_t$ and defined $\mathscr{B}_{I}$ as above.}
\end{remark*}

\begin{remark*}
{\rm $\sigma( X_t; t \in [u, v])$, in particular $\sigma\big( X^{(1)}_t; t \in [u, v]\big)$, has almost all the information of $\sigma(N_t -N_s ; s, t, \in [u, v])$. However, the information of the jump at $v$ is not contained in $\sigma( X_t; t \in [u, v]))$. Therefore, we need to consider the right-continuous $\sigma$-fields.}
\end{remark*}

A functional of the process $X_t$ has the geometric mixing property.

\begin{proposition*}
\label{A1}
$\mathscr{B}_I$ satisfies the condition \textnormal{[A1]}.
\end{proposition*}

From Proposition \ref{A1} and Theorem \ref{Main Thm}, we immediately obtain the asymptotic expansion for the distribution of a functional of the Hawkes core process.

\begin{theorem*}
\label{Main Thm3}
Let $p \in \naturels$ with $p \ge 2$  and $L_1, L_2 > 0$. Assume that a $\mathscr{B}_I$-adapted stochastic process $Y_T$ satisfies the condition \textnormal{[A2]}.  Then, there exist  $D>0$ and $\Gamma \in \naturels$ such that for any $f \in \mathscr{E}(\Gamma, L_1, L_2)$,
\begin{eqnarray*}
\left| E\left[ f\left(\frac{Y_T}{\sqrt{T}}\right)\right] - \int_{\reels^d}f(z) p_{T,p,D}(z)dz \right| = o\left(T^{-(p-2)/2}\right),
\end{eqnarray*}
where $p_{T,p,D}(z)$ is defined as in (\ref{p}) with replaced $Z_T$ by $Y_T$.
\end{theorem*}

We may also apply Theorem \ref{Main Thm2 New}. Write $g_T = (g_{ab})_{a,b=1, 2, 3} = (-\nu_{ab}(\theta_0))_{a,b=1, 2, 3}$. As proved in Appendix, the exponential Hawkes process satisfies the condition [A3]. Thus, we can also define $g_T^{-1} = (g^{ab;})_{a, b=1,2,3}$. The following statement is the main theorem of this article. (For the definition of each symbol, see Section 2.)

\begin{theorem*}
\label{Main Thm4}
Let $L_1, L_2>0$. The conditions \textnormal{[A1]-[A3]} and \textnormal{[C1]} are satisfied. Moreover, there exist $L > 1$, $\gamma \in (0,1)$ and $q_1, q_2, q_3 > 0$ such that \textnormal{[B0]-[B4]} with (\ref{restriction}) hold. Thus, there exists $\Gamma \in \naturels$ such that for any $f \in \mathscr{E}(\Gamma, L_1, L_2)$,
\begin{eqnarray*}
\left| E\left[ f\big(\sqrt{T}( \mle - \theta_0 )\big)\right] - \int_{\reels^d}f(z^{(1)}) q_{T,3}(z^{(1)})dz^{(1)} \right| = o\left(T^{-\frac{1}{2}}\right),
\end{eqnarray*}
where 
\begin{eqnarray*}
q_{T,3}(z^{(1)})
&=& \phi(z^{(1)}; g_T^{-1}) + \frac{1}{\sqrt{T}}\Bigg\{ \left(\frac{1}{6}\tilde{\kappa}^{a_1a_2a_3;}_T + \mu^{a_3;}_{b_1b_2}g^{b_1a_1;}g^{b_2a_2;}\right)h_{a_1a_2a_3}(z^{(1)}; g_T^{-1}) \\
&&+ \mu^{a_1;}_{b_1b_2}g^{b_1b_2;}h_{a_1}(z^{(1)}; g_T^{-1}) \Bigg\}\phi(z^{(1)}; g_T^{-1}),
\end{eqnarray*}
and
\begin{itemize}
\item $\tilde{\lambda}^{a_1a_2a_3;}_T$ is the $(a_1, a_2, a_3)$-cumulant of $g_T^{-1}Z_T^{(1)}$ and $\tilde{\kappa}^{a_1a_2a_3;}_T = T^{1/2}\tilde{\lambda}^{a_1a_2a_3;}_T$;
\item $V^{a;}_{bc} = Cov[Z^{a;}_{b}, Z^{a_1;}]g_{a_1c}$ and $\mu^{a;}_{a_1a_2} = \big( V^{a;}_{a_1a_2} + V^{a;}_{a_2a_1} + \nu^{a;}_{a_1a_2} \big)/2$.
\end{itemize}

\end{theorem*}


\section{Simulation}
In this section, we show the result of a simulation for Theorem \ref{Main Thm4}. We need to compute $g_T^{-1}$, $\tilde{\kappa}^{a_1a_2a_3;}_T$ and $\mu^{a;}_{a_1a_2}$ in Theorem \ref{Main Thm4}. However, it is difficult to get their expressions for the true parameter. Here, we approximate these values numerically using the Monte Carlo method. It must be emphasized that the simulation here is not exact in this sense. Furthermore, when we model real data by the Hawkes process, of course, we do not know the true parameter. One solution to this problem is to use an estimator instead of the true parameter. This is nothing but the bootstrap method. An error evaluations of the bootstrap method are for further study. All experiments are done by using R. The code can be found on GitHub page \url{https://github.com/goda235/Edgeworth_expansion_for_Hawkes_MLE}.

By using the algorithm in \cite{OgataSimu}, we simulate the values of the Hawkes process for $MC$ times. From these data, we can get $MC$ number of values for the Hawkes core process $X_s^{(1)}, X_s^{(2)}$ and $X_s^{(3)}$, in particular, $MC$ values for $Z_T$. From these data, we can get the unbiased estimator of $Var[Z_T]$. With the help of the condition [B0] (iii), we can compute the value of $g_T^{-1}$ from $Var[Z_T]$. Since $\nu_{a_1a_2a_3}$ have the following representation;


\begin{eqnarray*}
&&\nu_{111} =  E\left[ \frac{1}{T}\sum_{i: \tau_i \le T} \frac{2}{\big(\mu_0 + X^{(1)}_{\tau_i}\big)^3} \right], \
\nu_{112} =  E\left[ \frac{1}{T}\sum_{i: \tau_i \le T} \frac{2X^{(1)}_{\tau_i}}{\alpha_0\big(\mu_0 + X^{(1)}_{\tau_i}\big)^3} \right], \ 
\nu_{113} =  E\left[ \frac{1}{T}\sum_{i: \tau_i \le T} -\frac{2X^{(2)}_{\tau_i}}{\big(\mu_0 + X^{(1)}_{\tau_i}\big)^3} \right], \\
&&\nu_{221} =  E\left[ \frac{1}{T}\sum_{i: \tau_i \le T} \frac{2\big(X^{(1)}_{\tau_i}\big)^2}{\alpha_0^2\big(\mu_0 + X^{(1)}_{\tau_i}\big)^3} \right], \ 
\nu_{222} = E\left[ \frac{1}{T}\sum_{i: \tau_i \le T} \frac{2\big(X^{(1)}_{\tau_i}\big)^3}{\alpha_0^3\big(\mu_0 + X^{(1)}_{\tau_i}\big)^3} \right], \
\nu_{223} = E\left[ \frac{1}{T}\sum_{i: \tau_i \le T} \frac{2\mu_0X^{(1)}_{\tau_i}X^{(2)}_{\tau_i}}{\alpha_0^2\big(\mu_0 + X^{(1)}_{\tau_i}\big)^3} \right], \\
&&\nu_{331} =  E\left[ \frac{1}{T}\sum_{i: \tau_i \le T} \frac{2\big(X^{(2)}_{\tau_i}\big)^2 - X^{(3)}_{\tau_i} \big(\mu_0 + X^{(1)}_{\tau_i}\big)}{\big(\mu_0 + X^{(1)}_{\tau_i}\big)^3} \right], \nu_{332} = E\left[ \frac{1}{T}\sum_{i: \tau_i \le T} \frac{-X^{(1)}_{\tau_i}X^{(3)}_{\tau_i}\big(\mu_0 + X^{(1)}_{\tau_i}\big) - 2\mu_0\big(X^{(2)}_{\tau_i}\big)^2}{\alpha_0\big(\mu_0 + X^{(1)}_{\tau_i}\big)^3} \right], \\
&&\nu_{333} = E\left[ \frac{1}{T}\sum_{i: \tau_i \le T} \frac{3X^{(2)}_{\tau_i}X^{(3)}_{\tau_i}\big(\mu_0 + X^{(1)}_{\tau_i}\big) -2\big(X^{(2)}_{\tau_i}\big)^3}{\big(\mu_0 + X^{(1)}_{\tau_i}\big)^3} \right], \
\nu_{123} =  E\left[ \frac{1}{T}\sum_{i: \tau_i \le T} \frac{\mu_0X^{(2)}_{\tau_i} - X^{(1)}_{\tau_i}X^{(2)}_{\tau_i}}{\alpha_0\big(\mu_0 + X^{(1)}_{\tau_i}\big)^3} \right],
\end{eqnarray*}

then we can compute an approximated value of $\mu^{a;}_{a_1a_2}$ by taking mean. From the representation of cumulants by moment, $\tilde{\kappa}^{a_1a_2a_3;}$ is computed from means of $Z_{a_1}Z_{a_2}Z_{a_3}$.

\

We set an exponential Hawkes process $N_t$ with its parameters $\mu = 0.5, \alpha = 1.0$ and $\beta = 1.3$, i.e. its intensity $\lambda_t$ has the representation

\begin{eqnarray*}
\lambda_t = 0.5 + \int_{(0,t)} 1.0 e^{-1.3(t-s)}dN_s.
\end{eqnarray*}

We set an observation time $T = 30$. For this model, we compute the MLE for $3000$ times and obtain histograms for each parameter.  In addition, we add the density function curves for the marginal distribution of $\phi(z; g_T^{-1})$ and the marginal distribution of $q_{T,3}(z)$. Here, $q_{T,3}(z)$ is computed by the above method with $MC=5000$. The curve of $\phi(z; g_T^{-1})$ is described by a broken line, and $q_{T,3}(z)$ is described by a solid line. The simulation results are as follows.

\begin{figure}[H]
	\centering
	\includegraphics[width=14cm]{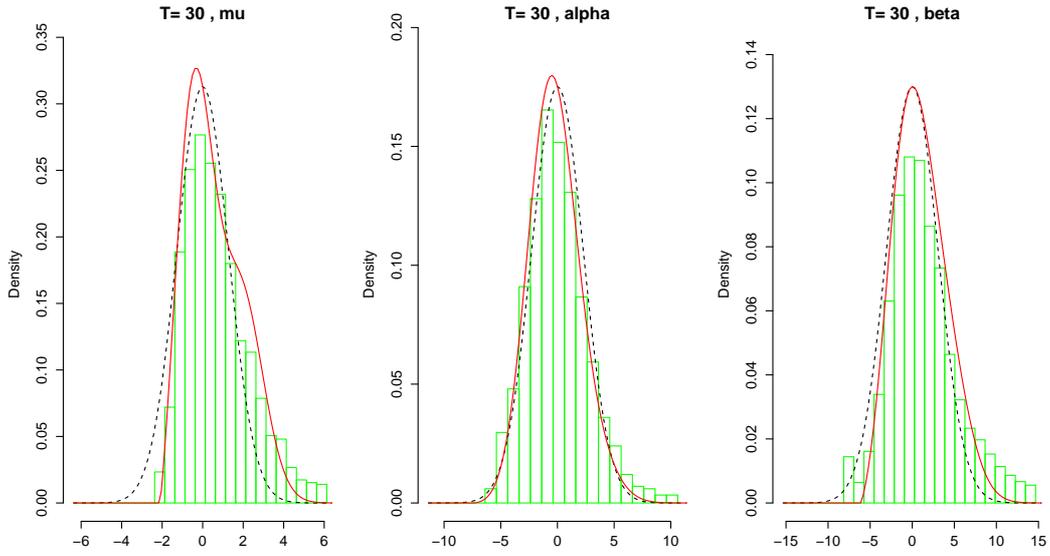}
	\caption{Histogram of MLE in the case of $T = 30$. }
\end{figure}

We can see that the curve of $q_{T,3}(z)$ fits the histogram better than the normal distribution. The next figure is Q-Q plot for each marginal distribution.

\begin{figure}[H]
	\centering
	\includegraphics[width=14cm]{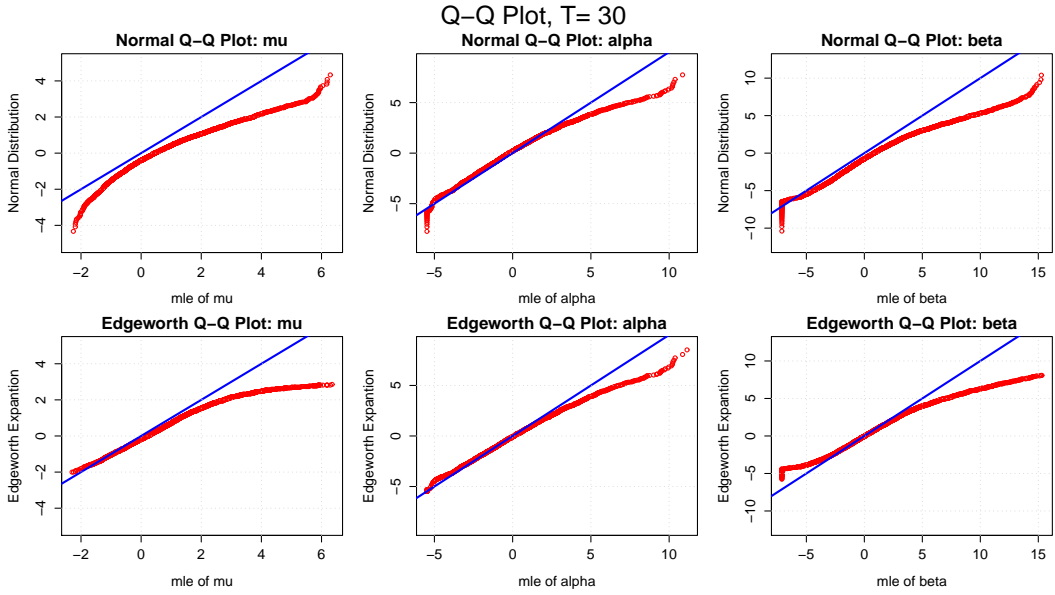}
	\caption{Q-Q plots of each distribution in the case of $T = 30$. }
\end{figure}

The Q-Q plots also shows that $q_{T,3}(z)$ better fits the data than the normal distribution. Change only the observation time to $T = 300$ and try the simulation in the same situation. The result is as follows.

\begin{figure}[H]
	\centering
	\includegraphics[width=14cm]{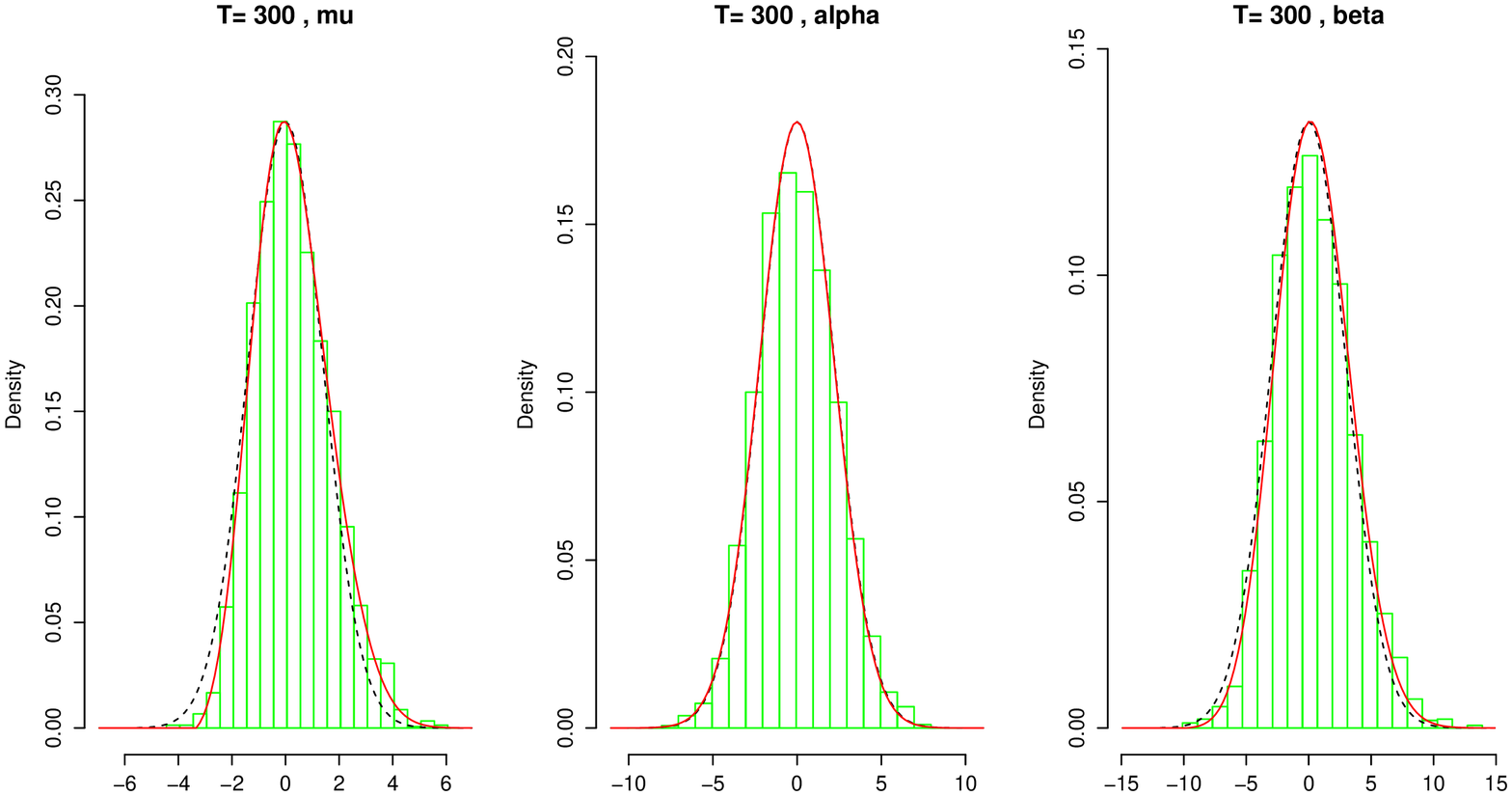}
	\caption{Histogram of MLE in the case of $T = 300$.}
\end{figure}

\begin{figure}[H]
	\centering
	\includegraphics[width=14cm]{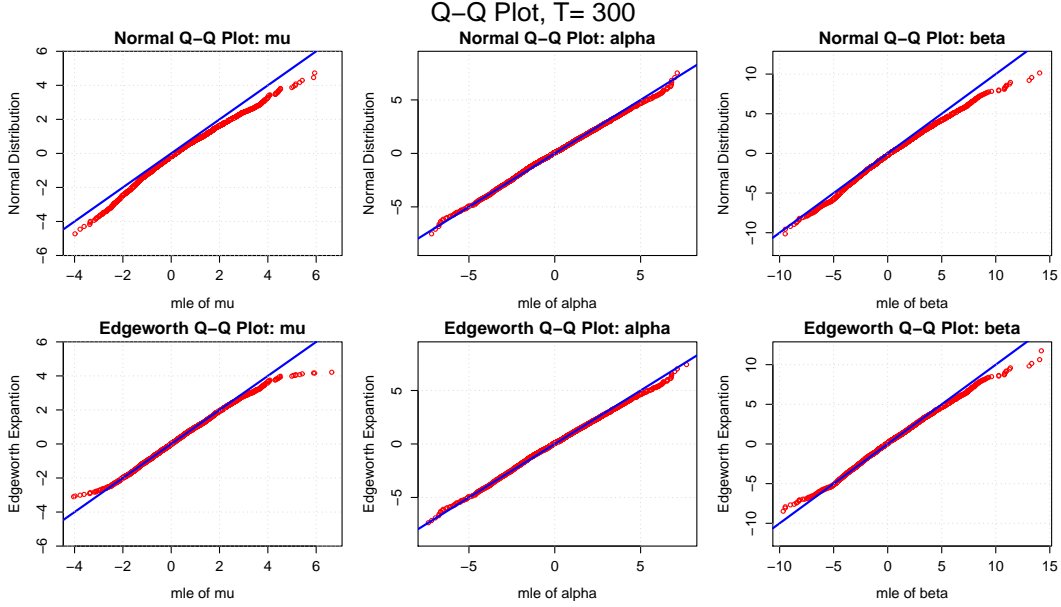}
	\caption{Q-Q plots of each distribution in the case of $T = 300$.}
\end{figure}

When the observation time is sufficiently large, the approximation by $q_{T,3}(z)$ is close to the approximation by normal distribution.


\section{Appendix}
Hereafter, when we write as $X(T) \lesssim T^a$ for $a \in \reels$, it means that there exist positive constants $C$ and $T'$ such that $X(T) \le CT^a$ holds for any $T>T'$.

\subsection{ Proofs of Subsection 2.1}
To prove Theorem \ref{Main Thm}, we should give the asymptotic expansion for the characteristic function of $S_T$. The following discussion is a rework of \cite{GotzeHipp} and \cite{Yoshidaparmix} to a form allowed when the variance is non-degenerate.

First, we introduce some notations. Let $N(T) = \lfloor T \rfloor + 1$ and divide the interval $[0, T]$ into intervals $\{I_i\}_{i= 0, \dots, N(T)}$ such that  $I_0 = [0, 0]$, $I_i = [i-1,i]$ for $i= 1, \dots, N(T) -1$ and $I_{N(T)} = [N(T)-1, T]$. Denote $Z_{I_i}$ as $Z_i$ for any $i = 0, \dots, N(T)$\footnote{We must remark that the notation of $Z_i$ has a different mean in the other section.}. There exists a smooth function $\phi : \reels^d \iku [0,1]$ such that  $\phi(x) = 1$ if $|x| \le 1/2$, and $ \phi(x)=0$ if $|x| \ge 1$. Choose a positive constant  $ \beta \in (0, \frac{1}{2})$ and put $\phi_T(x) = x\phi\big(x/2T^{\beta}\big)$. Then $\phi_T(x) = x$ if $|x| \le T^{\beta}$ and  $\phi_T(x) = 0$ if $|x| \ge 2T^{\beta}$. Let $Z^*_i = \phi_T(Z_{i}) - E[\phi_T(Z_{i})]$ for any $i = 0, \dots, N(T)$ and $S_T^* = T^{-\frac{1}{2}}\sum_{i = 0}^{N(T)} Z^*_i$. Write the characteristic function of $S^*_T$ by $H_T(u) = E[e^{iu'S_T^*}]$ for $u \in \reels^d$. For random variables $X$ and $V$, we define $E[X](V) = E[Xe^{iV}]/E[e^{iV}]$. Define the cumulant of real-valued random variables $X_1, \dots, X_r$ shifted by a random variable $V$ as
\begin{eqnarray*}
\kappa \left[ X_1, \dots, X_r  \right](V) =  \frac{\partial^r}{\partial \ep_1 \cdots \partial \ep_r}\bigg|_{\ep_1 = \cdots = \ep_r = 0} \log \bigg( E\big[\exp\big(i\ep_1X_1 + \cdots + i\ep_rX_r \big)\big](V)\bigg)
\end{eqnarray*}
and write $\kappa \left[ X_1, \dots, X_r \right] = \kappa \left[ X_1, \dots, X_r \right](0)$. In this subsection, we assume that the conditons [A1] and [A2] hold. By using the mixing property, it is possible to evaluate cumulants as follows. Write the $i$-th element of a vector $X$ as $X^{(i)}$.

\begin{proposition*}
\label{cumest_1}
Let $\bar{L} > 0$. Set $r \in \naturels$ with $r \le \bar{L}$ and $a_1, \dots, a_r \in \{1, \dots, d\}$. Then, for any $\ep > 0$, there exists $\delta \in (0,1)$ such that
\begin{eqnarray*}
1_{\left\{ |u| < T^{\delta} \right\}}(u) \left| \kappa \left[ S_T^{* (a_1)}, \dots, S_T^{* (a_r)}  \right] \left( \eta u' S_T^*\right)  \right| \ \lesssim \ T^{-\frac{r-2}{2} + \ep(r-1)} \quad \text{uniformly in $u \in \reels^d$ and $\eta \in [0,1]$}. 
\end{eqnarray*}
\end{proposition*}

\begin{proof}
It follows in a similar way as the proof of Lemma 5 in \cite{Yoshidaparmix}. 
  \end{proof}

The next proposition is similar to Lemma 6 in \cite{Yoshidaparmix}. However, our assumption [A2] is stronger than the assumption in \cite{Yoshidaparmix}, thus we may take an arbitrary $L_3 > 0$ as the following.

\begin{proposition*}
\label{cumest_2}
For any $L_3 > 0$, $r \in \naturels$ and $a_1, \dots, a_r \in \{1, \dots, d\}$,
\begin{eqnarray*}
\left| \kappa \left[ S_T^{* (a_1)}, \dots, S_T^{* (a_r)}  \right] - \kappa \left[ S_T^{(a_1)}, \dots, S_T^{(a_r)}  \right] \right| \
\lesssim \ T^{ - L_3\beta}.
\end{eqnarray*}
\end{proposition*}

\begin{proof}
We immediately get
\begin{eqnarray*}
&&\left| \kappa \left[ S_T^{* (a_1)}, \dots, S_T^{* (a_r)}  \right] - \kappa \left[ S_T^{(a_1)}, \dots, S_T^{(a_r)}  \right] \right| \\
&\le& T^{-r/2}\sum_{0 \le j_1, \dots, j_r \le N(T)} \left| \kappa\left[Z_{j_1}^{*(a_1)}, \dots, Z_{j_r}^{*(a_r)}\right] - \kappa\left[Z_{j_1}^{(a_1)}, \dots, Z_{j_r}^{(a_r)}\right] \right|\\
&\le& T^{-r/2}\sum_{0 \le j_1, \dots, j_r \le N(T)} \sum_{l = 1}^r \sum_{\substack{ \alpha_1, \dots, \alpha_l;\\ \alpha_1 + \dots + \alpha_l = \{1, \dots, r\}}} \frac{(-1)^{l-1}}{l} \left| \prod_{m = 1}^l E \left[  \prod_{i \in \alpha_m} \phi^{(a_i)} _T(Z_{j_i})\right] - \prod_{m = 1}^l E \left[  \prod_{i \in \alpha_m} Z_{j_i}^{(a_i)} \right] \right|.
\end{eqnarray*}

Moreover,
\begin{eqnarray*}
&&\left| \prod_{m = 1}^l E \left[  \prod_{i \in \alpha_m} \phi^{(a_i)} _T(Z_{j_i})\right] - \prod_{m = 1}^l E \left[  \prod_{i \in \alpha_m} Z_{j_i}^{(a_i)} \right] \right|\\
&\le& \sum_{m = 1}^l \left( \prod_{m'  = 1}^{m - 1} \left| E \left[  \prod_{i \in \alpha_{m'}} \phi^{(a_i)} _T(Z_{j_i})\right] \right| \right)
\left| E \left[  \prod_{i \in \alpha_m} \phi^{(a_i)} _T(Z_{j_i}) -  \prod_{i \in \alpha_m} Z_{j_i}^{(a_i)} \right] \right|
\left( \prod_{m'  = m + 1}^{l} \left| E \left[  \prod_{i \in \alpha_{m'}} Z_{j_i}^{(a_i)} \right] \right| \right),
\end{eqnarray*}
and the conditon [A2] yeilds
\begin{eqnarray*}
&&\left| E \left[  \prod_{i \in \alpha_m} \phi^{(a_i)} _T(Z_{j_i}) -  \prod_{i \in \alpha_m} Z_{j_i}^{(a_i)} \right] \right|\\
&=& \left| \sum_{k'=1}^k E \left[  \left(\prod_{i =1}^{k'-1} \phi^{(a_i)} _T(Z_{j_i})\right)
\left(\prod_{i =k'+1}^{k} Z_{j_i}^{(a_i)} \right)
\left( \phi^{(a_{k'})} _T(Z_{j_{k'}}) -  Z_{j_{k'}}^{(a_{k'})} \right) 1_{\left\{ \left| Z_{j_{k'}}^{(a_{k'})} \right| \ge T^{\beta} \right\}} \right] \right| \\
&\le& \left| \sum_{k'=1}^k T^{ - L_3 \beta } E \left[  \left(\prod_{i =1}^{k'-1} \phi^{(a_i)} _T(Z_{j_i})\right)
\left(\prod_{i =k'+1}^{k} Z_{j_i}^{(a_i)} \right)
\left( \phi^{(a_{k'})} _T(Z_{j_{k'}}) -  Z_{j_{k'}}^{(a_{k'})} \right) \left| Z_{j_{k'}}^{(a_{k'})} \right|^{L_3}1_{\left\{ \left| Z_{j_{k'}}^{(a_{k'})} \right| \ge T^{\beta} \right\}} \right] \right| \ \lesssim \ T^{ - L_3 \beta }.
\end{eqnarray*}
Therefore,
\begin{eqnarray*}
\left| \kappa \left[ S_T^{* (a_1)}, \dots, S_T^{* (a_r)}  \right] - \kappa \left[ S_T^{(a_1)}, \dots, S_T^{(a_r)}  \right] \right| 
\ \lesssim \ T^{ - L_3 \beta  + \frac{r}{2}}.
\end{eqnarray*}
Since $L_3$ is arbitrary, we get the conclusion.
  \end{proof}

From Proposition \ref{cumest_1} and Proposition \ref{cumest_2}, we immediately get the following statement.

\begin{corollary*}
\label{cumest_3}
Let $\bar{L} > 0$. Set $r \in \naturels$ with $r \le \bar{L}$ and $a_1, \dots, a_r \in \{1, \dots, d\}$.  Then, for any $\ep > 0$,
\begin{eqnarray*}
\left| \kappa \left[ S_T^{(a_1)}, \dots, S_T^{(a_r)}  \right] \right| \
\lesssim \ T^{-\frac{r-2}{2} + \ep(r-1)}.
\end{eqnarray*}
\end{corollary*}

We evaluate the gap between $H_T(u)$ and its expansion $\hat{\Psi}_{T,p,D}(u)$. Allowing for the abuse of symbols, we define $\mathbb{D}$ as the derivative with respect to $u$ in the same way as (\ref{D}).
\begin{proposition*}
\label{cf}
Let $\bar{L}>0$. There exist $D>0$, $\delta > 0$ and $\delta_0 > d\delta$ such that
\begin{eqnarray*}
1_{\left\{ |u| < T^{\delta} \right\}}(u) \left|\mathbb{D}^{\idx}\left( H_T(u) - \hat{\Psi}_{T,p,D}(u) \right)\right| \
\lesssim \ T^{-\frac{p-2}{2} - \delta_0}
\end{eqnarray*}
uniformly in $u \in \reels^d$ and  $\idx \in \{1, \dots, d\}^l$ with $l \le \bar{L}$.
\end{proposition*}

\begin{proof}
Denote $\kappa^*[u^{\otimes r}](V) = \kappa \left[ u'S_T^*, \dots, u'S_T^* \right] (V)$, $\kappa[u^{\otimes r}](V) = \kappa \left[ u'S_T, \dots, u'S_T \right] (V)$, $\kappa^*[u^{\otimes r}] = \kappa^*[u^{\otimes r}](0)$ and $\kappa[u^{\otimes r}] = \kappa[u^{\otimes r}](0)$. We have
\begin{eqnarray*} 
H_T(u)
= \hat{\Psi}^*_{T,p}(u) + R^*_{p+1}(u),
\end{eqnarray*}
where
\begin{eqnarray*}
\hat{\Psi}^*_{T,p}(u)
= \exp \left( \chi_{T,2}(u) \right) \left\{ 1 + \sum_{j = 1}^p \sum_{r_1, \dots, r_j = 1}^{p-2} 1_{\{r_1 + \dots + r_j \le p-2\}} (-1)^j i^{r_1+ \dots + r_j } \frac{\kappa^*[u^{\otimes r_1+2}] \cdots \kappa^*[u^{\otimes r_j+2}]}{j!(r_1+2)!\cdots(r_j+2)!} \right\},
\end{eqnarray*}
\begin{eqnarray*}
R^*_{p+1}(u)
&=& \exp \left( \chi_{T,2}(u) \right) \Bigg\{ \sum_{j = 1}^p \sum_{r_1, \dots, r_j = 1}^{p-2} 1_{\{r_1 + \dots + r_j \ge p-1\}} (-1)^j i^{r_1+ \dots + r_j } \frac{\kappa^*[u^{\otimes r_1+2}] \cdots \kappa^*[u^{\otimes r_j+2}]}{j!(r_1+2)!\cdots(r_j+2)!} \\
&&+ \sum_{j=1}^p \sum_{j'=0}^{j-1} \frac{1}{j'!} \Bigg( \begin{array}{ll} j \\ j' \end{array} \Bigg) \left( \sum_{r=3}^p \frac{i^r}{r!} \kappa^*[u^{\otimes r}] \right)^{j'} \left( R_{p+1}(u)\right)^{j-j'} \\
&&+ \left( \sum_{r=3}^p \frac{i^r}{r!} \kappa^*[u^{\otimes r}] + R_{p+1}(u) \right)^{p+1} \frac{1}{p!} \int_0^1(1-t)^p \exp\left( t \sum_{r=3}^p \frac{i^r}{r!} \kappa^*[u^{\otimes r}] + t R_{p+1}(u)  \right) dt \Bigg\}
\end{eqnarray*}
and
\begin{eqnarray*}
R_{p+1}(u)
= \frac{i^{p+1}}{p!} \int_0^1 (1-s)^p \kappa^*[u^{\otimes p+1}](su'S_T^*) ds - \frac{1}{2}\left( \kappa^*[u^{\otimes 2}] + \chi_{T,2}(u) \right).
\end{eqnarray*}

\

First, we consider $\left|\mathbb{D}^{\idx}\left( \hat{\Psi}^*_{T,p}(u) - \hat{\Psi}_{T,p,D}(u) \right)\right|$. From the definition, we have
\begin{eqnarray*}
&&\left|\mathbb{D}^{\idx}\left( \hat{\Psi}^*_{T,p}(u) - \hat{\Psi}_{T,p,D}(u) \right)\right| \\
\ &\lesssim& \ \Bigg|\mathbb{D}^{\idx} \Bigg[ \Big( e^{ \chi_{T,2}(u) } - e^{ - \frac{1}{2} u' \Sigma_{T,D}u } \Big) \Bigg\{ 1 + \sum_{j = 1}^p \sum_{r_1, \dots, r_j = 1}^{p-2} 1_{\{r_1 + \dots + r_j \le p-2\}} \kappa^*[u^{\otimes r_1+2}] \cdots \kappa^*[u^{\otimes r_j+2}] \Bigg\} \Bigg] \Bigg| \\
&&+ \ \Bigg|\mathbb{D}^{\idx} \Bigg[  e^{ - \frac{1}{2} u' \Sigma_{T,D}u }\sum_{j = 1}^p \sum_{r_1, \dots, r_j = 1}^{p-2} 1_{\{r_1 + \dots + r_j \le p-2\}} \Big( \kappa^*[u^{\otimes r_1+2}] \cdots \kappa^*[u^{\otimes r_j+2}] - \kappa[u^{\otimes r_1+2}] \cdots \kappa[u^{\otimes r_j+2}] \Big) \Bigg] \Bigg|.
\end{eqnarray*}
In the following, we assume that $u$ satisfies $|u| \le T^{\delta}$. We have
\begin{eqnarray}
\label{exp deff}
\Big| e^{ \chi_{T,2}(u) } - e^{ - \frac{1}{2} u' \Sigma_{T,D}u } \Big|
\ \le \ \Big|  e^{ \frac{1}{2}T^{-D}|u|^2 } - 1 \Big|
\ \lesssim \ T^{-D + 2\delta}.
\end{eqnarray}
For the first term, by applying Proposition \ref{cumest_1}, Corollary \ref{cumest_3} and (\ref{exp deff}), we get
\begin{eqnarray*}
&&\Bigg|\mathbb{D}^{\idx} \Bigg[ \Big( e^{ \chi_{T,2}(u) } - e^{ - \frac{1}{2} u' \Sigma_{T,D}u } \Big) \Bigg\{ 1 + \sum_{j = 1}^p \sum_{r_1, \dots, r_j = 1}^{p-2} 1_{\{r_1 + \dots + r_j \le p-2\}} \kappa^*[u^{\otimes r_1+2}] \cdots \kappa^*[u^{\otimes r_j+2}] \Bigg\} \Bigg] \Bigg| \nonumber \\
&\lesssim& \left(1 + \big|Var[S_T]\big| + T^{-D} \right)^{\sharp \idx} T^{-D + \delta(2 + \sharp \idx)}   \sum_{j = 1}^p \sum_{r_1, \dots, r_j = 1}^{p-2} 1_{\{r_1 + \dots + r_j \le p-2\}}  T^{ (r_1 + \dots + r_j)\left( -\frac{1}{2} + \ep + \delta\right) + j(\ep + 2\delta )} \nonumber \\
&\lesssim& T^{-D  + \ep (p + \bar{L}) + \delta(2p + 2 + \bar{L})}.
\end{eqnarray*}
By Proposition \ref{cumest_2} and Corollary \ref{cumest_3}, the second term is estimated as below;
\begin{eqnarray*}
&& \Bigg|\mathbb{D}^{\idx} \Bigg[  e^{ - \frac{1}{2} u' \Sigma_{T,D}u } \sum_{j = 1}^p \sum_{r_1, \dots, r_j = 1}^{p-2} 1_{\{r_1 + \dots + r_j \le p-2\}} \Big( \kappa^*[u^{\otimes r_1+2}] \cdots \kappa^*[u^{\otimes r_j+2}] - \kappa[u^{\otimes r_1+2}] \cdots \kappa[u^{\otimes r_j+2}] \Big) \Bigg] \Bigg| \\
&\lesssim&  \Bigg|\mathbb{D}^{\idx} \Bigg[  e^{ - \frac{1}{2} u' \Sigma_{T,D}u } \sum_{j = 1}^p \sum_{r_1, \dots, r_j = 1}^{p-2} 1_{\{r_1 + \dots + r_j \le p-2\}}\\
&& \times \sum_{k=1}^j \kappa^*[u^{\otimes r_1+2}] \cdots \kappa^*[u^{\otimes r_{k-1}+2}] \Big( \kappa^*[u^{\otimes r_k+2}] -  \kappa[u^{\otimes r_k+2}] \Big) \kappa[u^{\otimes r_{k+1}+2}] \cdots \kappa[u^{\otimes r_j+2}] \Bigg] \Bigg|\\
&\lesssim&  \left(1 + \big|Var[S_T]\big| + T^{-D} \right)^{\sharp \idx} T^{\delta\sharp \idx} T^{- L_3 \beta +  \ep p + \delta(3p-2)} 
\ \lesssim \ T^{-L_3\beta +  \ep (p + \bar{L}) + \delta(3p - 2 + \bar{L}) }.
\end{eqnarray*}
Therefore, for any $\ep$, $\delta$ and $\delta_0$, we can choose sufficiently large $D$ and $L_3$ such that 
\begin{eqnarray*}
\left|\mathbb{D}^{\idx}\left( \hat{\Psi}^*_{T,p}(u) - \hat{\Psi}_{T,p,D}(u) \right)\right| \ \lesssim \ T^{-\frac{p-2}{2} - \delta_0}.
\end{eqnarray*}

Finally, we have to show that $\left|\mathbb{D}^{\idx}\left( H_T(u) - \hat{\Psi}^*_{T,p}(u) \right)\right| = \left|\mathbb{D}^{\idx} R_{p+1}^*(u) \right| \lesssim  T^{-\frac{p-2}{2} - \delta_0}$. However, it follows by the same method as the proof of Lemma 7 in \cite{Yoshidaparmix}. In particular, we can choose $\delta_0$ and $\delta$ with $\delta_0 > d\delta$ in this proof. 
  \end{proof}


Referring to \cite{GotzeHipp1978}, we will prove Theorem \ref{Main Thm} by using the smoothness of a function $f$. Let $S'_T = T^{-\frac{1}{2}}\sum_{i = 0}^{N(T)} \phi_T(Z_i)$ and $e_T = T^{-\frac{1}{2}}\sum_{i = 0}^{N(T)} E[\phi_T(Z_i)]$. Note that
\begin{eqnarray*}
\left| E[\phi_T(Z_i)] \right|
= \left| E[\phi_T(Z_i) - Z_i]\right|
\le E[ |Z_i| 1_{\{|Z_i| > T^{\beta}\}}]
\le T^{-n\beta}E[|Z_i|^{n+1}]
\end{eqnarray*}
 for any $i = 0, \dots, N(T)$ and $n \in \naturels$. Therefore, [A2] yields 
\begin{eqnarray}
\label{e_T}
|e_T| \ \lesssim \ T^{-L_4}
\end{eqnarray}
for an arbitrarily large constant $L_4 > 0$.

\begin{proof}[{\bf Proof of Theorem \ref{Main Thm} }]
Let $\Gamma= \lceil \frac{p-1}{2\delta} \rceil$ for some $\delta \in (0,1)$ and $f \in \mathscr{E}(\Gamma, L_1, L_2)$. Since 
\begin{eqnarray*}
\left| E\left[ f( S_T )\right] - \int_{\reels^d}f(z) p_{T,p,D}(z)dz \right|
&\le& \Big| E\left[f(S_T)\right] - E\left[f(S'_T)\right] \Big| 
+ \left| E\left[f(S'_T)\right] - \int_{\reels^d}f(z + e_T) p_{T,p,D}(z)dz \right| \\
&&+ \left| \int_{\reels^d}f(z + e_T) p_{T,p,D}(z)dz - \int_{\reels^d}f(z) p_{T,p,D}(z)dz \right|
=: \Delta_1 + \Delta_2 + \Delta_3,
\end{eqnarray*}
we only have to estimate $\Delta_1, \Delta_2$ and $\Delta_3$.

First, we consider $\Delta_1$. Let $\eta > 0$. We can assume that $L_1$ is even by retaking $L_1$ and $L_2$ that satisfy $\sup_{|\alpha| \le \Gamma}|\partial^\alpha f(x)| \le L_2(1+|x|)^{L_1}$ for every $x \in \reels^d$. We set $A = \{ |S_T| \le T^{\eta}\}$ and $B = \{ |S'_T| \le T^{\eta}\}$. Similarly to Lemma 3.3 in \cite{GotzeHipp}, we get 

\begin{eqnarray*}
\Delta_1 
&\lesssim& T^{\eta L_1}P\big[S_T \neq S'_T\big] + E\big[|S_T|^{L_1} 1_{A^c}\big]  + E\big[|S'_T|^{L_1} 1_{B^c}\big] \\
&\lesssim& T^{\eta L_1}P\big[S_T \neq S'_T\big] + E\big[|S'_T|^{L_1} 1_{B^c}\big] + \Big| E\big[|S_T|^{L_1}\big] - E\big[|S'_T|^{L_1}]\big]\Big|.
\end{eqnarray*}
Let $L > 0$ be an arbitrary large constant. Since $P\big[S_T \neq S'_T\big] \le \sum_{i=0}^{N(T)}P\big[|Z_i| > T^{\beta}] \lesssim T^{-n\beta+\frac{1}{2}}$ for any $n \in \naturels$, $T^{\eta L_1}P\big[S_T \neq S'_T\big] \lesssim T^{-L}$ holds. 
Since $|S'_T| \le |S^*_T| + |e_T|$, we have
\begin{eqnarray}
\label{S'_T}
E\big[|S'_T|^{L_1} 1_{B^c}\big]  \ \lesssim \ E \left[ |S^*_T|^{L_1} 1_{ \{ |S^*_T| > \frac{ T^{\eta} }{2} \} } \right] + |e_T|^{L_1}.
\end{eqnarray}
The moment has the representation by cumulants; for any even $n \in \naturels$
\begin{eqnarray*}
E \left[ |S^*_T|^n \right] 
=\sum_{|\alpha| = n} \sum_{k = 1}^n \sum_{\substack{ \alpha_1, \dots, \alpha_k;\\ \alpha_1 + \dots + \alpha_k = \alpha}} 
\frac{\alpha!}{k!\alpha_1! \cdots \alpha_k!}
\prod_{m = 1}^k \kappa_{\alpha_m}\left[S^*_T\right],
\end{eqnarray*}
where $\kappa_{\alpha_m}\left[S^*_T\right] = (-i)^{|\alpha_m|}\partial^{\alpha_m}\log E[e^{iu'S^*_T}]|_{u=0}$ and $\alpha! = \alpha^{1;}! \cdots \alpha^{d;}!$ for $\alpha = (\alpha^{1;}, \dots, \alpha^{d;}) \in \mathbb{Z}_+^d$.
From Proposition \ref{cumest_1} and the representation of moments by cumulants, $E \left[ |S^*_T|^n \right] \ \lesssim \ T^{\ep n}$ for any even $n \in \naturels$ and $\ep > 0$. Therefore, 
\begin{eqnarray*}
 E \left[ |S^*_T|^{L_1} 1_{ \{ |S^*_T| > \frac{ T^{\eta} }{2} \} } \right]
\ \le \ \left(\frac{T^{\eta}}{2}\right)^{-n} E \left[ |S^*_T|^{L_1+n} \right]
\ \lesssim \ T^{-\eta n + \ep(L_1 + n)}.
\end{eqnarray*}
(\ref{e_T}), (\ref{S'_T}) and the above inequality lead $E\big[|S'_T|^{L_1} 1_{B^c}\big]  \lesssim T^{-L}$ by choosing $\eta > \ep$ and sufficiently large $n$.
Since we took $L_1$ as an even number, the representation of moments by cumulants leads
\begin{eqnarray*}
\Big| E\big[|S_T|^{L_1}\big] - E\big[|S'_T|^{L_1}]\big]\Big|
\le \sum_{|\alpha| = L_1} \sum_{k = 1}^{L_1} \sum_{\substack{ \alpha_1, \dots, \alpha_k;\\ \alpha_1 + \dots + \alpha_k = \alpha}} 
\frac{\alpha!}{k!\alpha_1! \cdots \alpha_k!}
\left|\prod_{m = 1}^k \kappa_{\alpha_m}\left[S_T\right] - \prod_{m = 1}^k \kappa_{\alpha_m}\left[S'_T\right] \right|.
\end{eqnarray*}
From the definition of the cumulant, $\kappa_{\alpha_m}\left[S'_T\right] = \kappa_{\alpha_m}\left[S^*_T\right]$ for $|\alpha_m| \ge 2$ and $\kappa_{\alpha_m}\left[S'_T\right] =  e_T$ for $|\alpha_m| = 1$. Thus, Proposition \ref{cumest_1},  Proposition \ref{cumest_2}, Corollary \ref{cumest_3} and (\ref{e_T}) yeild
\begin{eqnarray*}
\left|\prod_{m = 1}^k \kappa_{\alpha_m}\left[S_T\right] - \prod_{m = 1}^k \kappa_{\alpha_m}\left[S'_T\right] \right|
&\lesssim& \left|\prod_{m = 1}^k \kappa_{\alpha_m}\left[S_T\right] - \prod_{m = 1}^k \kappa_{\alpha_m}\left[S^*_T\right] \right| + T^{-L}\\
&=&\left|\sum_{k'=1}^k \Bigg( \prod_{m = 1}^{k'-1} \kappa_{\alpha_m}\left[S_T\right] \Bigg) \left( \kappa_{\alpha_{k'}}\left[S_T\right] - \kappa_{\alpha_{k'}}\left[S^*_T\right]\right) \Bigg( \prod_{m = k'+1}^k \kappa_{\alpha_m}\left[S^*_T\right] \Bigg) \right| + T^{-L}\\
&\lesssim& T^{-\frac{1}{2}(|\alpha_1| + \cdots + |\alpha_k| - 2k) + \ep(|\alpha_1| + \cdots + |\alpha_k| - k) - L_3\beta} + T^{-L}
\lesssim T^{-L}.
\end{eqnarray*}
Therefore, we get $\Delta_1 \lesssim T^{-L}$. Since $L$ is an arbitrary constant, we get $\Delta_1 \ \lesssim \ T^{-\frac{p-2}{2}}$.

Second, we estimate $\Delta_2$. Write $h_{e_T}(z) = h(z + e_T)$ for any function $h$ on $\reels^d$. Denote the distribution of $S^*_T$ as $dQ^*_T$. Then, we can rewrite
\begin{eqnarray*}
\Delta_2
=\left| \int_{\reels^d}f_{e_T}(z)  d\big(Q^*_T - \Psi_{T,p,D} \big)(z)\right| 
\end{eqnarray*}
For $z, u \in \reels^d$, the Taylor's theorem yields
\begin{eqnarray*}
f_{e_T}(z)
= \sum_{\alpha; |\alpha|\le\Gamma}\frac{\partial^{\alpha}f_{e_T}(z + u)}{\alpha!}(-u)^{\alpha} + g_T^{-1}(z, u)
\end{eqnarray*}
where 
\begin{eqnarray*}
 g_T^{-1}(z, u)
 = \sum_{\alpha; |\alpha| = \Gamma} (-u)^{\alpha} \frac{\Gamma}{\alpha!} \int_0^1 \nu^{\Gamma } \Big( \partial^{\alpha}f_{e_T}(z + \nu u) - \partial^{\alpha}f_{e_T}(z + u) \Big)d\nu.
\end{eqnarray*}
Let $\mathcal{K}$ be a probability measure on $\reels^d$ such that $\int_{\reels^d}|z|^{\bar{L}}d\mathcal{K}(z) < \infty$ for sufficiently large $\bar{L} > 0$ and its Fourier transformation $\hat{\mathcal{K}}(u)$ satisfies $\hat{\mathcal{K}}(u)=0$ if $|u| > 1$. (Such $\mathcal{K}$ exists. See Theorem 10.1 in \cite{Bhattacharya}.) Moreover, let $d\mathcal{K}_T(u) = d\mathcal{K}(T^{-\delta}u)$ and $d\mathcal{K}_{T, \alpha}(u) = u^{\alpha}d\mathcal{K}_T(u)$. We have 
\begin{eqnarray}
\label{partation}
&&\int_{\reels^d}f_{e_T}(z)  d\big(Q^*_T - \Psi_{T,p,D} \big)(z) \nonumber\\
&=& \sum_{\alpha; |\alpha| \le \Gamma} \int_{\reels^d \times \reels^d} \frac{\partial^{\alpha}f_{e_T}(z + u)}{\alpha!}(-u)^{\alpha}   d\big(Q^*_T - \Psi_{T,p,D} \big)(z) d\mathcal{K}_T(u) 
+  \int_{\reels^d \times \reels^d}  g_T^{-1}(z, u)  d\big(Q^*_T - \Psi_{T,p,D} \big)(z) d\mathcal{K}_T(u) \nonumber\\
&=& \sum_{\alpha; |\alpha| \le \Gamma} \frac{(-1)^{\alpha}}{\alpha!} \int_{\reels^d} \partial^{\alpha}f_{e_T}(x)  d\left(\mathcal{K}_{T,\alpha}*\big(Q^*_T - \Psi_{T,p,D} \big)\right)(x) 
+  \int_{\reels^d \times \reels^d}  g_T^{-1}(z, u)   d\big(Q^*_T - \Psi_{T,p,D} \big)(z) d\mathcal{K}_T(u).
\end{eqnarray}
We know that $\partial^{\alpha}f_{e_T}(x) = \partial^{\alpha}f(x + e_T)$ and $|e_T|$ is bounded in $T$. Then, Lemma 11.6 \cite{Bhattacharya} and well-known properties of Fourier transform lead
\begin{eqnarray*}
&&\left| \sum_{\alpha; |\alpha| \le \Gamma} \frac{(-1)^{\alpha}}{\alpha!} \int_{\reels^d} \partial^{\alpha}f_{e_T}(x)  d\left(\mathcal{K}_{T,\alpha}*\big(Q^*_T - \Psi_{T,p,D} \big)\right)(x) \right| \\
&\le& \left(\sup_{\substack{|\alpha| \le \Gamma\\x \in \reels^d}}\frac{\left| \partial^\alpha f(x) \right|}{1 + |x|^{L_1}}\right) \left| \sum_{\alpha; |\alpha| \le \Gamma} \frac{(-1)^{\alpha}}{\alpha!} \int_{\reels^d} 1 + |x + e_T|^{L_1}  d\left(\mathcal{K}_{T,\alpha}*\big(Q^*_T - \Psi_{T,p,D} \big)\right)(x) \right|\\
&\lesssim& \sum_{\alpha; |\alpha| \le \Gamma}  \max_{|\beta| \le L_1 + d + 1} \int_{\reels^d} \left|\partial^{\beta}\left(\hat{\mathcal{K}}_{T,\alpha}(u)\big(H_T(u) - \hat{\Psi}_{T,p,D}(u) \big)\right)\right| du.
\end{eqnarray*}
Since $\hat{\mathcal{K}}_{T,\alpha}(u) = i^{-|\alpha|}\partial^{\alpha}\hat{\mathcal{K}}_{T}(u)$, we have $\mathop{\mathrm{supp}}\nolimits  \hat{\mathcal{K}}_{T,\alpha}(u) \subset \{ |u| < T^{\delta}\}$. Moreover, $\big|\partial^{\beta} \hat{\mathcal{K}}_{T,\alpha}(u)\big| \le \int_{\reels^d}|z|^{\alpha + \beta} d\mathcal{K}_T(z) \lesssim T^{-\delta(|\alpha| + |\beta|)}$ holds. Thus, from Proposition \ref{cf}, we can choose $\delta \in (0,1)$ and $\delta_0 > d\delta$ such that 
\begin{eqnarray*}
&&\max_{|\beta| \le L_1 + d + 1} \int_{\reels^d} \left|\partial^{\beta}\left(\hat{\mathcal{K}}_{T,\alpha}(u)\big(H_T(u) - \hat{\Psi}_{T,p,D}(u) \big)\right)\right| du\\
&\lesssim& \max_{|\beta_1|, |\beta_2| \le L_1 + d + 1} \int_{\{|u| < T^{\delta}\}} \left|\partial^{\beta_1}\hat{\mathcal{K}}_{T,\alpha}(u)\right| \left| \partial^{\beta_2}\big(H_T(u) - \hat{\Psi}_{T,p,D}(u) \big)\right| du\\
&\lesssim&  T^{-\frac{p-2}{2} - \delta_0 + d\delta} \ \lesssim \ T^{-\frac{p-2}{2}}.
\end{eqnarray*}
It means that
\begin{eqnarray}
\label{moment est}
\left| \sum_{\alpha; |\alpha| \le \Gamma} \frac{(-1)^{\alpha}}{\alpha!} \int_{\reels^d} \partial^{\alpha}f_{e_T}(x)  d\left(\mathcal{K}_{T,\alpha}*\big(Q^*_T - \Psi_{T,p,D} \big)\right)(x) \right| \ \lesssim \ T^{-\frac{p-2}{2}}.
\end{eqnarray}
On the other hand, from Proposition \ref{cumest_1}, $|\Sigma_{T,D}| \le T^{\ep}$ holds for any $\ep>0$. Therefore, $ \int_{\reels^d } |z|^{L_1} p_{T,p,D}(z)dz \le T^{\frac{L_5 \ep}{2}}$ for some constant $L_5 > 0$ which depends on $p$. Thus, by taking sufficiently small $\ep$,
\begin{eqnarray*}
&&\left| \int_{\reels^d \times \reels^d}  g_T^{-1}(z, u)   d\big(Q^*_T - \Psi_{T,p,D} \big)(z) d\mathcal{K}_T(u) \right|\\
&=& \left| \int_{\reels^d \times \reels^d} \sum_{\alpha; |\alpha| = \Gamma} (-u)^{\alpha} \frac{\Gamma}{\alpha!} \int_0^1 \nu^{\Gamma } \Big( \partial^{\alpha}f_{e_T}(z + \nu u) - \partial^{\alpha}f_{e_T}(z + u) \Big) d\nu d\big(Q^*_T - \Psi_{T,p,D} \big)(z) d\mathcal{K}_T(u) \right|\\
&\lesssim&  (T^{-\delta})^{\Gamma} \left(\sup_{\substack{|\alpha| \le \Gamma \\ x \in \reels^d}}\frac{\left| \partial^\alpha f(x) \right|}{1 + |x|^{L_1}}\right) \bigg| \int_{\reels^d \times \reels^d} \int_0^1u^{\alpha} \nu^{\Gamma } \Big( 1 + |z + T^{-\delta}u + e_T|^{L_1} \Big) d\nu d\big(Q^*_T - \Psi_{T,p,D} \big)(z) d\mathcal{K}(u) \bigg|\\
&\lesssim& T^{-\frac{p-1}{2}} \bigg| \int_{\reels^d } 1 + |z|^{L_1} d\big(Q^*_T - \Psi_{T,p,D} \big)(z) \bigg|
\ \lesssim \ T^{-\frac{p-1}{2}} \left( E\left[ |S^*_T|^{L_1}\right] +  \int_{\reels^d } |z|^{L_1} p_{T,p,D}(z)dz  \right)
\ \lesssim \ T^{-\frac{p-2}{2}}.
\end{eqnarray*}
In conclusion, we get $\Delta_2 \ \lesssim \ T^{-\frac{p-2}{2}}$ from (\ref{partation}), (\ref{moment est}) and the above inequality.

Finally, we consider $\Delta_3$. With the help of the mean value theorem, we can deduce as below; for some $\tau \in (0, 1)$,
\begin{eqnarray*}
\Delta_3
\le \int_{\reels^d} \left| f(z + e_T) - f(z)  \right| p_{T,p,D}(z)dz
&\le& \sum_{|\alpha| = 1} \int_{\reels^d} \left| \partial^{\alpha} f(z + \tau e_T) \right| |e_T| p_{T,p,D}(z)dz \\
&\lesssim& |e_T|  \left(\sup_{\substack{|\alpha| \le \Gamma \\ x \in \reels^d}}\frac{\left| \partial^\alpha f(x) \right|}{1 + |x|^{L_1}}\right) \int_{\reels^d} (1 + |z + e_T|^{L_1}) p_{T,p,D}(z)dz\\
&\lesssim& T^{-L_4 + \frac{L_5\ep}{2} } \ \lesssim \ T^{-\frac{p-2}{2}}.
\end{eqnarray*}

Therefore, we get the conclusion.
  \end{proof}


\subsection{ Proofs of Subsection 2.2}
Before prove Proposition \ref{tildeS_T}, we consider an asymptotic expansion of $\tilde{Z}_T$. We assume that $Z_T$ satisfies the conditions [A1] and [A2].  Then, from Theorem \ref{Main Thm}, for any $L_1, L_2>0$, there exist  $D>0$ and  $\Gamma \in \naturels$  such that for any $f \in \mathscr{E}(\Gamma, L_1, L_2)$,
\begin{eqnarray*}
\left| E\left[ f\left(\frac{Z_T}{\sqrt{T}}\right)\right] - \int_{\reels^d}f(z) p_{T,3,D}(z)dz \right| = o\left(T^{-1/2}\right),
\end{eqnarray*}
where
\begin{eqnarray*}
p_{T,3,D}(z) = \phi(z;\Sigma_{T,D}) + \frac{1}{6\sqrt{T}}\kappa^{a_1a_2a_3;}_T h_{a_1a_2a_3}(z;\Sigma_{T,D})\phi(z;\Sigma_{T,D}),
\end{eqnarray*}
for the modified cumulant $\kappa^{a_1a_2a_3;}_T$ and the Hermite polynomial $h_{a_1a_2a_3}(z;\Sigma_{T,D})$ defined in (\ref{modi cumulant}) and (\ref{Hermite}) respectively.
From the concrete form of $C_T$ and $M_{T, D}$, they are clearly non-degenerate and $f \circ M_{T, D} \circ C_T \in \mathscr{E}(\Gamma, L_1, L_2)$ holds for any $f \in \mathscr{E}(\Gamma, L_1, L_2)$. Owing to the variable transformation and the multi-linearity of the cumulant, the following inequality is immediately obtained. (See Proposition 7.1 in \cite{sakyos} for the proof details.)
\begin{lemma*}
\label{tildeZ exp}
Let $L_1, L_2>0$. Suppose that the conditions \textnormal{[A1]-[A3]} and  \textnormal{[B0]} hold. Then, there exist  $D>0$ and  $\Gamma \in \naturels$ such that for any $f \in \mathscr{E}(\Gamma,, L_1, L_2)$,
\begin{eqnarray*}
\left| E\left[ f\big(\tilde{Z}_T\big)\right] - \int_{\reels^d}f(z) \tilde{p}_{T,3,D}(z)dz \right| = o\left(T^{-1/2}\right),
\end{eqnarray*}
where $\tilde{\lambda}^{a_1a_2a_3;}_T$ is the $(a_1, a_2, a_3)$-cumulant of $\tilde{Z}_T$, $\tilde{\kappa}^{a_1a_2a_3;}_T = T^{1/2}\tilde{\lambda}^{a_1a_2a_3;}_T$ and
\begin{eqnarray*}
\tilde{p}_{T,3,D}(z) = \phi(z;\tilde{\Sigma}_{T,D}) + \frac{1}{6\sqrt{T}}\tilde{\kappa}^{a_1a_2a_3;}_T h_{a_1a_2a_3}(z;\tilde{\Sigma}_{T,D})\phi(z;\tilde{\Sigma}_{T,D}).
\end{eqnarray*}
\end{lemma*}


\begin{proof}[{\bf Proof of Proposition \ref{tildeS_T} }]
It is proved in the same way as the proof of Theorem 5.1 in \cite{sakyos} for
\begin{eqnarray*}
\tilde{q}_{T,3,D}(z^{(1)}) 
&=& \int_{\reels^{p^2}}\phi(z;\tilde{\Sigma}_{T,D})dz^{(2)} + \frac{1}{\sqrt{T}}\Bigg\{ \int_{\reels^{p^2}}\frac{1}{6}\tilde{\kappa}^{a_1a_2a_3;}_T h_{a_1a_2a_3}(z;\tilde{\Sigma}_{T,D})\phi(z;\tilde{\Sigma}_{T,D})dz^{(2)} \\
&&- \sum_{a=1,\dots,p} \frac{\partial}{\partial z^{a;}} \int_{\reels^{p^2}}\tilde{Q}_1^{a;}(z)\phi(z;\tilde{\Sigma}_{T,D})dz^{(2)} \Bigg\}.
\end{eqnarray*}
by using the Bhattacharya-Ghosh map and transforming asymptotic expansion in Lemma \ref{tildeZ exp}.
Thus, We only need to consider the form of $\tilde{q}_{T,3,D}$. Due to the orthogonalization, it immediately follows that
\begin{eqnarray*}
 h_A(z;\tilde{\Sigma}_{T,D})\phi(z;\tilde{\Sigma}_{T,D}) = h_{A^{(1)}}(z^{(1)}; \tilde{\Sigma}_{T,D}^{(1,1)})\phi(z^{(1)};\tilde{\Sigma}_{T,D}^{(1,1)})h_{A^{(2)}}(z^{(2)};\tilde{\Sigma}_{T,D}^{(2, 2)})\phi(z^{(2)};\tilde{\Sigma}_{T,D}^{(2, 2)})
\end{eqnarray*}
for $A^{(1)} = A\cap\{1,\dots,p\}$ and  $A^{(2)} = A\cap\{p+1,\dots,p+p^2\}$. We decompose the polynomial $\tilde{Q}_1^{a;}(z)$ by the Hermite polynomials. Let
\begin{eqnarray*}
\tilde{Q}_1^{a;}(z) = \pi^{a;}_{1,\phi}(z^{(1)}) + \pi^{a;}_{1,a_1}(z^{(1)})h^{a_1;}(z^{(2)};\tilde{\Sigma}_{T,D}^{(2, 2)}),
\end{eqnarray*}
where $h^{a;}(x;\sigma) = \sigma^{aa_1;}h_{a_1}(x;\sigma)$ for $\sigma = (\sigma^{ab;})_{a,b=1, \dots, p}$ and $\pi^{a;}_{1,\phi}(z^{(1)}), \pi^{a;}_{1,a_1}(z^{(1)})$ are polynomials for $z^{(1)}$. It is well known that the orthogonality of the Hermite polynomial
\begin{eqnarray*}
\int h^A(z^{(2)};\tilde{\Sigma}_{T,D}^{(2, 2)}) h_B(z^{(2)};\tilde{\Sigma}_{T,D}^{(2, 2)}) \phi(z^{(2)};\tilde{\Sigma}_{T,D}^{(2, 2)}) dz^{(2)} 
= \begin{cases}
    A! & \text{if $A=B$} \\
    0 & \text{otherwise.}
  \end{cases}
\end{eqnarray*}
This orthogonality gives the following representation of $\pi^{a;}_{1,\phi}(z^{(1)})$,
\begin{eqnarray*}
\pi^{a;}_{1,\phi}(z^{(1)}) 
= \int_{\reels^{p^2}} \tilde{Q}_1^{a;}(z) \phi(z^{(2)};\tilde{\Sigma}_{T,D}^{(2, 2)}) dz^{(2)}
= \tilde{\mu}^{a;}_{a_1a_2}z^{a_1;}z^{a_2;}
= \tilde{\mu}^{a;}_{a_1a_2}\tilde{g}^{a_1b_1;}\tilde{g}^{a_2b_2;}h_{b_1b_2}(z^{(1)};\tilde{\Sigma}_{T,D}^{(1,1)}) + \tilde{\mu}^{a;}_{a_1a_2}\tilde{g}^{a_1a_2;},
\end{eqnarray*}
where we used $h_{b_1b_2}(z^{(1)};\tilde{\Sigma}_{T,D}^{(1,1)}) = \tilde{g}_{b_1a_1}\tilde{g}_{b_2a_2}z^{a_1;}z^{a_2;} - \tilde{g}_{b_1b_2}$. Therefore, we get
\begin{eqnarray*}
&&\sum_{a=1,\dots,p} \frac{\partial}{\partial z^{a;}} \int_{\reels^{p^2}}\tilde{Q}_1^{a;}(z)\phi(z;\tilde{\Sigma}_{T,D})dz^{(2)}\\
&=& \sum_{a=1,\dots,p} \frac{\partial}{\partial z^{a;}} \int_{\reels^{p^2}} \left(\pi^{a;}_{1,\phi}(z^{(1)}) + \pi^{a;}_{1,a_1}(z^{(1)})h^{a_1;}(z^{(2)};\tilde{\Sigma}_{T,D}^{(2, 2)})\right)\phi(z^{(1)};\tilde{\Sigma}_{T,D}^{(1,1)})\phi(z^{(2)};\tilde{\Sigma}_{T,D}^{(2, 2)})dz^{(2)}\\
&=& -\left(\tilde{\mu}^{a;}_{a_1a_2}\tilde{g}^{a_1b_1;}\tilde{g}^{a_2b_2;}h_{b_1b_2a}(z^{(1)};\tilde{\Sigma}_{T,D}^{(1,1)}) + \tilde{\mu}^{a;}_{a_1a_2}\tilde{g}^{a_1a_2;}h_a(z^{(1)};\tilde{\Sigma}_{T,D}^{(1,1)})\right)\phi(z^{(1)};\tilde{\Sigma}_{T,D}^{(1,1)}).
\end{eqnarray*}
Thus, we can get the desired form of $\tilde{q}_{T,3,D}$.
 \end{proof}

\begin{proof}[{\bf Proof of Proposition \ref{evaluateR_2} }]
This proof is the almost same as Theorem 6.2. in \cite{sakyos}. Let $\gamma' \in \left( \frac{2}{3}, \gamma - \frac{L}{q_2} \right)$ and $\gamma'' \in  \left( \frac{L}{q_3}, 3\gamma - 2 \right)$. We set
\begin{eqnarray*}
\mathscr{X}_{T,0} = \Bigg\{ \omega \in \Omega &\Bigg|&
\inf_{\substack{T>0, |x|=1\\ \theta_1, \theta_2 \in \tilde{\Theta}}} \left| x' \int^1_0  \nu_{a b}\left(\theta_1 + s(\theta_2 - \theta_1)\right)ds\right| > C', \
\left| T^{-\frac{2-\gamma}{2}}l_{a_1} \right| < C', \ 
\sup_{\theta \in \Theta}\left| T^{-1}l_{a_1a_2}(\theta) - \nu_{a_1a_2}(\theta) \right| < \frac{C'}{2p^2}  \Bigg\}
\end{eqnarray*}
for some constant $C'>0$, and
\begin{eqnarray*}
\mathscr{X}_{T,1} = \left\{ \omega \in \Omega \ \left|  \
\left| T^{-1}l_{a_1a_2} - \nu_{a_1a_2} \right| < T^{-\frac{\gamma'}{2}}, \
\left| T^{-1}l_{a_1a_2a_3} - \nu_{a_1a_2a_3} \right| < T^{-\frac{\gamma'}{2}}, \ 
\sup_{\theta \in \Theta}\left| T^{-1}l_{a_1a_2a_3a_4}(\theta)\right| < T^{\frac{\gamma''}{2}} \right.\right\}.
\end{eqnarray*}
For appropriate $C'>0$ and sufficiently large $T$, it is known that there exists a unique $\mle \in \tilde{\Theta}$ such that $\partial_{\theta}l_T(\mle) = 0$ and $|\mle - \theta_0| < T^{-\frac{\gamma}{2}}$ on the set $\mathscr{X}_{T,0}$. In particular, $\mathscr{X}_{T,0} \subset \Omega_T$ holds for large $T$. Moreover, it is also proved that $P[(\mathscr{X}_{T,0})^c] \ \lesssim \ T^{-\frac{L}{2}}$.
Here, we used the conditions [B0], [B1], [B2] and [B3], see the proof of Theorem 6.1 in \cite{sakyos} for details. On the other hand, the conditions [B2] and [B4] lead
\begin{eqnarray*}
P[(\mathscr{X}_{T,1})^c] 
&\le& P\left[T^{\frac{\gamma}{2}}\left|T^{-1}l_{a_1a_2} - \nu_{a_1a_2} \right| \ge T^{-\frac{\gamma'}{2} + \frac{\gamma}{2}}\right] + P\left[T^{\frac{\gamma}{2}}\left|T^{-1}l_{a_1a_2a_3} - \nu_{a_1a_2a_3} \right| \ge T^{-\frac{\gamma'}{2} + \frac{\gamma}{2}}\right] \\
&&+ P\left[  \sup_{\theta \in \Theta}\left| T^{-1}l_{a_1a_2a_3a_4}(\theta)\right| \ge T^{\frac{\gamma''}{2}} \right] \\
&\lesssim& T^{-\frac{(\gamma - \gamma')q_2}{2}} + T^{-\frac{\gamma'' q_3}{2}} \ \lesssim \ T^{-\frac{L}{2}}.
\end{eqnarray*}
Since $g_T^{-1}$ converge to a non-singular matrix by the conditon [A3], we have $|Z^{a;}| \lesssim T^{-\frac{\gamma -1}{2}}$, $|Z^{a;}_{a_1}| \lesssim T^{-\frac{\gamma' -1}{2}}$, $|Z^{a;}_{a_1a_2}| \lesssim T^{-\frac{\gamma' -1}{2}}$ and $|\bar{\theta}| \lesssim T^{-\frac{\gamma -1}{2}}$ on $\mathscr{X}_{T,0} \cap \mathscr{X}_{T,1}$. Moreover, the conditon [B4] guarantees
\begin{eqnarray}
\label{nu 3}
\big| \nu^{a;}_{a_1 a_2} \big| \le \big| g^{ab;} \big| E\left[\left|T^{-1}l_{b,a_1 a_2}\right|\right] < \infty \quad \text{uniformly with respect to $T$}.
\end{eqnarray}
Hereafter, we consider the following inequalities on $\mathscr{X}_{T,0} \cap \mathscr{X}_{T,1}$. Let $a \in \{1, \dots, p\}$. First, we get 
\begin{eqnarray*}
\left|T^{-\frac{1}{2}} \bar{R}^{a;}_2 \right| &=& \left|T^{-\frac{1}{2}} \left( \frac{1}{2}Z^{a;}_{a_1 a_2}\bar{\theta}^{a_1 a_2;} +\frac{1}{2}\left\{\int^1_0(1-u)^2g^{ab;}\left( \frac{1}{T}l_{b a_1 a_2 a_3}\left(\theta_0 + u(\mle-\theta_0)\right)\right)du\right\} \bar{\theta}^{a_1 a_2 a_3;} \right) \right| \\
&\lesssim& T^{-\frac{1}{2} -\frac{\gamma' -1}{2} - (\gamma -1)} + T^{-\frac{1}{2} +\frac{\gamma''}{2}  - \frac{3(\gamma -1)}{2}} 
\ \lesssim \ T^{-\frac{\ep}{2}}
\end{eqnarray*}
for some small constant $0< \ep < \min(2\gamma + \gamma' -2, 3\gamma - \gamma'' -2)$.  Similarly, we have
\begin{eqnarray*}
\left| \bar{R}^{a;}_1 \right| &=& \left| Z^{a;}_{a_1}\bar{\theta}^{a_1;} + \frac{1}{2}\nu^{a;}_{a_1 a_2}\bar{\theta}^{a_1 a_2;} + T^{-\frac{1}{2}}\bar{R}^{a;}_2 \right| 
\ \lesssim \ T^{ -\frac{\gamma' -1}{2} -\frac{\gamma -1}{2} } + T^{- (\gamma -1)} + T^{-\frac{\ep}{2}}
\ \lesssim \ T^{-\frac{\ep'}{2}}
\end{eqnarray*}
for a positive constant $0 < \ep' <  \min(\gamma + \gamma' -2, \ep)$. Finally we have
\begin{eqnarray*}
\left|T^{-\frac{1}{2}} \check{R}^{a;}_2 \right| &=& \left|T^{-\frac{1}{2}} \left( Z^{a;}_{a_1}\bar{R}^{a_1; }_1 + \bar{R}^{a;}_2 \right)  + T^{-1}\left(\frac{1}{2}\nu^{a;}_{a_1 a_2}\bar{R}^{a_1; }_1\bar{R}^{a_2; }_1 \right)  \right| 
\ \lesssim \ T^{-\frac{1}{2} -\frac{\gamma' -1}{2} -\frac{\ep'}{2}} + T^{-\frac{\ep}{2}} + T^{-1 -\ep'} 
\ \lesssim \ T^{-\frac{\ep'}{2}}.
\end{eqnarray*}
Therefore, we get the desired conclusion
\begin{eqnarray*}
P\left[\Omega_T \cap \left\{ T^{-1}|\check{R}_2^{a;}| \le CT^{-\frac{1 + \ep'}{2}}, \ a = 1, \dots, p \right\} \right] \ \ge \ P[\mathscr{X}_{T,0} \cap \mathscr{X}_{T,1} ] = 1 - o(T^{-\frac{L}{2}}).
\end{eqnarray*}
  \end{proof}


\begin{proof}[{\bf Proof of Theorem \ref{Main Thm2 New} }]
From Proposition \ref{tildeS_T}, we see that
\begin{eqnarray*}
&& \left| E\left[ f\big(\sqrt{T}( \mle - \theta_0 )\big)\right] - \int_{\reels^d}f(z^{(1)}) q_{T,3}(z^{(1)})dz^{(1)} \right| \\
&\lesssim& \left| E\bigg[ \right\{ f\big(\sqrt{T}( \mle - \theta_0 )\big) - f\big( \tilde{S}_T\big) \left\}1_{\Omega_T^c}\bigg] \right| + \left| E\left[ \left\{ f\big( \tilde{S}_T + T^{-1}\check{R}_2\big) - f\big( \tilde{S}_T\big) \right\} 1_{\Omega_T \cap \left\{T^{-1}|\check{R}_2^{a;}| > CT^{-\frac{1 + \ep'}{2}}\right\}}\right] \right| \\
&&+ \left| E\left[ \left\{ f\big( \tilde{S}_T + T^{-1}\check{R}_2 )\big) - f\big( \tilde{S}_T\big) \right\} 1_{\Omega_T \cap \left\{T^{-1}|\check{R}_2^{a;}| \le CT^{-\frac{1 + \ep'}{2}}\right\}} \right] \right| \\
&&+\left| \int_{\reels^d}f(z^{(1)}) \tilde{q}_{T,3,D}(z^{(1)})dz^{(1)} - \int_{\reels^d}f(z^{(1)}) q_{T,3}(z^{(1)})dz^{(1)} \right|
+T^{-\frac{1}{2}} \\
&=:& \Delta_1 + \Delta_2 + \Delta_3 + \Delta_4 + T^{-\frac{1}{2}}.
\end{eqnarray*}
From the definition of $\tilde{S}_T$ and the representation of (\ref{main form of MLE}),
\begin{eqnarray*}
\big\| \tilde{S}_T \big\|_{L^k(P)} 
\le \sum_{a=1,\dots, p} \bigg\| Z^{a;} +T^{-\frac{1}{2}}Z^{a;}_{a_1}Z^{a_1; } + \frac{1}{2}T^{-\frac{1}{2}} \nu^{a;}_{a_1 a_2}Z^{a_1; }Z^{a_2; } \bigg\|_{L^k(P)}.
\end{eqnarray*}
By Corollary \ref{cumest_3} and the representation of moments by cumulants, we have 
\begin{eqnarray*}
\big\| T^{-\frac{1}{2}}Z_T \big\|_{L^k(P)}
\ \lesssim \ T^{\ep k}
\end{eqnarray*}
for any $\ep > 0$ and $k>0$. From (\ref{nu 3}) and the above inequality, $\big\| \tilde{S}_T \big\|_{L^k(P)} \ \lesssim \ T^{\ep k}$ holds for any $\ep > 0$ and $k>0$. Thus, the condition [C1] yields $\big\| T^{-1}\check{R}_2 \big\|_{L^k(P)} \ \lesssim \ T^{\ep k}$ for any $\ep > 0$ and $k>0$. 

We evaluate $\Delta_1, \Delta_2, \Delta_3$ and $\Delta_4$. From Proposition \ref{evaluateR_2} and by choosing sufficiently small $\ep$, we get
\begin{eqnarray*}
\Delta_1
\ \lesssim \ E\left[\left(1 +  \big|\tilde{S}_T\big| + \big|\sqrt{T}( \mle - \theta_0 )\big| \right)^{L_1} 1_{\Omega_T^c}\right]
\ \lesssim \ T^{-\frac{1}{2}}.
\end{eqnarray*}
Similarly,
\begin{eqnarray*}
\Delta_2
\ \lesssim \  E\left[\left(1 +  \big|\tilde{S}_T\big| + \big|T^{-1}\check{R}_2\big| \right)^{L_1} 1_{\Omega_T \cap \left\{T^{-1}|\check{R}_2^{a;}| > CT^{-\frac{1 + \ep'}{2}}\right\}}\right]
\ \lesssim \ T^{-\frac{1}{2}}.
\end{eqnarray*}
On the other hand, from the Taylor expansion, we have
\begin{eqnarray*}
f\big( \tilde{S}_T + T^{-1}\check{R}_2\big) - f\big( \tilde{S}_T\big)
= \sum_{|\alpha|=1}T^{-1}\check{R}_2\int_0^1\partial^{\alpha}f\big( \tilde{S}_T + uT^{-1}\check{R}_2\big)du
\ \lesssim \ T^{-1}\big|\check{R}_2\big|\left(1 + \big|\tilde{S}_T\big| + \big|T^{-1}\check{R}_2\big| \right)^{L_1}.
\end{eqnarray*}
Thus, we get 
\begin{eqnarray*}
\Delta_3
\ \lesssim \ E\left[T^{-1}\big|\check{R}_2\big|\left(1 +  \big|\tilde{S}_T\big| + \big|T^{-1}\check{R}_2\big| \right)^{L_1} 1_{\left\{T^{-1}|\check{R}_2^{a;}| \le CT^{-\frac{1 + \ep'}{2}}\right\}}\right]
\ \lesssim \ T^{-\frac{1}{2} - \frac{\ep'}{2} + \ep L_1} \ \lesssim \ T^{-\frac{1}{2}},
\end{eqnarray*}
since we can choose small $\ep$ arbitrary. Finally, we only have to show that $\Delta_4 \lesssim T^{-\frac{1}{2}}$. From the definition of $f$,
\begin{eqnarray}
\label{main ineq}
\Delta_4
&\lesssim& \int_{\reels^d} (1+|z^{(1)}|)^{L1} \left| \tilde{q}_{T,3,D}(z^{(1)}) - q_{T,3}(z^{(1)}) \right|dz^{(1)} \nonumber\\
&\lesssim& \int_{\reels^d} (1+|z^{(1)}|)^{L1} \left| \frac{\tilde{q}_{T,3,D}(z^{(1)})}{\phi(z^{(1)}; \tilde{g}_T^{-1})} - \frac{q_{T,3}(z^{(1)})}{\phi(z^{(1)}; g_T^{-1})} \right| \phi(z^{(1)}; \tilde{g}_T^{-1}) dz^{(1)} \nonumber\\
&& + \int_{\reels^d} (1+|z^{(1)}|)^{L1} \left| \frac{q_{T,3}(z^{(1)})}{\phi(z^{(1)}; g_T^{-1})} \left( 1 - \frac{\phi(z^{(1)}; g_T^{-1})}{\phi(z^{(1)}; \tilde{g}_T^{-1})} \right) \right| \phi(z^{(1)}; \tilde{g}_T^{-1}) dz^{(1)}.
\end{eqnarray}
We see that $\tilde{g}_T - g_T = ( I - g_T\tilde{g}_T^{-1})\tilde{g}_T =  T^{-D}g_T^{-1}\tilde{g}_T$. Since $(g_T^{-1}\tilde{g}_T)^{-1} = I + T^{-D}g_T^{-1}$ is positive definite, $g_T^{-1}\tilde{g}_T$ is also positive definite. With the help of the conditons [A2]-[A3] and Corollary \ref{cumest_3}, we can choose a sufficiently large $K>0$ such that
\begin{eqnarray}
\label{sub ineq 1}
 \left| \frac{\tilde{q}_{T,3,D}(z^{(1)})}{\phi(z^{(1)}; \tilde{g}_T^{-1})} - \frac{q_{T,3}(z^{(1)})}{\phi(z^{(1)}; g_T^{-1})} \right| 
&=& \frac{1}{\sqrt{T}}\Bigg| \Bigg\{ \left(\frac{1}{6}\tilde{\kappa}^{a_1a_2a_3;}_T + \tilde{\mu}^{a_3;}_{b_1b_2}\tilde{g}^{b_1a_1;}\tilde{g}^{b_2a_2;}\right)h_{a_1a_2a_3}(z^{(1)}; \tilde{g}_T^{-1}) + \tilde{\mu}^{a_1;}_{b_1b_2}\tilde{g}^{b_1b_2;}h_{a_1}(z^{(1)}; \tilde{g}_T^{-1}) \Bigg\} \nonumber\\
&&- \Bigg\{ \left(\frac{1}{6}\tilde{\kappa}^{a_1a_2a_3;}_T + \mu^{a_3;}_{b_1b_2}g^{b_1a_1;}g^{b_2a_2;}\right)h_{a_1a_2a_3}(z^{(1)}; g_T^{-1}) + \mu^{a_1;}_{b_1b_2}g^{b_1b_2;}h_{a_1}(z^{(1)}; g_T^{-1}) \Bigg\}\Bigg| \nonumber\\
&\lesssim& T^{-D}(1 + |z^{(1)}| )^K,
\end{eqnarray}
and
\begin{eqnarray}
\label{sub ineq 2}
\left| \frac{q_{T,3}(z^{(1)})}{\phi(z^{(1)}; g_T^{-1})} \right|
&\lesssim& (1 + |z^{(1)}| )^K.
\end{eqnarray}
On the other hand, we obtain
\begin{eqnarray}
\label{sub ineq 3}
\left| 1 - \frac{\phi(z^{(1)}; g_T^{-1})}{\phi(z^{(1)}; \tilde{g}_T^{-1})} \right| 
&=& \left| 1 -\sqrt{\frac{|\tilde{g}_T^{-1}|}{|g_T^{-1}|}} \exp \left( -\frac{1}{2} z^{(1)'} \left( g_T - \tilde{g}_T \right)z^{(1)} \right)  \right| \nonumber\\
&\le& \left| 1 - \sqrt{\frac{|\tilde{g}_T^{-1}|}{|g_T^{-1}|}} \right| +  \sqrt{\frac{|\tilde{g}_T^{-1}|}{|g_T^{-1}|}} \left| 1 -  \exp \left( -\frac{T^{-D}}{2} z^{(1)'} g_T^{-1}\tilde{g}_T z^{(1)} \right) \right| \nonumber\\
&\le& \left| 1 - \sqrt{\frac{|g_T^{-1} + T^{-D}(g_T^{-1})^2|}{|g_T^{-1}|}}\right| + \sqrt{\frac{|g_T^{-1} + T^{-D}(g_T^{-1})^2|}{|g_T^{-1}|}}T^{-D} |z^{(1)'} g_T^{-1}\tilde{g}_T z^{(1)}| \nonumber\\
&\lesssim& T^{-\frac{D}{2}} +T^{-D} |z^{(1)}|^2.
\end{eqnarray}
From  (\ref{main ineq}), (\ref{sub ineq 1}),  (\ref{sub ineq 2}) and (\ref{sub ineq 3}), we get the conclusion by taking sufficiently large $D>0$.
  \end{proof}


\subsection{ Proofs of Subsection 3.2}
Throughout this subsection, denote the $i$-th jump time of $N^x_t$ by $\tau^x_i$.

\begin{proof}[{\bf Proof of Proposition \ref{strong moment}}]
Let $M_1, K_1$ and $K_2$ be positive constants and $\mathscr{A}$ be the operator in Lemma \ref{generator}. Define $g(y, t)$ and $\bar{\mathscr{A}}$ by $g(y, t) = e^{M_1y}e^{K_1t}$ and $\bar{\mathscr{A}}g(y, t) = e^{K_1t}(\mathscr{A}e^{M_1y} + K_1e^{M_1y})$. From Lemma \ref{generator},
\begin{eqnarray}
\label{5B}
\bar{\mathscr{A}}g(y, t) 
\le e^{K_1t}(-K_1e^{M_1y} + K_2 + K_1e^{M_1y})
= e^{K_1t}K_2.
\end{eqnarray}
Since $\lambda^x_{t \wedge \tau^x_i}$ is bounded, thus $g(\lambda^x_{t \wedge \tau^x_i}, t \wedge \tau^x_i) =e^{M_1 \lambda^x_{t \wedge \tau^x_i}}e^{K_1(t \wedge \tau^x_i)}$ is integrable. Furthermore, one may get
\begin{eqnarray}
\label{5A}
g(\lambda^x_{t \wedge \tau^x_i}, t \wedge \tau^x_i) - g(\lambda^x_0, 0) 
&=&\int_{(0,t \wedge \tau^x_i]} g(\lambda^x_s + \alpha, s) - g(\lambda^x_s, s) dN^x_s +  \int_{(0,t \wedge \tau^x_i]}  \frac{d}{ds}  g(\lambda^x_s, s) ds \nonumber\\
&=& \int_{(0, t \wedge \tau^x_i]} g(\lambda^x_s + \alpha, s) -g(\lambda^x_s, s) d\tilde{N}^x_s + \int_{(0, t \wedge \tau^x_i]} \bar{\mathscr{A}}g(\lambda^x_s, s)ds.
\end{eqnarray}
Since $\int_{(0, t]} g(\lambda^x_s + \alpha, s) - g(\lambda^x_s, s) d\tilde{N}^x_s$ is a $\tau^x_i$-local martingale, see Theorem 18.7 in \cite{LS}, (\ref{5B}) and (\ref{5A}) yield
\begin{eqnarray*}
E\left[ g(\lambda^x_{t\wedge \tau^x_i}, t \wedge \tau^x_i) \right] 
&=& g(x, 0) + E\left[ \int_{(0, t \wedge \tau^x_i]} \bar{\mathscr{A}} g(\lambda^x_s, s) ds\right] \\
&\le& g(x, 0) + E\left[ \int_{(0, t \wedge \tau^x_i]} e^{K_1t}K_2 ds\right] \le g(x, 0) + \frac{K_2}{K_1}\left( e^{K_1t} - 1\right).
\end{eqnarray*}
Then, by the Fatou's lemma, we have
\begin{eqnarray*}
E\left[e^{M_1\lambda^x_t}\right] \le e^{-K_1t}\left(g(x, 0) + \frac{K_2}{K_1}\left( e^{K_1t} - 1\right)\right).
\end{eqnarray*}
Thus, we get the conclusion.
  \end{proof}

\begin{proof}[{\bf Proof of Lemma \ref{P in DomA}}]
Since $\mathscr{A}$ is linear, we only have to prove that, for $p(y) = y^m$ with $m \in \naturels$, $M^p_t = p(\lambda^x_t) - p(\lambda^x_0) - \int_{(0,t]} \mathscr{A}p(\lambda^x_s) ds$ is a $\mathscr{F}^x_t$-martingale. Take any large $T>0$. In the same way as (\ref{5A}) in the proof of Proposition \ref{strong moment}, one may confirm that $M^p_t = \int_0^{t} 1_{\{s\le T\}}(\lambda^x_s + \alpha)^m - (\lambda^x_s)^m d\tilde{N}^x_s$ for $t \le T$. Then, Theorem 18.7 in \cite{LS} and Proposition \ref{strong moment} lead the conclusion.
  \end{proof}

\begin{proof}[{\bf Proof of Lemma \ref{Remainder}}]
Let $p(y) = a_my^m + \cdots a_1y + a_0$, where $a_0, \dots, a_m \in \reels$ and $m \in \naturels$. Then, the linearity of $\mathscr{A}$ leads
\begin{eqnarray*}
&&E\left[ \left|\int_{(s,t]}\int_{(s,u_1]}\cdots\int_{(s,u_{n-1}]}E\left[\left.\mathscr{A}^np(\lambda^x_{u_n})\right| \mathscr{F}^x_s\right]du_n\dots du_2du_1\right| \right]\\
&\le& \int_{(s,t]}\int_{(s,u_1]}\cdots\int_{(s,u_{n-1}]}E\left[ \left|\mathscr{A}^np(\lambda^x_{u_n})\right|\right]du_n\dots du_2du_1\\
&\le&\sum_{k=1}^m |a_k|\int_{(s,t]}\int_{(s,u_1]}\cdots\int_{(s,u_{n-1}]}E\left[ \left|\mathscr{A}^n(\lambda^x_{u_n})^k\right|\right]du_n\dots du_2du_1.
\end{eqnarray*}
Therefore, we only have to evaluate $\int_{(s,t]}\int_{(s,u_1]}\cdots\int_{(s,u_{n-1}]}E\left[ \left|\mathscr{A}^n(\lambda^x_{u_n})^k\right|\right]du_n\dots du_2du_1$ for any $k \in \naturels$. In particular, 
\begin{eqnarray*}
\int_{(s,t]}\int_{(s,u_1]}\cdots\int_{(s,u_{n-1}]}E\left[ \left|\mathscr{A}^n(\lambda^x_{u_n})^k\right|\right]du_n\dots du_2du_1
\le \frac{(t-s)^n}{n!}\sup_{u\in(s,t]}E\left[ \left|\mathscr{A}^n(\lambda^x_u)^k\right|\right],
\end{eqnarray*}
thus it is enough to prove that the above right hand side converges to zero as $n \to \infty$. 
There exist constants $C_i$, $i = 0,\dots, k$ such that $\mathscr{A}y^k = C_ky^k + \cdots + C_1y + C_0$. Then, we inductively get
\begin{eqnarray}
\label{5C}
\mathscr{A}^ny^k
&=& \mathscr{A}^{n-1}(\mathscr{A}y^k) 
= C_k\mathscr{A}^{n-2}(\mathscr{A}y^k)  +\sum_{i=1}^{k-1}C_i\mathscr{A}^{n-1}y^i  \nonumber\\
&=& C_k^2\mathscr{A}^{n-3}(\mathscr{A}y^k) + C_k\sum_{i=1}^{k-1}C_i\mathscr{A}^{n-2}y^i  + \sum_{i=1}^{k-1}C_i\mathscr{A}^{n-1}y^i   \nonumber\\
&=&\cdots= C_k^{n-1}\mathscr{A}y^k + C_k^{n-2}\sum_{i=1}^{k-1}C_i\mathscr{A}y^i + C_k^{n-3}\sum_{i=1}^{k-1}C_i\mathscr{A}^2y^i + \cdots +\sum_{i=1}^{k-1}C_i\mathscr{A}^{n-1}y^i \nonumber\\
&=& C_k^ny^k +\sum_{i=1}^{k-1} C_i\left(C_k^{n-1}y^i + C_k^{n-2}\mathscr{A}y^i + \cdots + \mathscr{A}^{n-1}y^i\right) + C_0C_k^{n-1}.
\end{eqnarray}
Hence,
\begin{eqnarray*}
\sup_{u\in(s,t]}E\left[ \left|\mathscr{A}^n(\lambda^x_u)^k\right|\right]
&\le& C_k^n\sup_{u\in(s,t]}E\left[ \left|(\lambda^x_u)^k\right|\right]\\
&&+ \sum_{i=1}^{k-1}C_i\sup_{u\in(s,t]}E\left[\left|C_k^{n-1}(\lambda^x_u)^i + C_k^{n-2}\mathscr{A}(\lambda^x_u)^i + \cdots + \mathscr{A}^{n-1}(\lambda^x_u)^i\right|\right] +C_0C_k^{n-1}.
\end{eqnarray*}
Furthermore, one may concretely compute as $C_k = k(\alpha - \beta)$. Now, we introduce the following assumption.
\begin{assk}
For any $i= 0, \dots, k-1$ and $C = j(\alpha-\beta)$ with $j = k, k+1, \cdots$,
\begin{eqnarray*}
\frac{(t-s)^n}{n!}\sup_{u\in(s,t]}E\left[\left|C^{n-1}(\lambda^x_u)^i + C^{n-2}\mathscr{A}(\lambda^x_u)^i + \cdots + \mathscr{A}^{n-1}(\lambda^x_u)^i\right|\right] \to 0 \quad \text{as $n \to \infty$}.
\end{eqnarray*}
\end{assk}
Proposition \ref{strong moment} guarantees $\sup_{u\in(s,t]}E\left[ \left|(\lambda^x_u)^k\right|\right] < \infty$. Thus, if ASS(k) holds, by taking $C = k(\alpha-\beta)$ in ASS(k), we have
\begin{eqnarray*}
\frac{(t-s)^n}{n!}\sup_{u\in(s,t]}E\left[ \left|\mathscr{A}^n(\lambda^x_u)^k\right|\right] \to 0.
\end{eqnarray*}
We prove that ASS(k) holds for any $k \in \naturels$ by induction. In the case of $k=1$, this assumption is obvious. Assume that ASS(k) holds. Again we denote $\mathscr{A}y^k = C_ky^k + \cdots + C_1y + C_0$. For $C = j(\alpha-\beta)$ with $j = k+1, k+2, \cdots$, by using the equation (\ref{5C}), we have
\begin{eqnarray*}
\label{5D}
&&C^{n-1}y^k + C^{n-2}\mathscr{A}y^k + \cdots + \mathscr{A}^{n-1}y^k\\
&=& C^{n-1}y^k + C^{n-2}\left\{C_ky^k + \sum_{i=1}^{k-1}C_iy^i + C_0\right\}+ C^{n-3}\left\{C_k^2y^k +  \sum_{i=1}^{k-1}C_i(C_ky^{i} + \mathscr{A}y^{i}) + C_0C_k\right\}\\
&&+ C^{n-4}\left\{C_k^3y^k+ \sum_{i=1}^{k-1}C_i(C_k^2y^i +C_k\mathscr{A}y^i+ \mathscr{A}^2y^i) + C_0C_k^2\right\}\\
&& +\cdots+\left\{C_k^{n-1}y^k +\sum_{i=1}^{k-1} C_i\left(C_k^{n-2}y^i + C_k^{n-3}\mathscr{A}y^i + \cdots + \mathscr{A}^{n-2}y^i\right) +C_0C_k^{n-2}\right\}\\
&=& \sum_{i=0}^{n-1}C^iC_k^{n-1-i}y^k+ C_{k-1}\sum_{j=0}^{n-2}\left(\sum_{i=0}^{n-2-j}C^iC_k^{n-2-i-j}\right)\mathscr{A}^{j}y^{k-1} + C_{k-2}\sum_{j=0}^{n-2}\left(\sum_{i=0}^{n-2-j}C^iC_k^{n-2-i-j}\right)\mathscr{A}^{j}y^{k-2}\\
&&+ \cdots+ C_{1}\sum_{j=0}^{n-2}\left(\sum_{i=0}^{n-2-j}C^iC_k^{n-2-i-j}\right)\mathscr{A}^{j}y + C_0\left(\sum_{i=0}^{n-2}C^iC_k^{n-2-i}\right)\\
&=& \frac{C^n-C_k^n}{C-C_k}y^k+ C_{k-1}\sum_{j=0}^{n-2}\frac{C^{n-1-j}-C_k^{n-1-j}}{C-C_k}\mathscr{A}^{j}y^{k-1} + C_{k-2}\sum_{j=0}^{n-2}\frac{C^{n-1-j}-C_k^{n-1-j}}{C-C_k}\mathscr{A}^{j}y^{k-2}\\
&&+ \cdots+ C_{1}\sum_{j=0}^{n-2}\frac{C^{n-1-j}-C_k^{n-1-j}}{C-C_k}\mathscr{A}^{j}y + C_0\frac{C^{n-1}-C_k^{n-1}}{C-C_k}\\
&=&\frac{1}{C-C_k}\Bigg\{ (C^n-C_k^n)y^k +C_{k-1}\Bigg\{\sum_{i=0}^{n-1}C^i\mathscr{A}^{n-1-i}y^{k-1}- \sum_{i=0}^{n-1}C_k^i\mathscr{A}^{n-1-i}y^{k-1}\Bigg\} \\
&& +C_{k-2}\Bigg\{\sum_{i=0}^{n-1}C^i\mathscr{A}^{n-1-i}y^{k-2}- \sum_{i=0}^{n-1}C_k^i\mathscr{A}^{n-1-i}y^{k-2}\Bigg\} + \cdots +C_1\Bigg\{\sum_{i=0}^{n-1}C^i\mathscr{A}^{n-1-i}y- \sum_{i=0}^{n-1}C_k^i\mathscr{A}^{n-1-i}y\Bigg\}+C_0(C^{n-1}-C_k^{n-1})\Bigg\}.
\end{eqnarray*}
Therefore,
\begin{eqnarray*}
&&\frac{(t-s)^n}{n!}\sup_{u\in(s,t]}E\left[\left|C^{n-1}(\lambda^x_u)^k + C^{n-2}\mathscr{A}(\lambda^x_u)^k + \cdots + \mathscr{A}^{n-1}(\lambda^x_u)^k\right|\right] \\
&\le& \frac{(t-s)^n}{n!}\left|\frac{1}{C-C_k}\right|\Bigg\{\left|C^n-C_k^n\right|\sup_{u\in(s,t]}E\left[(\lambda^x_u)^k\right] \\
&& \quad +\left|C_{k-1}\right|\left\{\sup_{u\in(s,t]}E\left[\left|\sum_{i=0}^{n-1}C^i\mathscr{A}^{n-1-i}(\lambda^x_u)^{k-1} \right|\right] +\sup_{u\in(s,t]}E\left[\left|\sum_{i=0}^{n-1}C_k^i\mathscr{A}^{n-1-i}(\lambda^x_u)^{k-1} \right|\right]\right\}\\
&& \quad +\left|C_{k-2}\right|\left\{\sup_{u\in(s,t]}E\left[\left|\sum_{i=0}^{n-1}C^i\mathscr{A}^{n-1-i}(\lambda^x_u)^{k-2} \right|\right] +\sup_{u\in(s,t]}E\left[\left|\sum_{i=0}^{n-1}C_k^i\mathscr{A}^{n-1-i}(\lambda^x_u)^{k-2} \right|\right]\right\}\\
&& \quad+\cdots+ \left|C_1\right|\left\{\sup_{u\in(s,t]}E\left[\left|\sum_{i=0}^{n-1}C^i\mathscr{A}^{n-1-i}\lambda^x_u\right|\right]+\sup_{u\in(s,t]}E\left[\left|\sum_{i=0}^{n-1}C_k^i\mathscr{A}^{n-1-i}\lambda^x_u \right|\right]\right\}+\left|C_0(C^{n-1}-C_k^{n-1})\right|\Bigg\}.
\end{eqnarray*}
Hence, ASS(k) leads ASS(k+1) and we have completed the proof.
  \end{proof}

To prove Theorem \ref{Markovian property}, we prepare the following lemmas.

\begin{lemma*}
\label{exp markov}
For any $u \in \reels$ and $t \ge s \ge 0$, $E[e^{iu\lambda^x_t}| \mathscr{F}^x_s] = E[e^{iu\lambda^x_t}| \lambda^x_s] \ \ a.s.$
\end{lemma*}

\begin{proof}
Fix $s$ and $t$ with $t \ge s \ge 0$. It is sufficient to show that for any bounded $\mathscr{F}^x_s$-measurable function $g : \Omega \to \reels$,
\begin{eqnarray*}
E[e^{iu\lambda^x_t}g] = E\left[E[e^{iu\lambda^x_t}| \lambda^x_s]g\right].
\end{eqnarray*}
Note that
\begin{eqnarray*}
E\left[E[e^{iu\lambda^x_t}| \lambda^x_s]g\right]
=E\left[E[e^{iu\lambda^x_t}| \lambda^x_s]E[g| \lambda^x_s]\right]
=E\left[e^{iu\lambda^x_t}E[g| \lambda^x_s]\right].
\end{eqnarray*}
Let $D = \{z \in \complex ; \ \operatorname{Re}z < \frac{M_1}{2}\}$, where $M_1$ is the positive constant chosen in Proposition \ref{strong moment}. First, we will prove that $f(z) = E[e^{z\lambda^x_t}g]$ is holomorphic on $D$ for any $\mathscr{F}^x_s$-measurable function $g$. Let $z = a + ib$, where $a, b \in \reels$ with $a < \frac{M_1}{2}$. Then, we have
\begin{eqnarray*}
\left|f(z)\right| \le \|g\|_{\infty}E[e^{a\lambda^x_t}] < \infty,
\end{eqnarray*}
and thus $f(z)=E[e^{z\lambda^x_t}g]$ is defined on $D$. Define $u(a,b)$ and $v(a,b)$ as the real part and the imaginary part of $f(z)$ respectively, namely,
\begin{eqnarray*}
f(z) = E[\cos(b\lambda^x_t)e^{a\lambda^x_t}g] + iE[\sin(b\lambda^x_t)e^{a\lambda^x_t}g] = u(a,b) + iv(a,b).
\end{eqnarray*}
Write $\partial_a = \frac{\partial}{\partial a}$ and $\partial_b = \frac{\partial}{\partial b}$. $|\partial_a(\cos(b\lambda^x_t)e^{a\lambda^x_t}g)|$, $|\partial_b(\cos(b\lambda^x_t)e^{a\lambda^x_t}g)|$, $|\partial_a(\sin(b\lambda^x_t)e^{a\lambda^x_t}g)|$ and $|\partial_b(\sin(b\lambda^x_t)e^{a\lambda^x_t}g)|$ are dominated by an integrable random variable $|\lambda^x_te^{\frac{M_1}{2}\lambda^x_t}g|$ on $D$. Hence, the permutation of differential and integral is permitted, and thus we have
\begin{align*}
\partial_au(a,b) = \partial_bv(a,b) = E[\lambda^x_t\cos(b\lambda^x_t)e^{a\lambda^x_t}g] \ \text{ and } \
\partial_bu(a,b) = -\partial_av(a,b) = -E[\lambda^x_t\sin(b\lambda^x_t)e^{a\lambda^x_t}g].
\end{align*}
The Lebesgue's theorem guarantees that $\partial_au(a,b)$, $\partial_bu(a,b)$, $\partial_av(a,b)$ and $\partial_bv(a,b)$ are continuous with respect to $a$ and $b$. In particular, they are total differentiable. Then, the Cauchy-Riemann relations lead that $f(z)$ is holomorphic on $D$. Completely similarly, we can prove that $z  \mapsto E\left[e^{z\lambda^x_t}E[g| \lambda^x_s]\right]$ is also holomorphic on $D$.

Second, we will confirm that $E[e^{z\lambda^x_t}g] = E\left[e^{z\lambda^x_t}E[g| \lambda^x_s]\right]$ for $z \in (-\frac{M_1}{2}, \frac{M_1}{2})$. Let $p_N(x) = \sum_{n=0}^N\frac{(zx)^n}{n!}$. Then, (\ref{Markov for polynomial}) leads 
\begin{eqnarray*}
E\left[\sum_{n=0}^N\frac{(z\lambda^x_t)^n}{n!}g\right]
&=& E\left[E\left[\left.\sum_{n=0}^N\frac{(z\lambda^x_t)^n}{n!}g\right|\mathscr{F}^x_s\right]\right] = E\left[E\left[\left.\sum_{n=0}^N\frac{(z\lambda^x_t)^n}{n!}\right|\mathscr{F}^x_s\right]g\right]\\
&=& E\left[E\left[\left.\sum_{n=0}^N\frac{(z\lambda^x_t)^n}{n!}\right|\lambda^x_s\right]g\right] =  E\left[\sum_{n=0}^N\frac{(z\lambda^x_t)^n}{n!}E[g| \lambda^x_s]\right].
\end{eqnarray*}
On the other hand, since $\sum_{n=0}^N\frac{(z\lambda^x_t)^n}{n!} \to e^{z\lambda^x_t}$ as $N \to \infty$ and $|\sum_{n=0}^N\frac{(z\lambda^x_t)^n}{n!}| \le |\sum_{n=0}^N\frac{(\frac{M_1}{2}\lambda^x_t)^n}{n!}| \le e^{\frac{M_1}{2}\lambda^x_t}$ hold for every $z \in (-\frac{M_1}{2}, \frac{M_1}{2})$, we have
\begin{eqnarray*}
\sum_{n=0}^N\frac{(z\lambda^x_t)^n}{n!}g \to e^{z\lambda^x_t}g \quad \text{as $N \to \infty$ in $L^1$-sence,}
\end{eqnarray*}
 and 
 \begin{eqnarray*}
 \sum_{n=0}^N\frac{(z\lambda^x_t)^n}{n!}E[g| \lambda^x_s] \to e^{z\lambda^x_t}E[g| \lambda^x_s] \quad \text{as $N \to \infty$ in $L^1$-sence}
 \end{eqnarray*}
by the Lebesgue's theorem. Therefore, we get the desired equation
 \begin{eqnarray*}
 E\left[e^{z\lambda^x_t}g\right] = E\left[e^{z\lambda^x_t}E[g| \lambda^x_s]\right] \quad \text{for $z \in \left(-\frac{M_1}{2}, \frac{M_1}{2}\right)$.}
 \end{eqnarray*}
Now, the identity theorem guarantees the conclusion.
  \end{proof}

\begin{proof}[{\bf Proof of Theorem \ref{Markovian property}}]
For almost every $a, b \in \reels$ with $a < b$, we will prove that
\begin{eqnarray}
\label{F levy}
P \Big[\lambda^x_t \in (a,b] \Big| \mathscr{F}^x_s \Big] = \lim_{A \to \infty} \frac{1}{2\pi} \int_{-A}^A \frac{e^{-iua} - e^{-iub}}{iu}E\left[e^{iu\lambda^x_t} \Big| \mathscr{F}^x_s\right]du \quad a.s,
\end{eqnarray}
and
\begin{eqnarray}
\label{G levy}
P \Big[\lambda^x_t \in (a,b] \Big| \lambda^x_s \Big]  = \lim_{A \to \infty} \frac{1}{2\pi} \int_{-A}^A \frac{e^{-iua} - e^{-iub}}{iu}E\left[e^{iu\lambda^x_t} \Big| \lambda^x_s\right]du \quad a.s.
\end{eqnarray}
However, by considering a probability measure $P_F(d\omega) = P[d\omega \cap F]/P[F]$, the L\'evy's inversion formula gives
\begin{eqnarray*}
P \Big[ \big\{ \lambda^x_t \in (a,b] \big\} \cap F \Big] = \lim_{A \to \infty} \frac{1}{2\pi} \int_{-A}^A \frac{e^{-iua} - e^{-iub}}{iu}E\left[e^{iu\lambda^x_t} 1_F \right]du
\end{eqnarray*}
 for any set $F \in \mathscr{F}^x_s$. Moreover, we know
 \begin{eqnarray*}
 \frac{1}{2\pi} \int_{-A}^A \frac{e^{-iua} - e^{-iub}}{iu}E\left[e^{iu\lambda^x_t} 1_F \right]du
 =  E\left[ E\left[ \frac{1}{\pi}\big\{ D_A(\lambda^x_t-a) - D_A(\lambda^x_t-b) \big\} \Big| \mathscr{F}^x_s\right]1_F \right]\\
 \end{eqnarray*} 
 where $D_A$ is the Dirichlet integral, i.e.
 \begin{eqnarray*}
 D_A(\alpha) = \int_0^A \frac{\sin(u\alpha)}{u}du.
 \end{eqnarray*}
Then, as is well known, we can apply the Lebesgue's theorem and get
\begin{eqnarray*}
 \lim_{A\to \infty} \frac{1}{2\pi} \int_{-A}^A \frac{e^{-iua} - e^{-iub}}{iu}E\left[e^{iu\lambda^x_t} 1_F \right]du
 =  E\left[ \left(\lim_{A \to \infty} \frac{1}{2\pi} \int_{-A}^A \frac{e^{-iua} - e^{-iub}}{iu} E\left[e^{iu\lambda^x_t} \Big| \mathscr{F}^x_s\right]du \right)1_F\right].
\end{eqnarray*}
Thus, (\ref{F levy}) holds. In the same way, (\ref{G levy}) also holds.
Then, from Lemma \ref{exp markov}, we get $P[\lambda^x_t \in (a,b] | \mathscr{F}^x_s] = P[\lambda^x_t \in (a,b] | \lambda^x_s] \ \ a.s.$ for almost every $a, b \in \reels$ with $a < b$. With the help of the monotone class theorem, $E[f(\lambda^x_t)| \mathscr{F}^x_s] = E[f(\lambda^x_t)| \lambda^x_s] \ \ a.s.$ holds for any bounded measurable function $f$.
  \end{proof}


\subsection{ Proofs of Subsection 3.3}

\subsubsection{ Markovian property }

\begin{proof}[{\bf Proof of Proposition \ref{X Markovian property}}]

From Theorem \ref{Markovian property}, we immediately get, for any $t \ge s \ge 0$ and bounded measurable function $f$, 
\begin{eqnarray}
\label{(1) Markov}
E\big[f\big(X^{x,(1)}_t\big) \big| \mathscr{F}^x_s \big] = E\big[f\big(X^{x,(1)}_t\big) \big| X^{x,(1)}_s \big]  \ \ a.s.
\end{eqnarray}
$X_t^x$ has the following relation. For any $t \ge s \ge 0$,
\begin{eqnarray*}
X^{x,(2)}_t 
&=& \left(x_1t + x_2\right)e^{-\beta t} + \int_{(0,s)} \alpha (t-u)e^{-\beta(t-u)}dN_u^{x_1} + \int_{[s,t)} \alpha (t-u)e^{-\beta(t-u)}dN_u^{x_1}\\
&=& \Bigg\{ (t-s)\Big(x_1e^{-\beta s} + \int_{(0,s)} \alpha e^{-\beta(s-u)}dN_u^{x_1}\Big)+ \Big( (x_1s + x_2)e^{-\beta s} + \int_{(0,s)} \alpha (s-u)e^{-\beta(s-u)}dN_u^{x_1}\Big) \Bigg\}e^{-\beta(t-s)} \\
&&+  \int_{[s,t)} \alpha (t-u)e^{-\beta(t-u)}dN_u^{x_1}\\
&=&  \left( X^{x,(1)}_s(t-s) + X^{x,(2)}_s \right)e^{-\beta(t-s)} +  \int_{[s,t)} \alpha (t-u)e^{-\beta(t-u)}dN_u^{x_1},
\end{eqnarray*}
and similarly,
\begin{eqnarray*}
X^{x,(3)}_t = \left(X^{x,(1)}_s(t-s)^2 + 2X^{x,(2)}_s(t-s) + X^{x,(3)}_s\right)e^{-\beta (t-s)} + \int_{[s,t)} \alpha (t-u)^2e^{-\beta(t-u)}dN_u^{x_1}.
\end{eqnarray*}
Therefore,  $X^x_t$ is $\sigma \big( X^{x,(1)}_u, X^{x,(2)}_s, X^{x,(3)}_s; u \in [s,t] \big)$-measurable for any $t \ge s\ge 0$.

Let $g_1, g_2$ and $g_3$ be $\sigma \big( X^{x,(1)}_u; u \in [s,t] \big), \sigma \big( X^{x,(2)}_s \big)$ and $\sigma \big(X^{x,(3)}_s \big)$ measurable bounded functions respectively. Then, (\ref{(1) Markov}) and the monotone class theorem lead
\begin{eqnarray*}
E[g_1|X^x_s] 
= E[g_1|\mathscr{F}^x_s|X^x_s]
= E[g_1|X^{x,(1)}_s|X^x_s]
= E[g_1|X^{x,(1)}_s]
= E[g_1|\mathscr{F}^x_s]\ \ a.s.
\end{eqnarray*}
Thus,
\begin{eqnarray*}
E[g_1g_2g_3|\mathscr{F}^x_s]
= g_2g_3E[g_1|\mathscr{F}^x_s]
= g_2g_3E[g_1|X^x_s]
= E[g_1g_2g_3|X^x_s]\ \ a.s.
\end{eqnarray*}
By the monotone class theorem, we get the conclusion.
  \end{proof}

Before we prove the homogeneous Markov property of process $X$, we prepare the following  technical lemma.
\begin{lemma*}
\label{Y homo}
For any bounded function $f$ defined on the path space of $X^{x,(1)}_u, u \in [s,t]$,
\begin{eqnarray*}
E\big[ f \big( X^{x,(1)}_u; u \in [s,t] \big) \big| \mathscr{F}^x_s \big](\omega) =  E\big[ f \big( X^{y,(1)}_{u-s}; u \in [s,t] \big) \big] \big|_{y = X^{x,(1)}_s(\omega) } \ \ a.s \ \omega.
\end{eqnarray*}
\end{lemma*} 

\begin{proof}
From Theorem \ref{Markovian property} and the monotone class theorem, $E\big[f \big( X^{x,(1)}_u; u \in [s,t] \big) \big| \mathscr{F}^x_s \big] = E\big[f \big( X^{x,(1)}_u; u \in [s,t] \big) \big| X^{x,(1)}_s \big] \ a.s.$ holds. Therefore, we only have to prove the statement replaced $ \mathscr{F}^x_s$ by $ \sigma \big( X^{x,(1)}_s \big)$.
For any $p \in \mathscr{P}$, we know that for almost every $\omega \in \Omega$ and $s \le t$,
\begin{eqnarray*}
E\left[\left. p \big( X^{x,(1)}_{t} \big) \right| X^{x,(1)}_s\right](\omega) 
&=& E\left[\left. p \left( \lambda^{x_1}_{t} - \mu \right) \right| \lambda_s^{x_1} \right](\omega) \\
&=& e^{(t-s)\mathscr{A}} p \left( \lambda^{x_1}_{s}(\omega) - \mu \right) \\
&=& e^{(t-s)\mathscr{A}} p \left( \lambda^{X^{x,(1)}_{s}(\omega)}_0 (\tilde{\omega})-\mu \right) \\
&=& E\left[\left. p \left( \lambda^y_{t-s}-\mu \right)  \right] \right|_{y=X^{x,(1)}_s(\omega)}\\
&=& E\left[\left. p \left( X^{y,(1)}_{t-s} \right)  \right] \right|_{ y=X^{x,(1)}_s(\omega) },
\end{eqnarray*}
where the operator $e^{(t-s)\mathscr{A}}$ is defined in Subsection 3.2. In particular, for $u \in \reels$ with $|u| \le \frac{M_1}{2} $ and $p_N(y) = \sum_{n=0}^N \frac{(uy)^n}{n!} \in \mathscr{P}$,
\begin{eqnarray*}
E\left[\left. p_N \big( X^{x,(1)}_{t} \big) \right| X^{x,(1)}_s\right] = E\left[\left. p_N \left( X^{y,(1)}_{t-s} \right)  \right] \right|_{ y=X^{x,(1)}_s } \ \ a.s.
\end{eqnarray*}
Moreover, Proposition \ref{strong moment} and the Lebesgue's theorem give
\begin{eqnarray*}
E\left[\left.e^{iuX^{x,(1)}_t } \right| X^{x,(1)}_s\right] =  E\left[\left.e^{iuX^{y,(1)}_{t-s}}  \right] \right|_{ y=X^{x,(1)}_s } \ \ a.s.
\end{eqnarray*}
In the same way of the proof of Lemma \ref{exp markov}, one may confirm that Proposition \ref{strong moment} and the identity theorem guarantee that for general $u \in \reels$,
\begin{eqnarray*}
E\left[\left.e^{iuX^{x,(1)}_t } \right| X^{x,(1)}_s\right] =  E\left[\left.e^{iuX^{y,(1)}_{t-s}}  \right] \right|_{ y=X^{x,(1)}_s } \ \ a.s.
\end{eqnarray*}
Finally, it is proved in the same way as Theorem \ref{Markovian property} that
\begin{eqnarray*}
E\left[\left. f \big( X^{x,(1)}_{t} \big) \right| X^{x,(1)}_s\right] = E\left[\left. f \left( X^{y,(1)}_{t-s} \right)  \right] \right|_{ y=X^{x,(1)}_s } \ \ a.s.
\end{eqnarray*}
for any bounded measurable function $f$.
Thus, for any $s \le u_1 \le u_2 \le t$ and bounded functions $g_1,g_2$, 
\begin{eqnarray*}
E\bigg[ g_1\big( X^{x,(1)}_{u_1} \big) g_2\big( X^{x,(1)}_{u_2} \big) \bigg| X^{x,(1)}_s \bigg]
&=& E\bigg[g_1\big( X^{x,(1)}_{u_1} \big) E\big[ g_2\big( X^{x,(1)}_{u_2} \big) \big| \mathscr{F}^x_{u_1} \big] \bigg| X^{x,(1)}_s \bigg]\\
&=& E\bigg[g_1\big( X^{x,(1)}_{u_1} \big) E\big[ g_2\big( X^{x,(1)}_{u_2} \big) \big| X^{x,(1)}_{u_1} \big] \bigg| X^{x,(1)}_s \bigg]\\
&=& E\bigg[g_1\big( X^{x,(1)}_{u_1} \big) E\big[ g_2\big( X^{y_1,(1)}_{u_2-u_1} \big)  \big] \big|_{y_1 = X^{x,(1)}_{u_1}} \bigg| X^{x,(1)}_s \bigg]\\
&=& E\bigg[ g_1\big( X^{y_2,(1)}_{u_1-s} \big) E\big[ g_2\big( X^{y_1,(1)}_{u_2-u_1} \big)  \big] \big|_{y_1 = X^{y_2,(1)}_{u_1-s}} \bigg]  \bigg|_{y_2 = X^{x,(1)}_s}\\
&=& E\bigg[g_1\big( X^{y_2,(1)}_{u_1-s} \big) E\big[ g_2\big( X^{y_2,(1)}_{u_2-s} \big)  \big|  X^{y_2,(1)}_{u_1-s} \big] \bigg]  \bigg|_{y_2 = X^{x,(1)}_s}\\
&=& E\bigg[g_1\big( X^{y,(1)}_{u_1-s} \big) g_2\big( X^{y,(1)}_{u_2-s} \big) \bigg]  \bigg|_{y = X^{x,(1)}_s} \ \ a.s.
\end{eqnarray*}
Inductively, we also get for any $k \in \naturels$, $s \le u_1 \le \dots \le u_k \le t$ and bounded functions $g_1, \dots, g_k$,
\begin{eqnarray*}
E\bigg[ g_1\big( X^{x,(1)}_{u_1} \big) \cdots g_k\big( X^{x,(1)}_{u_k} \big) \bigg| X^{x,(1)}_s \bigg]
=E\bigg[g_1\big( X^{y,(1)}_{u_1-s} \big) \cdots g_k\big( X^{y,(1)}_{u_k-s} \big) \bigg]  \bigg|_{y = X^{x,(1)}_s} \ \ a.s.
\end{eqnarray*}
By considering cylinder sets in $\sigma \big( X^{x,(1)}_u; u \in [s,t] \big)$, the monotone class theorem gives the conclusion.
  \end{proof}

\begin{proof}[{\bf Proof of Proposition \ref{X homo property}}]
Let $h_1$ be a bounded function defined on the path space of $X^{x,(1)}_u, u \in [s,t]$. Moreover, let $h_2$ and $h_3$ be bounded functions on $\reels$.  Lemma \ref{Y homo} leads
\begin{eqnarray*}
E\bigg[  h_1 \big( X^{x,(1)}_u; u \in [s,t] \big) h_2 \big( X^{x,(2)}_s \big) h_3 \big( X^{x,(3)}_s \big) \bigg| X^x_s \bigg]
&=& h_2 \big( X^{x,(2)}_s \big) h_3 \big( X^{x,(3)}_s \big) E\bigg[  h_1 \big( X^{x,(1)}_u; u \in [s,t] \big)  \bigg| X^x_s \bigg]\\
&=& h_2 \big( X^{x,(2)}_s \big) h_3 \big( X^{x,(3)}_s \big) E\bigg[  h_1 \big( X^{x,(1)}_u; u \in [s,t] \big) \bigg| \mathscr{F}^x_s \bigg| X^x_s  \bigg]\\
&=& h_2 \big( X^{x,(2)}_s \big) h_3 \big( X^{x,(3)}_s \big) E\bigg[ h_1 \big( X^{y,(1)}_{u_k-s} ; u \in [s,t] \big) \bigg]  \bigg|_{y = X^{x,(1)}_s}\\
&=& E\bigg[ h_1 \big( X^{y_1,(1)}_{u_k-s} ; u \in [s,t] \big)  h_2 ( y_2 ) h_3 ( y_3 ) \bigg]  \bigg|_{(y_1, y_2, y_3) = X^x_s} \ \ a.s.
\end{eqnarray*}
Therefore, the monotone class theorem yields that 
\begin{eqnarray*}
E\bigg[ h( X^{x,(1)}_u, X^{x,(2)}_s, X^{x,(3)}_s; u \in [s,t] \big) \bigg| X^x_s \bigg] = E\bigg[ h \big( X^{y_1,(1)}_{u-s}, y_2, y_3; u \in [s,t] \big) \bigg]  \bigg|_{(y_1, y_2, y_3) = X^x_s} \ \ a.s.
\end{eqnarray*}
for any bounded function $h$. Note that $X^{x,(1)}_u(\omega), u \in [s,t]$ completely determines the jumps of $N^{x_1}_u(\omega), u \in [s,t)$. Thus, for any bounded measurable function $f$, we can conclude
\begin{eqnarray*}
E\big[ f  \big( X^x_t \big) \big| X^x_s \big] 
&=& E\left[ \left. f\left(
    \begin{array}{ccc}
      X^{x,(1)}_s\\
      \left( X^{x,(1)}_s(t-s) + X^{x,(2)}_s \right)e^{-\beta(t-s)} +  \int_{[s,t)} \alpha (t-u)e^{-\beta(t-u)}dN_u^{x_1} \\
      \left(X^{x,(1)}_s(t-s)^2 + 2X^{x,(2)}_s(t-s) + X^{x,(3)}_s\right)e^{-\beta (t-s)} + \int_{[s,t)} \alpha (t-u)^2e^{-\beta(t-u)}dN_u^{x_1} \\
    \end{array}
  \right)\right| X^x_s \right] \\
&=& E\left[ \left. f\left(
    \begin{array}{ccc}
      y_1\\
      \left( y_1(t-s) + y_2 \right)e^{-\beta(t-s)} +  \int_{[s,t)} \alpha (t-u)e^{-\beta(t-u)}dN_{u-s}^{y_1} \\
      \left( y_1(t-s)^2 + 2y_2(t-s) + y_3\right)e^{-\beta (t-s)} + \int_{[s,t)} \alpha (t-u)^2e^{-\beta(t-u)}dN_{u-s}^{y_1} \\
    \end{array}
  \right) \right] \right|_{(y_1, y_2, y_3) = X^x_s}  \\
&=& E\left[ \left. f\left(
    \begin{array}{ccc}
      y_1\\
      \left( y_1(t-s) + y_2 \right)e^{-\beta(t-s)} +  \int_{[0,t-s)} \alpha (t-s-u)e^{-\beta(t-s-u)}dN_u^{y_1} \\
      \left( y_1(t-s)^2 + 2y_2(t-s) + y_3\right)e^{-\beta (t-s)} + \int_{[0,t-s)} \alpha (t-s-u)^2e^{-\beta(t-s-u)}dN_u^{y_1} \\
    \end{array}
  \right) \right] \right|_{(y_1, y_2, y_3) = X^x_s}  \\
&=& E\big[ f  \big( X^y_{t-s} \big) \big]  \big|_{y = X^x_s} 
= \int_{\reels^3} f(y) dP^{t-s}(X^x_s, dy)\ \ a.s.
\end{eqnarray*}
 \end{proof}

To prove Proposition \ref{inv meas}, we prepare the following lemma. 
\begin{lemma*}
\label{5E}
For any $0 \le s < t$ and a bounded measurable function $f$,
\begin{eqnarray*}
E\left[f(X^x_{t-s})\right]|_{x=\bar{X}_s(\bar{\omega})} = \bar{E}\left[f(\bar{X}_t)|\bar{X}_s\right](\bar{\omega}) \quad a.s. \ \bar{\omega}
\end{eqnarray*}
\end{lemma*}

\begin{proof}
We again set $g(x, t) = e^{M_1x}e^{K_1t}$ and the operator $\bar{\mathscr{A}}$ same as Proposition \ref{strong moment}. Denote the $i$-th jump time of  $\bar{N}_t$ from time zero by $\tau_i$, i.e. $\tau_i = \inf\{t \ge 0 \ | \ \bar{N}[ (0,t] ] = i \}$.
Then, $\int_{(0, t]} g(\bar{\lambda}^{(1)}_s + \alpha, s) - g(\bar{\lambda}^{(1)}_s, s) \big(d\bar{N}_s - \bar{\lambda}^{(1)}_s ds \big)$ is a $\tau_i$-local martingale, see Theorem 18.7 in \cite{LS}. In the same way as the proof of Proposition \ref{strong moment}, we get,
\begin{eqnarray*}
g\big(\bar{\lambda}^{(1)}_{t \wedge \tau_i}, t \wedge \tau_i \big) -  g\big( \bar{\lambda}^{(1)}_0, 0 \big)
= \int_{(0, t \wedge \tau_i]} g(\bar{\lambda}^{(1)}_s + \alpha, s) -g(\bar{\lambda}^{(1)}_s, s) \big(d\bar{N}_s - \bar{\lambda}^{(1)}_s ds \big) + \int_{(0, t \wedge \tau_i]} \bar{\mathscr{A}}g(\bar{\lambda}^{(1)}_s, s)ds  \quad a.s.
\end{eqnarray*}
Thus, we have
\begin{eqnarray*}
\bar{E}\left[ g(\bar{\lambda}^{(1)}_{t\wedge \tau_i}, t \wedge \tau_i) - g(\bar{\lambda}^{(1)}_0, 0) \Big| \bar{\lambda}^{(1)}_0 = x \right] 
= E\left[ \int_{(0, t \wedge \tau_i]} \bar{\mathscr{A}} g(\bar{\lambda}_s^{(1)}, s) ds  \bigg| \bar{\lambda}^{(1)}_0 = x \right] 
\le \frac{K_2}{K_1}\left( e^{K_1t} - 1\right).
\end{eqnarray*}
From the Fatou's lemma, we have
\begin{eqnarray*} 
\bar{E}\left[ e^{M_1\bar{\lambda}^{(1)}_t}e^{K_1t}  \Big| \bar{\lambda}^{(1)}_0 = x \right] - e^{M_1x}  \le \frac{K_2}{K_1}\left( e^{K_1t} - 1\right).
\end{eqnarray*}
Then, from the stationarity of $\bar{\lambda}^{(1)}_0$, we also get the finiteness of moments of $\bar{\lambda}^{(1)}_t$ by
\begin{eqnarray*} 
\bar{E}\left[ e^{M_1\bar{\lambda}^{(1)}_t} \right] \le \frac{K_2}{K_1}.
\end{eqnarray*}
For the operator $\mathscr{A}$ same as (\ref{exgenerator}), $\bar{E}\big[ p\big( \bar{\lambda}^{(1)}_t \big) \big| \bar{\lambda}^{(1)}_s \big] = e^{(t-s)\mathscr{A}}p\big(\bar{\lambda}^{(1)}_s\big)$ a.s. holds for any $p \in \mathscr{P}$ in the same way of the proofs for Lemma \ref{P in DomA}, Lemma \ref{Remainder} and (\ref{Markov for polynomial}). These properties lead the Markovian property of $\bar{\lambda}^{(1)}_t$ as in the proof of Theorem \ref{Markovian property}. Furthermore, for almost every $\bar{\omega}$,
\begin{eqnarray*} 
\bar{E}\left[ p(\bar{X}^{(1)}_t) \Big| \bar{X}^{(1)}_s \right] (\bar{\omega})
= e^{(t-s)\mathscr{A}}p\left(\bar{\lambda}^{(1)}_s(\bar{\omega}) - \mu \right)
= e^{(t-s)\mathscr{A}}p\left( \lambda^{\bar{X}^{(1)}_s(\bar{\omega})}_0 - \mu\right)
= E\left[p(X^{x,(1)}_{t-s})\right] \Big|_{x=\bar{X}^{(1)}_s(\bar{\omega})}.
\end{eqnarray*}
Thus, similarly as the proofs of Lemma \ref{Y homo} and Proposition \ref{X homo property}, we get the conclusion.
  \end{proof}

\begin{proof}[{\bf Proof of Proposition \ref{inv meas}}]
Let $t\ge0$ and $A\in\mathscr{B}(\reels^3)$. By taking $s=0$, $f(x) = 1_A(x)$ and integrating both sides of the equation of Lemma \ref{5E};
\begin{eqnarray*}
\int_{\bar{\Omega}\times\Omega}1_A\left(X^x_t(\omega)|_{x=\bar{X}_0(\bar{\omega})}\right)dP(\omega)d\bar{P}(\bar{\omega}) = \bar{P}\left[\bar{X}_t  \in A\right] = P^{\bar{X}}[A],
\end{eqnarray*}
where we used the stationarity of $\bar{X}$. The above left hand side equals
\begin{eqnarray*}
\int_{\bar{\Omega}\times\Omega}1_A\left(X^x_t(\omega)|_{x=\bar{X}_0(\bar{\omega})}\right)dP(\omega)d\bar{P}(\bar{\omega})
= \int_{\reels^3_+}\int_{\Omega}1_A\left(X^x_t(\omega)\right)dP(\omega)dP^{\bar{X}}(x)
= \int_{\reels^3_+}P^t(x,A)dP^{\bar{X}}(x), 
\end{eqnarray*}
and then we are done.
  \end{proof}


\subsubsection{ Ergodicity }

The $V$-geometric ergodicity has been proved for the process $X^{(1)}$, see Proposition 4.5 in \cite{ClinetYoshida}. For the Hawkes core process $X = (X^{(1)}, X^{(2)}, X^{(3)})$, we can also prove it in a similar way. That is, we apply Theorem 6.1 in \cite{MeynTweedie3}. First, we again consider the extended generator and the drift criterion. The following lemma is proved by the same method as Proof of Proposition 4.5. in \cite{ClinetYoshida}. 

\begin{lemma*}
\label{X generator}
Let $\alpha, \beta$ and $\mu$ be the parameters of the Hawkes process $N^{x_1}_t$. For a differentiable function $f : \reels^3 \to \reels$, we define the operator $\mathscr{A}_X$ by
\begin{eqnarray*}
\mathscr{A}_Xf(y) = (\mu + y_1) \left\{ f \left( y + \left(
    \begin{array}{c}
      \alpha\\
      0\\
      0
    \end{array}
  \right)
 \right) - f ( y )\right\} 
 + \big( \partial_yf(y) \big)' \left\{ - \beta y + \left(
    \begin{array}{c}
      0\\
      y_1\\
      2y_2
    \end{array}
  \right)
 \right\}, \ y = \left(
    \begin{array}{c}
      y_1\\
      y_2\\
      y_3
    \end{array}
  \right) \in \reels^3.
\end{eqnarray*} Then, there exist a positive constant vector $M = (M_1, M_2, M_3)$ and  positive constants $K_1, K_2$ such that for $V(y) = e^{My}$, 
\begin{eqnarray*}
\mathscr{A}_XV(y) \le -K_1V(y) + K_2.
\end{eqnarray*}
\end{lemma*}

Then, we can prove Proposition \ref{X strong moment} with the help of this operator $\mathscr{A}_X$.
\begin{proof}[{\bf Proof of Proposition \ref{X strong moment}}]
Now, it is proved in the completely same way as the proof of Proposition \ref{strong moment} replaced $g(x, t) = e^{M_1x}e^{K_1t}$ and $\bar{\mathscr{A}}$ by $g_X(x, t) = e^{Mx}e^{K_1t}$ and $\bar{\mathscr{A}}_X$ satisfying $\bar{\mathscr{A}}_Xg_X(x, t) = e^{K_1t}(\mathscr{A}_Xe^{Mx} + K_1e^{Mx})$ respectively.
  \end{proof}

Second, we need to show that every compact set is petite for some skeleton chain, i.e. there exists $\delta>0$ such that for any compact set $C \in \mathscr{B}(\reels^3)$, we can choose a probability measure $a$ on $\relatifs_+$ and a non-trivial measure $\phi_a$ on $\reels^3$ such that
\begin{eqnarray*}
\sum_{n \in \relatifs_+} P^{\delta n}(x, A) a[n] \ge \phi_a[A] \quad \text{for all $x \in C$ and $A \in \mathscr{B}(\reels^3)$}.
\end{eqnarray*}

The following concepts are closely related to petite sets. We call $\{X^x_{\delta n}\}_{n \in \relatifs_+}$ is an irreducible, if there exists a finite measure $\phi$ on $\mathscr{B}(\reels^3)$ such that if $\phi[A] > 0$ then
\begin{eqnarray*}
\sum_{n = 1}^{\infty} P^{\delta n}(x, A) > 0 \quad \text{for any $x \in \reels^3_+$}.
\end{eqnarray*}

Moreover, we call $\{X^x_{\delta n}\}_{n \in \relatifs_+}$ is a $T$-chain, if there exist $k \in \relatifs_+$ and non-trivial kernel $T$ such that 
\begin{itemize}
\item $T(x, \reels^3) > 0$ for any $x \in \reels^3_+$,\vspace{-2mm}\\
\item $x \mapsto T(x, A)$ is lower semi-continuous for any $A \in \mathscr{B}(\reels^3)$,\vspace{-2mm}\\ 
\item $P^{\delta k}(x, A) \ge T(x, A)$ for any $x \in \reels^3_+$ and $A \in \mathscr{B}(\reels^3)$.
\end{itemize}

We consider a relation between the existence of petite compact sets and T-chain properties. The following lemma is well known, see Theorem 3.2 in \cite{MeynTweedie1}.

\begin{lemma*}
Suppose that $\{X^x_{\delta n}\}_{n \in \relatifs_+}$ is an irreducible $T$-chain. Then, every compact set is petite.
\end{lemma*}

Furthermore, we call $x^* \in \reels^3$ is reachable, if for any open set $G \in \mathscr{B}(\reels^3)$ with $x^* \in G$,
\begin{eqnarray*}
\sum_{n=0}^{\infty}P^{\delta n}(y, G) > 0 \quad \text{for any $y \in \reels^3_+$}.
\end{eqnarray*}

\begin{proof}[{\bf Proof of Proposition \ref{geometric ergodicity}}]
We only have to prove that there exists $\delta > 0$ such that $\{X^x_{\delta n}\}_{n \in \relatifs_+}$ is an irreducible $T$-chain. First, we check the $T$-chain property.

Denote the $i$-th jump time of $N_t^x$ by $\tau^x_i$. Let $\Delta\tau_i^x$ be the interval time between the $(i-1)$-th  and $i$-th jump of $N^x_t$, i.e. $\Delta\tau_i^x = \tau^x_i - \tau^x_{i-1}$. As mentioned in Lemma A.4 of  \cite{ClinetYoshida}, $\Delta\tau_i^x$ has the conditional probability density (with respect to Lebesgue measure) 
\begin{eqnarray*}
f^{\Delta\tau_i^x} \big( t  \big|  X^x_{\tau^x_{i-1}} = y \big) = \left( \mu + y_1e^{-\beta t} \right) \exp\left( \int_0^t \mu + y_1e^{-\beta s} ds \right), 
\end{eqnarray*}
where $y = (y_1, y_2, y_3)' \in \reels^3_+$. Moreover, it is known that
\begin{eqnarray*}
f^{(\Delta\tau_1^x, \dots, \Delta\tau_i^x)} ( t_1, \dots, t_i  | y)
&=& f^{\Delta\tau_i^x} \big( t_i  | X^x_{\tau^x_{i-1}} = X( t_1, \dots, t_{i-1} |  y) \big) \\
&& \times  f^{\Delta\tau_{i-1}^x} \big( t_{i-1}  | X^x_{\tau^x_{i-2}} = X( t_1, \dots, t_{i-2} |  y) \big) \\
&& \times \cdots \times f^{\Delta\tau_1^x} \big( t_1  | X^x_{\tau^x_1} = y \big),
\end{eqnarray*}
where denote $ \sum_{k=i}^j t_k$ by $T_{(i, j)}$ and 
\begin{eqnarray*}
X( t_1, \dots, t_j |  y)
 &=&  \left(
    \begin{array}{c}
       X^{(1)} ( t_1, \dots, t_j |  y)\\
       X^{(2)} ( t_1, \dots, t_j |  y)\\
       X^{(3)} ( t_1, \dots, t_j |  y)
    \end{array}
  \right)\\
&=&  \left(
    \begin{array}{c}
      y_1e^{-\beta T_{(1, j)}} + \sum_{l=1}^j \alpha e^{-\beta T_{(l+1, j)}}  \\
      \left( y_1T_{(1, j)} + y_2 \right) e^{-\beta T_{(1, j)}} + \sum_{l=1}^j \alpha T_{(l+1, j)} e^{ -\beta T_{(l+1, j)} }  \\
      \left( y_1T_{(1, j)}^2 + 2y_2T_{(1, j)} + y_3 \right) e^{-\beta T_{(1, j)}} + \sum_{l=1}^j \alpha T_{(l+1, j)}^2 e^{ -\beta T_{(l+1, j)} }
    \end{array}
  \right).
\end{eqnarray*}
Note that $f^{(\Delta\tau_1^x, \dots, \Delta\tau_i^x)} ( t_1, \dots, t_i  | y)$ is obviously smooth in $y$. Then, for any $\delta > 0$ and $A \in \mathscr{B}(\reels^3)$,
\begin{eqnarray*}
P^{\delta}(x, A)
&=& P\left[ X^x_{\delta}  \in A \right]\\
&\ge& P\left[ X^x_{\delta}  \in A , \ \sharp \{j | \tau^x_j < \delta \} = 3 \right]\\
&=& \int_{\reels_+^4} 1_{\{ \check{X}(\delta; t_1, t_2, t_3 |  x) \in A \}} 1_{\{ T_{(1,3)} < \delta \} \cap \{ T_{(1,4)} \ge \delta \}} 
f^{(\Delta\tau_1^x, \Delta\tau_2^x, \Delta\tau_3^x, \Delta\tau_4^x)} ( t_1, t_2, t_3, t_4  | x) dt_1dt_2dt_3dt_4,
\end{eqnarray*}
where
\begin{eqnarray*}
\check{X}(\delta; t_1, t_2, t_3 |  x) 
=  \left(
    \begin{array}{c}
      x_1e^{-\beta \delta} + \sum_{l=1}^3 \alpha e^{-\beta (\delta - T_{(1, l)})}  \\
      \left( x_1\delta + x_2 \right) e^{-\beta \delta} + \sum_{l=1}^3 \alpha (\delta - T_{(1, l)}) e^{ -\beta (\delta - T_{(1, l)}) }  \\
      \left( x_1\delta^2 + 2x_2\delta + x_3 \right) e^{-\beta \delta} + \sum_{l=1}^3 \alpha (\delta - T_{(1, l)})^2 e^{ -\beta (\delta - T_{(1, l)}) }
    \end{array}
  \right)
\end{eqnarray*}
and it is obviously smooth in $x$. However, the indicator function $1_{\{ \check{X}(\delta; t_1, t_2, t_3 |  x) \in A \}}$ is not always lower semi-continuous in $x$. Thus, we consider a change of variable for the map $H_{x,\delta} : (t_1, t_2, t_3) \mapsto \check{X}(\delta; t_1, t_2, t_3 | x)$, as in Proof of Lemma A.3 of \cite{ClinetYoshida}. 
Denote the Jacobian matrix of $H_{x,\delta}$ at $(t_1, t_2, t_3)$ by $J_{\delta}(t_1, t_2, t_3) $. Then, completely elementary calculations leads
\begin{eqnarray*}
J_{\delta}(t_1, t_2, t_3)  = (J_{i,j})_{i,j = 1,2,3},
\end{eqnarray*}
where for $j=1,2,3$
\begin{eqnarray*}
J_{1,j} = \sum_{l=j}^3 \alpha\beta e^{-\beta (\delta - T_{(1,l)})}, \ 
J_{2,j} = \sum_{l=j}^3 \alpha\left\{ \beta(\delta - T_{(1,l)})  - 1 \right\}e^{-\beta (\delta - T_{(1,l)})}, \\
\text{and} \quad J_{3,j} = \sum_{l=j}^3  \alpha\left\{ \beta(\delta - T_{(1,l)})^2  -  2 (\delta - T_{(1,l)})  \right\} e^{-\beta (\delta - T_{(1,l)})}.
\end{eqnarray*}
The determinant of the Jacobian matrix has the following representation.
\begin{eqnarray*}
|J_{\delta}(t_1, t_2, t_3)| 
&=& \prod_{l=1}^3 \alpha \beta e^{-\beta (\delta - T_{(1,l)})} \times \left|
    \begin{array}{ccc}
     1 & 1& 1\\
    \big( \delta - T_{(1,1)} \big)& \big( \delta - T_{(1,2)} \big) & \big( \delta - T_{(1,3)} \big) \\
    \big( \delta - T_{(1,1)} \big)^2 & \big( \delta - T_{(1,2)} \big)^2 & \big( \delta - T_{(1,3)} \big)^2 \\
    \end{array}
  \right|.
\end{eqnarray*}
It is a Vandermonde determinant and thus not zero if $(t_1, t_2, t_3) = (\tau, \tau, \tau)$ for $\tau \in (0,\delta/3)$. We consider a neighborhood at the such point $(t_1, t_2, t_3) = (\tau, \tau, \tau)$. Set $B(t_1, t_2, t_3, t_4) = \{ T_{(1,3)} < \delta \} \cap \{ T_{(1,4)} > \delta \} \cap \{(t_1, t_2 ,t_3) \in (\tau-\ep,\tau+\ep)^3\}$ for sufficient small $\ep>0$.  Then, we get a non-trivial component $T(x, A)$ as below.
\begin{eqnarray*}
P^{\delta}(x, A)
&\ge& \int_{\reels_+^4} 1_{\{ \check{X}(\delta; t_1, t_2, t_3 |  x) \in A \}} 1_{\{ T_{(1,3)} < \delta \} \cap \{ T_{(1,4)} \ge \delta \}} 
f^{(\Delta\tau_1^x, \Delta\tau_2^x, \Delta\tau_3^x, \Delta\tau_4^x)} ( t_1, t_2, t_3, t_4  | x) dt_1dt_2dt_3dt_4\\
&\ge&  \int_{\reels_+^4} 1_{\{ (y_1, y_2, y_3) \in A \}} 1_{B(H^{-1}_{x, \delta}(y_1, y_2, y_3), t_4)}  
f^{(\Delta\tau_1^x, \Delta\tau_2^x, \Delta\tau_3^x, \Delta\tau_4^x)} ( H^{-1}_{x, \delta}(y_1, y_2, y_3), t_4 | x) \\
&&\left|J_{\delta} \left( H^{-1}_{x, \delta}(y_1, y_2, y_3) \right)\right|^{-1} dy_1dy_2dy_3dt_4
=: T(x, A). 
\end{eqnarray*}
Since $B(t_1, t_2, t_3, t_4)$ is a countable union of open intervals, continuity of $H^{-1}_{x,\delta}$ in $x$ leads that $x \mapsto T(x,A)$ is lower semi-continuous. Thus, $\{X^x_{\delta n}\}_{n \in \relatifs_+}$ is a $T$-chain. 

Finally, we prove that $\{X^x_{\delta n}\}_{n \in \relatifs_+}$ is irreducible. Since $\{X^x_{\delta n}\}_{n \in \relatifs_+}$ is a $T$-chain, we only have show that there exists a reachable point $x^* \in \reels^3_+$, i.e. for any open set $O \in \mathscr{B}(\reels^3)$ containing $x^*$,
\begin{eqnarray*}
\sum_{n=0}^{\infty} P^{\delta n}(y, O) > 0 \quad \text{for any $y \in \reels^3_+$,}
\end{eqnarray*}
see Proposition 6.2.1 in \cite{MeynTweedieBook}. However, we can easily show that $(0,0,0)$ is a reachable point. Indeed, if a jump will never occur, for any neighborhood $O$ of $(0,0,0)$, $X^x_{\delta n} \in O$ for sufficient large $n \in \naturels$. By the form of $f^{\Delta \tau_1^x} \big( t  \big| x\big)$, the probability there is no jump on $[0, \delta n]$ is positive. Thus, we get the conclusion.
  \end{proof}


\subsection{ Proofs of Section 4}
In this subsection, we will prove Theorem \ref{Main Thm4}. For this purpose, it is enough to confirm that there exist some constants satisfying (\ref{restriction}) and the conditons [A1]-[A3], [B0]-[B4], [C1] hold. We explain each condition separately by dividing each small section.

\subsubsection{Proof of Proposition \ref{A1} (Condition \textnormal{[A1]})}
In Markovian framework, as mentioned in \cite{KusuokaYoshida} and \cite{Yoshidaparmix}, the mixing property is derived from the ergodicity. Concretely, the geometric mixing property is reduced to the following property;

\begin{description}
\im[[A1$'$\!\!]] There exists a positive constant $a$ such that
\begin{eqnarray*}
\sup_{\substack{f \in \mathscr{FB}_{[t, \infty)} \\ : \left\| f \right\|_{\infty}  \le 1 }} \left\| E\left[f \left| X_s\right.\right] - E[f] \right\|_{L^1(P)} < a^{-1}e^{-a(t-s)} \quad \text{for any $t > s > 0$}.
\end{eqnarray*}
\end{description}

\begin{proposition*}
The Markovian property in Proposition \ref{X Markovian property} and \textnormal{[A1$'$]} lead \textnormal{[A1]}.
\end{proposition*}

\begin{proof}
For any $f \in \mathscr{FB}_{[0,s]}$ and $g \in \mathscr{FB}_{[t,\infty)}$ with $\left\| f \right\|_{\infty} \le 1$ and $\left\| g \right\|_{\infty} \le 1$,
\begin{eqnarray*}
\left| E[fg] - E[f]E[g] \right| 
&=& \left|E\left[f(g-E[g])\right]\right|
= \left|E\left[fE\left[g-E[g]\left|\mathscr{B}_{[0,s]}\right.\right]\right]\right|\\
&\le& \left\|E\left[g-E[g]\left|\mathscr{B}_{[0,s]}\right.\right]\right\|_{L^1(P)}
= \left\|E\left[g\left|X_s\right.\right] - E[g]\right\|_{L^1(P)}
\le a^{-1}e^{-a(t-s)}.
\end{eqnarray*}
 \end{proof}

\begin{proof}[{\bf Proof of Proposition \ref{A1} }]

We confirm that [A1$'$] follows from Proposition \ref{geometric ergodicity}. Let $s \le t$ and $f \in \mathscr{FB}_{[t,\infty)}$ with $\|f\|_{\infty} \le 1$. From the Markovian property, we have
\begin{eqnarray*}
E\left[f \left| X_s \right.\right]
= E\left[E\left[f \left|\mathscr{B}_{[0,t]}\right.\right]\left| X_s \right.\right]
= E\left[E\left[f \left|X_t \right.\right]\left| X_s \right.\right].
\end{eqnarray*}
There exists a measurable function $g$ such that $E\left[f \left|X_t \right.\right] = g(X_t)$ and $\|g\|_{\infty} \le 1$. From Proposition \ref{X homo property}, we get
\begin{eqnarray*}
E\left[f \left| X_s \right.\right]
= E\left[ g(X_t)\left| X_s \right.\right] = \int_{\reels^3} g(y) P^{t-s}(X_s, dy).
\end{eqnarray*}
On the other hand, we have
\begin{eqnarray*}
E\left[f\right]
= E\left[g(X_t)\right]
= \int_{\reels^3} g(y) P^t(X_0, dy).
\end{eqnarray*}
Therefore, by using Proposition \ref{geometric ergodicity},
\begin{eqnarray*}
&&\sup_{\substack{f \in \mathscr{FB}_{[t, \infty)} \\ : \left\| f \right\|_{\infty}  \le 1 }} \left\| E\left[f \left| X_s\right.\right] - E[f] \right\|_{L^1(P)} \\
&\le& \sup_{g : \| g \|_{\infty} \le 1} \left\|\int_{\reels^3} g(y) P^{t-s}(X_s, dy) - \int_{\reels_+} g(y) P^{t}(X_0, dy)\right\|_{L^1(P)} \\
&\le&  \sup_{g : \| g \|_{\infty} \le 1} \left\{ \left\|\int_{\reels^3} g(y) \left(P^{t-s}(X_s, dy) - P^{\bar{X}}(dy) \right)\right\|_{L^1(P)} + \left\| \int_{\reels^3} g(y) \left(P^{t}(X_0, dy) - P^{\bar{X}}(dy) \right) \right\|_{L^1(P)} \right\}\\
&\le&  E\left[\left\|P^{t-s}(X_s, \cdot) - P^{\bar{X}} \right\|_{e^{M\cdot}}\right] + E\left[\left\| P^{t}(X_0, \cdot) - P^{\bar{X}}\right\|_{e^{M\cdot}}\right]\\
&\le& E\left[ B(e^{M X_s } + 1)r^{t-s} \right] + E\left[ B(e^{M X_0} + 1)r^{t}\right]\\
&=& r^{t-s}B\left(  E\left[ e^{M X_s} \right] + 1 + 2r^s\right).
\end{eqnarray*}
Finally, from Proposition \ref{X strong moment}, we may choose sufficient small $a>0$ that satisfies [A1$'$].
  \end{proof}


\subsubsection{Condition \textnormal{[A2]}}
$Z_0 \in \bigcap_{p>1}L^p(P)$ and $P[Z_0]=0$ are obvious.  We can write each component of $Z^t_{t+h}$ as
\begin{eqnarray*}
\int_t^{t+h} \frac{p_1(X_s)}{\lambda_s^2} d\tilde{N}_s + \int_t^{t+h} \frac{p_2(X_s)}{\lambda_s^2} ds - E\left[ \int_t^{t+h} \frac{p_2(X_s)}{\lambda_s^2} ds \right]
\end{eqnarray*}
where $p_1$ and $p_2$ are $3$-variable polynomial functions. From Proposition \ref{X strong moment}, we have $\sup_t ||X_t||_{L^p(P)} < \infty$ for any $p>1$. By considering $t \in [0, T]$ for an arbitrary $T>0$, $\int_0^t p_1(X_s)/\lambda_s^2 d\tilde{N}_s$ is a square integrable martingale, see Theorem 18.8 in \cite{LS}. Thus, we immediately get $E\big[Z^t_{t+\Delta}\big] = 0$ for any $\Delta > 0$ and $t  > 0$.

The rest of the proof is $\sup_{t \in \reels_+, 0 \le h \le \Delta} \left\| Z^t_{t+h}\right\|_{L^p(P)} < \infty$.
When we consider the $L^p$ boundedness, it is enough to consider the form of $p = 2^k$ for $k \in \naturels$.
We get
\begin{eqnarray*}
\left\|E\left[ \int_t^{t+h} \frac{p_2(X_s)}{\lambda_s^2}  ds \right]\right\|_{L^p(P)}
< \frac{\sup_s E[|p_2(X_s)|] }{\mu_0^2} h.
\end{eqnarray*}
Moreover, since $h^{-1} ds$ is a probability measure on $[t, t+h]$, by the Jensen's inequality,
\begin{eqnarray*}
\left\| \int_t^{t+h} \frac{p_2(X_s)}{\lambda_s^2} ds \right\|_{L^p(P)}
&\le& \left(E\left[ \int_t^{t+h} \left(  \frac{p_2(X_s)}{\lambda_s^2}h \right)^p h^{-1}ds \right]\right)^{\frac{1}{p}}\\
&\le& \frac{\sup_s \|p_2(X_s)\|_{L^p(P)}}{\mu_0^2} h.
\end{eqnarray*}

On the other hand,
\begin{eqnarray*}
\mathscr{M}_h = \int_t^{t+h} \frac{p_1(X_s)}{\lambda_s^2} d\tilde{N}_s
\end{eqnarray*}
is also a square integrable martingale. Then, the Burkholder-Davis-Gundy inequality leads that there exists a positive constant $C_k$ (take again new $C_k$ in the last step) such that
\begin{eqnarray*}
E\left[\left| \mathscr{M}_h \right|^{2^k}\right]
&\le& C_k E\left[ \left|[\mathscr{M} ]_h \right|^{2^{k-1}}\right] \nonumber \\
&=& C_k E\left[ \left| \int^{t+h}_t \left( \frac{p_1(X_s)}{\lambda_s^2}  \right)^2 dN_s \right|^{2^{k-1}}\right] \nonumber \\
&\le& C_k \left( E\left[ \left| \int^{t+h}_t \left( \frac{p_1(X_s)}{\lambda_s^2} \right)^2 d\tilde{N}_s \right|^{2^{k-1}}\right] + E\left[ \left| \int^{t+h}_t \left( \frac{p_1(X_s)}{\lambda_s^2} \right)^2 \lambda_sds \right|^{2^{k-1}}\right]\right),
\end{eqnarray*}
where $[\mathscr{M}]_h$ represents the quadratic variation of $\mathscr{M}_h$. We used the Jensen's inequality in the last estimation. By induction, one gets some constant $Q_k$ (take again new $Q_k$ in the last step) such that
\begin{eqnarray}
\label{B.H.}
E\left[ \left| \mathscr{M}_h  \right|^{2^k}\right]
&\le& Q_k\sum_{j = 1}^k E\left[ \left| \int^{t+h}_t \left( \frac{p_1(X_s)}{\lambda_s^2} \right)^{2^j} \lambda_sds \right|^{2^{k-j}}\right]\nonumber \\
&\le& Q_k\sum_{j = 1}^k E\left[ \int^{t+h}_t \left( \frac{p_1(X_s)}{\lambda_s^2} \right)^{2^k} \lambda_s^{2^{k-j}} h^{2^{k-j}} h^{-1}ds \right]\nonumber \\
&\le& Q_k\ (h + 1)^{2^{k-1}}.
\end{eqnarray}
Therefore, for any $\Delta > 0$ and $p>0$, $\sup_{t \in \reels_+, 0 \le h \le \Delta} \left\| Z^t_{t+h}\right\|_{L^p(P)} < \infty$ holds. Then, the conditon [A2] is verified.


\subsubsection{Condition \textnormal{[A3]}}
The conditon [A3] follows from Lemma 3.15. and the proof of Lemma A.7. in \cite{ClinetYoshida}.

\subsubsection{Condition \textnormal{[B0]}}
(i), (ii) and (iv) are obvious. (iii) immediately follows a square integrable martingale property:
\begin{eqnarray*}
Cov\left[ \int_0^t \frac{p_1(X_s)}{\lambda_s} d\tilde{N}_s , \int_0^t \frac{p_2(X_s)}{\lambda_s} d\tilde{N}_s \right]
= E\left[ \int_0^t \frac{p_1(X_s)p_2(X_s)}{\lambda_s^2} d[\tilde{N}]_s \right]
= E\left[ \int_0^t \frac{p_1(X_s)p_2(X_s)}{\lambda_s} ds \right]
\end{eqnarray*}
for any $3$-variable polynomial functions $p_1$ and $p_2$.

\subsubsection{Condition \textnormal{[B1]}}
We take any constant $L>1$. From (\ref{1st deriv}) and (\ref{B.H.}), we immediately deduce that for any $k \in \naturels$
\begin{eqnarray*}
E \left[ \left| T^{-\frac{1}{2}} l_a(\theta_0)\right|^{2^k} \right]
\ \le \  Q_k\ (1 + T^{-1})^{2^{k-1}}.
\end{eqnarray*}
Therefore, the conditon [B1] holds for any $q_1 > 1$. In particular, we can choose $q_1$ satisfying $q_1 >  3L$.

\subsubsection{Condition \textnormal{[B2]}}
Let $L, q_1$ and $q_3$ be positive constants with $L>1, q_1 >  3L$ and $q_3 > \frac{q_1L}{q_1-3L}$. We arbitrary set a positive constant $q_2$ with $q_2 > \max\left(3, \frac{3q_1L}{q_1-3L}\right)$ for given constants $L$ and $q_1$.
Let $Y_t(\theta) = \left(X_t^{(1)}(\theta_0), X_t^{(1)}(\theta), X_t^{(2)}(\theta), X_t^{(3)}(\theta), X_t^{(4)}(\theta) \right)$ for $\theta = (\mu, \alpha, \beta)$. From the relation 
\begin{eqnarray*}
\partial_{\theta} X_t^n(\theta) = \left( \begin{array}{ccc} 0 \\ \alpha^{-1}X_t^n(\theta) \\ -X_t^{n+1}(\theta) \end{array}\right)
\end{eqnarray*}
and a verification of the permutation rule of the symbol $\partial_{\theta}$ and $\int_0^T$, we can write, for both of the case $k = 2$ and $k = 3$,
\begin{eqnarray*}
&&T^{\frac{\gamma}{2}}\left(T^{-1}l_{a_1 \cdots a_k}(\theta) - \nu_{a_1 \cdots a_k}(\theta) \right) \\
&=& T^{\frac{\gamma}{2} -1} \int_0^T \frac{p_1(Y_s(\theta))}{\lambda_s^4(\theta)} d\tilde{N}_s + T^{\frac{\gamma}{2}} \left\{ \frac{1}{T}\int_0^T \frac{p_2(Y_s(\theta))}{\lambda_s^4(\theta)} ds -  E\left[\frac{1}{T} \int_0^T \frac{p_2(Y_s(\theta))}{\lambda_s^4(\theta)} ds \right] \right\}
\end{eqnarray*}
with some polynomial functions $p_1$ and $p_2$. Lemma A.5. in \cite{ClinetYoshida} guarantees 
\begin{eqnarray}
\label{intensity moment}
\sup_t \sum_{i=0}^4 \| \sup_{\theta \in \Theta} \partial_{\theta}^i \lambda_t(\theta) \|_{L^p(P)} < \infty
\end{eqnarray}
for any $p>1$. Thus, we have $\sup_t ||\sup_{\theta \in \Theta} Y_t(\theta)||_{L^p(P)} < \infty$ for any $p>1$. Moreover, $Y_t(\theta)$ is $\sigma\left(N_s ; s \le t \right)$-predictable. From the restriction of (\ref{restriction}), $\frac{\gamma}{2} -1 < -\frac{1}{2}$ holds. Then, in the same method of the proof of the conditon [A2], we can see that
\begin{eqnarray*}
\sup_{T>0, \theta \in \Theta} E\left[ \left| T^{\frac{\gamma}{2} -1} \int_0^T \frac{p_1(Y_s(\theta))}{\lambda_s^4(\theta)} d\tilde{N}_s\right|^{2^k} \right]
< \infty
\end{eqnarray*}
for any $k \in \naturels$. 

The later term is estimated by using the ergodicity of $X_t^{(1)}(\theta_0)$. Let 
\begin{eqnarray*}
\tilde{Y}(s, t, \theta)
&=& \Bigg(X_t^{(1)}(\theta_0), 
\int_{(s,t)} \alpha e^{-\beta(t-u)}dN_u^{x_1},  
\int_{(s,t)} \alpha(t-u) e^{-\beta(t-u)}dN_u^{x_1}, \\
&& \int_{(s,t)} \alpha(t-u)^2 e^{-\beta(t-u)}dN_u^{x_1}, 
\int_{(s,t)} \alpha(t-u)^3 e^{-\beta(t-u)}dN_u^{x_1} \Bigg).
\end{eqnarray*}
Denote $D_{\uparrow}\big(\reels_+^5, \reels\big)$ as the set of functions $\psi: \reels_+^5 \to \reels$ that satisfy:
\begin{itemize}
\item $\psi$ are of class $C^1(\reels_+^5)$. 
\item $\psi$ and $|\bigtriangledown \psi|$ are polynomial growth. 
\end{itemize}
By replacing $X^{\alpha}(t, \theta)$ by $Y_t(\theta)$ and $\tilde{X}^{\alpha}(s, t, \theta)$ by $\tilde{Y}(s, t, \theta)$
in the proof of Lemma A.6. and using Lemma 3.16. in \cite{ClinetYoshida}, we can get the following ergodicity property: There exist a mapping $\pi:  D_{\uparrow}\big(\reels_+^5, \reels\big) \times \Theta \to \reels$ and a constant $\gamma' \in \left( 0, \frac{1}{2} \right)$ such that for any $\psi \in D_{\uparrow}\big(\reels_+^5, \reels\big)$ and for any $p>1$,
\begin{eqnarray*}
\sup_{\theta \in \Theta} T^{\gamma'} \left\| \frac{1}{T} \int_0^T \psi(Y_s(\theta)) ds - \pi(\psi, \theta)\right\|_{L^p(P)} \to 0 \quad \text{as $T \to \infty$}.
\end{eqnarray*}
However, in the case of the exponential Hawkes process, we can choose $\gamma' \in \left( 0, \frac{1}{2} \right)$ arbitrarily. This arbitrariness follows from the fact $\|Y_t(\theta) - \bar{Y}_t(\theta) \|_{L_1(P)}$ is exponentially decreasing uniformly in $\theta$ for some stationary process $\bar{Y}_t(\theta)$, see the proof of the stability
condition part in Lemma A.6. of \cite{ClinetYoshida}. Therefore, by taking $\gamma' \in \left( 0, \frac{1}{2} \right)$ and $\gamma = 2\gamma'$ satisfying $\frac{2}{3} + \max\left(\frac{L}{q_2}, \frac{L}{3q_3}\right) < \gamma < 1 - \frac{L}{q_1}$, we get 
\begin{eqnarray*}
\left\| T^{\frac{\gamma}{2}} \left\{ \frac{1}{T}\int_0^T \frac{p_2(Y_s(\theta))}{\lambda_s^4(\theta)} ds -  E\left[\frac{1}{T} \int_0^T \frac{p_2(Y_s(\theta))}{\lambda_s^4(\theta)} ds \right] \right\} \right\|_{L^p(P)} \to 0 \quad \text{as $T \to \infty$}
\end{eqnarray*}
for any $p>1$. It means that the conditon [B2] holds for any $q_2 > \max\left(3, \frac{3q_1L}{q_1-3L}\right)$ and some $\gamma$ with $\frac{2}{3} + \max\left(\frac{L}{q_2}, \frac{L}{3q_3}\right) < \gamma < 1 - \frac{L}{q_1}$.

\subsubsection{Condition \textnormal{[B3]}}
We only have to show that there exist an open set $\tilde{\Theta}$ including $\theta_0$ and a positive constant $T_0$ such that
\begin{eqnarray}
\label{[B3]}
\inf_{T>T_0, \theta \in \tilde{\Theta},|x|=1} \left| x' \nu_{a b} (\theta) \right| > 0.
\end{eqnarray}
Because, if (\ref{[B3]}) holds, continuity of $\nu_{a b}(\theta)$ and $x' \nu_{a b}(\theta) \neq 0$ lead
\begin{eqnarray*}
\left| \int_0^1 x' \nu_{a b} (\theta_1 + s(\theta_2 - \theta_1))ds \right| > \inf_{\theta \in \tilde{\Theta}} \left| x' \nu_{a b} (\theta) \right|
\end{eqnarray*}
for any $\theta_1, \theta_2 \in \tilde{\Theta}, T>T_0$ and $x$ with $|x|=1$.
Therefore, we consider to prove (\ref{[B3]}). We can write
\begin{eqnarray*}
\nu_{a b} (\theta)
= - E\left[ \frac{1}{T} \int^T_0  \frac{\left(\partial_{\theta}\lambda_s (\theta) \right)^{\otimes 2}}{\lambda_s^2(\theta)} \lambda_s(\theta_0) ds \right] + E\left[ \frac{1}{T} \int^T_0  \frac{\partial_{\theta}^2 \lambda_s(\theta)}{\lambda_s(\theta)} \left( \lambda_s(\theta_0) - \lambda_s(\theta) \right)ds \right].
\end{eqnarray*}
With the help of (\ref{intensity moment}), for the first term, we have
\begin{eqnarray*}
\left| g_T - E\left[ \frac{1}{T} \int^T_0  \frac{\left(\partial_{\theta}\lambda_s (\theta) \right)^{\otimes 2}}{\lambda_s^2(\theta)} \lambda_s(\theta_0) ds \right] \right|
\le  \frac{|\theta_0 - \theta| }{T} E\left[ \int^T_0  \sup_{\theta \in \Theta} \left| \partial_{\theta} \frac{\left(\partial_{\theta}\lambda_s (\theta) \right)^{\otimes 2}}{\lambda_s^2(\theta)} \right| \lambda_s(\theta_0) ds \right]
\le C_{\Theta, 1} |\theta_0 - \theta|,
\end{eqnarray*}
where $C_{\Theta, 1}$ is a positive constant that does not depend on $T$. For the second term, we also get
\begin{eqnarray*}
\left| E\left[ \frac{1}{T} \int^T_0  \frac{\partial_{\theta}^2 \lambda_s(\theta)}{\lambda_s(\theta)} \left( \lambda_s(\theta_0) - \lambda_s(\theta) \right)ds \right] \right|
\le  \frac{|\theta_0 - \theta| }{T} E\left[ \int^T_0  \sup_{\theta \in \Theta} \left| \frac{\partial_{\theta}^2 \lambda_s(\theta)}{\lambda_s^2(\theta)} \right| \sup_{\theta \in \Theta} \left| \partial_{\theta}\lambda_s (\theta) \right|  ds \right]
\le C_{\Theta, 2} |\theta_0 - \theta|,
\end{eqnarray*}
where $C_{\Theta, 2}$ is a positive constant independent of $T$. Since we have assumed $g_T$ is non-singular for large $T$ in the conditon [A3], we may choose $\tilde{\Theta}$ and $T_0>0$ such that
\begin{eqnarray*}
\inf_{T>T_0, \theta \in \tilde{\Theta},|x|=1} \left| x' \nu_{a b} (\theta) \right|
&\ge& \inf_{T>T_0,|x|=1} \left| x' g_T \right| - \sup_{T>T_0, \theta \in \tilde{\Theta}}\left| g_T - E\left[ \frac{1}{T} \int^T_0  \frac{\left(\partial_{\theta}\lambda_s (\theta) \right)^{\otimes 2}}{\lambda_s^2(\theta)} \lambda_s(\theta_0) ds \right] \right|  \\
&&- \sup_{T>T_0, \theta \in \tilde{\Theta}} \left| E\left[ \frac{1}{T} \int^T_0  \frac{\partial_{\theta}^2 \lambda_s(\theta)}{\lambda_s(\theta)} \left( \lambda_s(\theta_0) - \lambda_s(\theta) \right)ds \right] \right| > 0.
\end{eqnarray*}

\subsubsection{Condition \textnormal{[B4]}}
Let $L$ and $q_1$  be positive constants with $L>1$ and $q_1 >  3L$. We apply Sobolev's inequality (see Theorem 4.12 of \cite{AdamsFournier}). We take any integer $q_3 > \max\left(3, \frac{q_1L}{q_1-3L}\right)$ and some constant $K(\Theta, q_3)$ such that
\begin{eqnarray*}
E\left[ \sup_{\theta \in \Theta} \left| \frac{1}{T} l_{a_1 \cdots a_4}(\theta) \right|^{q_3} \right]
&\le& K(\Theta, q_3)\left\{ \int_{\Theta} E\left[ \left| \frac{1}{T} l_{a_1 \cdots a_4}(\theta) \right|^{q_3} \right] d\theta + \int_{\Theta}  E\left[ \left| \frac{1}{T} \partial_{\theta} l_{a_1 \cdots a_4}(\theta) \right|^p \right] d\theta \right\}\\
&\lesssim& \sum_{k=1}^4 \sup_{\theta \in \Theta} E\left[ \left| \frac{1}{T} l_{a_1 \cdots a_4}(\theta) \right|^{q_3} \right].
\end{eqnarray*}
Let $Y'_t(\theta) = \left(X_t^{(1)}(\theta_0), X_t^{(1)}(\theta), X_t^{(2)}(\theta), X_t^{(3)}(\theta), X_t^{(4)}(\theta), X_t^{(5)}(\theta) \right)$. We may easily confirm that
\begin{eqnarray*}
\frac{1}{T}l_{a_1, \dots, a_4}(\theta)
= \frac{1}{T} \int_0^T \frac{p_1(Y'_s(\theta))}{\lambda_s^{8}(\theta)} d\tilde{N}_s + \frac{1}{T}  \int_0^T \frac{p_2(Y'_s(\theta))}{\lambda_s^{8}(\theta)} ds
\end{eqnarray*}
with some polynomial functions $p_1$ and $p_2$.  In a similar way as Lemma A.5 in \cite{ClinetYoshida}, we can prove that $\sup_t \| \sup_{\theta \in \Theta}Y'_t(\theta) \|_{L^p(P)} < \infty$ for any $p>1$. Then, like (\ref{B.H.}), we have
\begin{eqnarray*}
\sup_{T>0, \theta \in \Theta} E\left[ \left| \frac{1}{T} \int_0^T \frac{p_1(Y'_s(\theta))}{\lambda_s^{8}(\theta)} d\tilde{N}_s\right|^{2^k} \right]
< \infty
\end{eqnarray*}
for any $k \in \naturels$. On the other hand, by the Jensen's inequality,
\begin{eqnarray*}
\sup_{T>0, \theta \in \Theta} E\left[ \left| \frac{1}{T} \int_0^{T} \frac{p_2(Y'_s(\theta))}{\lambda_s^{8}(\theta)} ds \right|^p \right]
&\le& \sup_{T>0} \frac{1}{T} \int_0^{T} E\left[\sup_{\theta \in \Theta} \left| \frac{p_2(Y'_s(\theta))}{\lambda_s^{8}(\theta)} \right|^p \right] ds < \infty
\end{eqnarray*}
for any $p>1$. Therefore, the condition [B4] holds for any constant $q_3 > \max\left(3, \frac{q_1L}{q_1-3L}\right)$.

\subsubsection{Condition \textnormal{[C1]}}
In Theorem 4.6 of \cite{ClinetYoshida}, the convergence of moments is proved for $\sqrt{T}\big(\mle - \theta_0\big)$. The conditon [C1] directly follows from this statement. 


%
%

\section*{Acknowledgment}
I am deeply grateful to Professor Yoshida. Without his guidance and help, this article would not have been completed. This research was supported by FMSP program of The University of Tokyo and Japan Science and Technology Agency CREST JPMJCR14D7.

%
%

\bibliographystyle{plain}      
\bibliography{ref.bib}  

%
%

\end{document}